\documentclass[a4paper,11pt, reqno]{amsart}
\usepackage{etex}
\reserveinserts{28}
\usepackage[T1]{fontenc}
\usepackage{silence,lmodern}
\usepackage[utf8]{inputenc}
\usepackage[english]{babel}
\usepackage[stretch=10]{microtype}
\usepackage{mathtools}
\usepackage{amsthm}
\usepackage{amssymb}
\usepackage{mathrsfs}
\usepackage{url}
\usepackage{enumerate}
\usepackage{wrapfig}
\usepackage{bbm}
\usepackage{color,tikz,pgfplots}
\usetikzlibrary{shapes.misc}
\usepackage{pst-func}
\usepackage{amsaddr}
\usepackage{graphicx}
\usepackage{caption}
\usepackage{subcaption}
\usepackage{float}
\usepackage{csquotes}
\usepackage[left=3cm, top=3cm,bottom=3cm,right=3cm]{geometry}
\usepackage{enumitem}
\usepackage{hyperref}

\usepackage{scalerel,stackengine}
\stackMath
\newcommand\reallywidehat[1]{%
	\savestack{\tmpbox}{\stretchto{%
			\scaleto{%
				\scalerel*[\widthof{\ensuremath{#1}}]{\kern.1pt\mathchar"0362\kern.1pt}%
				{\rule{0ex}{\textheight}}
			}{\textheight}%
		}{2.4ex}}%
	\stackon[-6.9pt]{#1}{\tmpbox}%
}

\usepackage[
backend = biber,
style = alphabetic,
giveninits = true,
hyperref = true,
maxbibnames = 10,
natbib=true
]{biblatex}
\AtBeginBibliography{\normalsize}

\newtheorem{theorem}{Theorem}

\newtheorem{lemma}{Lemma}
\newtheorem{proposition}{Proposition}
\newtheorem{corollary}{Corollary}
\newtheorem{conjecture}{Conjecture}
\newtheorem*{theorem*}{Theorem}
\newtheorem*{theoremA}{Theorem A}
\newtheorem*{theoremB}{Theorem B}

\theoremstyle{definition}
\newtheorem{definition}{Definition}

\newtheorem{remark}{Remark}
\newtheorem{exmp}{Example}
\newtheorem*{exmp*}{Example}

\newtheorem*{acknowledgements}{Acknowledgements}

\newtheorem{innercustomgeneric}{\customgenericname}
\providecommand{\customgenericname}{}
\newcommand{\newcustomtheorem}[2]{%
	\newenvironment{#1}[1]
	{%
		\renewcommand\customgenericname{#2}%
		\renewcommand\theinnercustomgeneric{##1}%
		\innercustomgeneric
	}
	{\endinnercustomgeneric}
}

\newcustomtheorem{customthm}{Theorem}
\newcustomtheorem{customlemma}{Lemma}

\numberwithin{equation}{section}

\usepackage{definitions}

\begin{document}

	\author[J.G.Gerstenberg]{Julian Gero Gerstenberg}
	\address{Institute for Mathematics, Goethe University Frankfurt am Main, Germany}
	\email{gerstenb@math.uni-frankfurt.de}


		\title[Exchangeable Laws in Borel Data Structures]{Exchangeable Laws in Borel Data Structures}
		
		\begin{abstract}
			Motivated by statistical practice, category theory terminology is used to introduce Borel data structures and study exchangeability in an abstract framework. A generalization of de Finetti's theorem is shown and natural transformations are used to present functional representation theorems (FRTs). Proofs of the latter are based on a classical result by D.N.Hoover providing a functional representation for exchangeable arrays indexed by finite tuples of integers, together with an universality result for Borel data structures. A special class of Borel data structures are array-type data structures, which are introduced using the novel concept of an indexing system. Studying natural transformations mapping into arrays gives explicit versions of FRTs, which in examples coincide with well-known Aldous-Hoover-Kallenberg-type FRTs for (jointly) exchangeable arrays. The abstract "index arithmetic" presented unifies and generalizes technical arguments commonly encountered in the literature on exchangeability theory. Finally, the category theory approach is used to outline how an abstract notion of seperate exchangeability can be derived, again motivated from statistical practice.
		\end{abstract}
		\maketitle
		
		\vspace{-0.5cm}
		\section{Introduction}\label{sec:INTRO}
		
		Let $\cS$ be a Borel space\footnote[1]{a measurable space $\cS$ is a Borel space if there exists a measurable subset $B\subseteq[0,1]$ and a bi-measurable bijection $f:\cS\rightarrow B$, see Appendix~\ref{appendix:borel_spaces} for basic properties of such spaces.}, $\bS_{\bN}$ the discrete group of bijections $\pi:\bN\rightarrow\bN$ and
		\begin{equation}\label{eq:groupaction}
			\bS_{\bN}\times\cS\rightarrow\cS, (\pi, x)\mapsto \pi x
		\end{equation}
		a measurable group action. (The law of) A $\cS$-valued random variable $X$ is called \emph{exchangeable} if $\pi X\ed X$ for every $\pi$, with $\ed$ being equality in distribution. In many examples motivated from statistics, $X$ is exchangeable iff $\pi X\ed X$ holds for all $\pi\in\bSi\subseteq\bS_{\bN}$, with $\bSi$ the countable group of bijections $\pi$ with $\pi(i)=i$ for all but finitely many $i$.\\
		
		This work studies exchangeability when $\bS_{\bN}\times\cS\rightarrow\cS$ is derived from a \emph{Borel data structure} (BDS), which is defined to be a functor
		\begin{equation*}
			D:\INJo\rightarrow\BOREL,
		\end{equation*}
		where $\INJ$ is the category of injections between finite sets, $\INJo$ its opposite and $\BOREL$ the category of measurable maps between Borel spaces. The main definitions and results are presented in Section~\ref{sec:results}, which starts with an explicit definition of Borel data structure in Definition~\ref{def:boreldatastructure}. No knowledge of category theory is assumed to read this paper, references for the used terminology are \cite{mac2013categories} and \cite{milewski2019category}, the latter providing a "programmers" view to category theory which fits the philosophy of how it is used in this work very well.\\
		This paper is addressed to readers interested in exchangeability and data structures, the emphasize is on decomposition, functional representation and foundations of statistical applications. Many surveys on exchangeability theory covering such topics exist, see \cite{aldous1985exchangeability}, \cite{austin2008exchangeable}, \cite{aldous2009more}, \cite{aldous2010exchangeability} or \cite{orbanz2014bayesian}.
		
		\begin{acknowledgements}
			Funded by the Deutsche Forschungsgemeinschaft (DFG, German Research Foundation), \emph{Exchangeability theory of ID-based data structures with applications in statistics}, 502386356.
		\end{acknowledgements}
	
		\tableofcontents
		
		\subsection{Overview of the results}
		
		The main achievement of this work may be the provided abstract framework, which allows to talk about many reoccurring phenomena and constructions in exchangeability theory literature in a general setting. The main results are:
		
		\begin{itemize}
			\item Theorem~\ref{thm:definetti}: a generalization of de~Finetti's theorem, which shows that exchangeable laws in BDS coincide precisely with mixtures of exchangeable laws satisfying an independence property, also see Theorem~\ref{thm:independence}, 
			\item Theorem~\ref{thm:equivalence}: Hoover's FRT for exchangeable arrays, Theorem~A below, has an equivalent formulation in the provided framework using the concept of natural transformations. Theorem~A is the most important \emph{ingredient} to our approach,
			\item Theorems~\ref{thm:frt-weak} and \ref{thm:frt-weak-depth}: a \emph{weak} FRT for exchangeable laws in arbitrary BDS using the concept of \emph{almost sure} natural transformations,
			\item Definitions~\ref{def:indexing-system}~and~\ref{def:arrays} providing the concepts of indexing system and array-type data structures,
			\item Theorems~\ref{thm:nat-in-array} and \ref{thm:nat-array} providing an explicit characterization of all (true) natural transformations mapping from any BDS into an array-type data structure via kernel functions. One application of this is in characterizing all \emph{local modification rules} on array-type data structures, Example~\ref{exmp:local_modification_rule}, a concept which has been introduced in \cite{austin2010testability}, 		
			\item Theorem~\ref{thm:frt-strong}: a \emph{strong} FRT for exchangeable laws in array-type data structures via true natural transformations. For a given array-type data structure and using the classification of natural transformations via kernels, it is seen that the derived FRT is often equivalent to some classical version of an Aldous-Hoover-Kallenberg-type FRT for (jointly) exchangeable arrays, see Corollary~\ref{cor:frt-array},
			\item Theorem~\ref{thm:universality}: it exists a Borel data structure that is \emph{universal} with respect to natural embedding,
			\item Theorem~\ref{thm:correspondence-limits}: by considering \emph{combinatorial} Borel data structures, a correspondence principle between exchangeable laws and \emph{limits of combinatorial structures} is shown. This generalizes many well-known of such correspondences, the most famous being between graph limits and exchangeable random graphs, see \cite{diaconis2007graph}, \cite{grubel2015persisting} or \cite{austin2008exchangeable} for a more general exposition. Another is between exchangeable posets and poset limits, see \cite{janson2011poset}, in which further examples are listed in the introduction. A very elementary instance of this correspondence can be formulated for exchangeable $\{0,1\}$-sequences, see \cite{gerstenberg2016boundary}.
			\item Section~\ref{sec:sep} in which a notion of \emph{seperate} exchangeability is presented for a wide range of Borel data structures. A special case is the classical notion of seperate exchangeability in arrays. The abstract construction of seperate exchangeability is motivated from its statistical philosophy and makes heavy use of the category theory approach to exchangeability via functors.
		\end{itemize}  
	
		Experts in the field may jump to read Section~\ref{sec:notations} (notations), followed by Section~\ref{sec:results} (definitions and main results), and come back to read the rest of the introduction later; at this point further motivations and connections to existing literature are presented.
		
		\subsection{Similar use of category theory terminology in related work}
		
		The categorical approach to exchangeability via Borel data structures can be motivated from a statistical perspective, see Section~\ref{sec:statisticalmotivation}, which is, in spirit, very close to the use of category theory in \cite{mccullagh2002statistical} where the more general question "What is a statistical model?" is discussed, see Remark~\ref{rem:whatis}.\\
		There is close connection to the notion of \emph{combinatorial species}, see \cite{bergeron1998combinatorial}, used in analytical combinatorics; (combinatorial) Borel data structures can be interpreted as combinatorial species equipped with a \emph{restriction mechanism} compatible with the relabeling mechanism; this approach was used in \cite{gerstenberg2018austauschbarkeit}. Like the case with combinatorial species, a great benefit of using category theory terminology with Borel data structures is that it becomes easy to \emph{build new examples of Borel data structures by composition}, which provides infinite examples by iterative constructions, see Example~\ref{exmp:comp}. Also, the category theory approach is the basis for introducing an abstract concept of seperate exchangeability in Section~\ref{sec:sep}.\\
		Several definitions in this work are close to the content presented in Section~3.1 of \cite{austin2010testability}, where contravariant functors, natural transformations and also exchangeable laws were introduced in a similar abstract setting, some aspects of that work were presented already in \cite{austin2008exchangeable} in an "explicit" form. More connections are explained throughout the work, also see Remark~\ref{rem:subcantor} discussing the different basic assumptions.
		
		\begin{remark}[Other connections]
			The approach to exchangeability via functors modeling data structures is complemented by the approach using model theory, we refer to Section~3.8 in \cite{austin2008exchangeable} and the references therein. Also, de~Finetti's theorem for exchangeable sequences has been approached from a more pure category theory perspective recently, see \cite{fritz2021journal}, \cite{jacobs2020finetti} or \cite{staton2022quantum}. To explain all these connections goes beyond the scope of this paper.
		\end{remark}
		 
		\subsection{Exchangeability in arrays} FRTs are often presented for different notions of exchangeability in arrays, many of which fit in the framework (\ref{eq:groupaction}) as follows: given is a Borel space $\cX$, a countable set of \emph{indices} $I_{\bN}$ and a group action 
		\begin{equation}\label{eq:indices}
			\bS_{\bN}\times I_{\bN}\rightarrow I_{\bN}, (\pi, \bbi)\mapsto \pi\bbi
		\end{equation} 
		on indices. This gives a (left-)group action on $\cS=\cX^{I_{\bN}}$ by defining for $x = (x_{\bbi})_{\bbi\in I_{\bN}} \in \cS$ the action as $\pi x = (x_{\pi^{-1}\bbi})_{\bbi\in I_{\bN}}$. In this situation, $\cS$-valued exchangeable random variables are \emph{arrays} of $\cX$-valued random variables indexed by $I_{\bN}$, that is $X = (X_{\bbi})_{\bbi\in I_{\bN}}$, such that  
		\begin{equation}\label{eq:array}
			X = (X_{\bbi})_{\bbi\in I_{\bN}}~\ed~(X_{\pi\bbi})_{\bbi\in I_{\bN}} = \pi X~~~\text{for all}~\pi\in\bS_{\bN}.
		\end{equation}
		Many basic examples of exchangeability in arrays are instances of (\ref{eq:array}), some examples together with their FRTs are presented next. Let $U_a, a\in\finN$ be iid $\unif[0,1]$-random variables indexed by finite subsets of $\bN$. The following results are organized in Chapter~7 of \cite{kallenberg2006probabilistic}.
		
		\begin{itemize}
			\item[(E1)] Sequences: $I_{\bN} = \bN$, that is $\bbi=i\in\bN$, and $\pi \bbi = \pi(i)$. The de~Finetti/Hewitt-Savage theorem states that laws of exchangeable sequences $X=(X_{i})_{i\in\bN}$ are precisely the mixtures of laws of iid processes, which directly translates into a FRT: for every exchangeable sequence $X$ there exists a measurable function $f:[0,1]^2\rightarrow \cX$ such that 
			$$X\ed (f(U_{\emptyset},U_{\{i\}}))_{\bbi=i\in\bN},$$
			with $U_{\emptyset}$ being responsible for the mixture over iid laws.
			\item[(E2)] Arrays indexed by size-2 sets: $I_{\bN} = \binom{\bN}{2}$, that is $\bbi = \{i_1,i_2\}\subset\bN, i_1\neq i_2$, and $\pi \bbi = \{\pi(i_1),\pi(i_2)\}$. A FRT (Aldous, Hoover) reads as follows: for every exchangeable $X = (X_{\bbi})_{\bbi\in \binom{\bN}{2}}$ there is a measurable function $f:[0,1]^4\rightarrow \cX$, symmetric in the second and third argument, such that 
			$$X \ed \Big(f\big(U_{\emptyset}, U_{\{i_1\}}, U_{\{i_2\}}, U_{\{i_1,i_2\}}\big)\Big)_{\bbi=\{i_1,i_2\}\in\binom{\bN}{2}}.$$
			An elementary proof of this is presented in \cite{austin2012exchangeable}. Note that, by symmetry of $f$, the value $f(U_{\emptyset},U_{\{i_1\}}, U_{\{i_2\}}, U_{\{i_1,i_2\}})$ does not depend on an enumeration of the set $\bbi=\{i_1,i_2\}$. In case $\cX=\{0,1\}$ exchangeable arrays indexed by $\binom{\bN}{2}$ correspond to \emph{exchangeable random graphs on nodes $\bN$}, the variables $X_{\{i_1,i_2\}}\in\{0,1\}$ indicating edges. 
			\item[(E3)] Arrays indexed by length-2 tuples with different entries: $I_{\bN}=\bN^2_{\neq}$, that is $\bbi = (i_1,i_2)\in\bN^2$ with $i_1\neq i_2$, and $\pi \bbi = (\pi(i_1),\pi(i_2))$. For every exchangeable $X=(X_{\bbi})_{\bbi\in\bN^2_{\neq}}$ there exists a measurable function $f:[0,1]^4\rightarrow\cX$, \emph{that does not have to be symmetric}, such that
			$$X \ed \Big(f\big(U_{\emptyset}, U_{\{i_1\}}, U_{\{i_2\}}, U_{\{i_1,i_2\}}\big)\Big)_{\bbi=(i_1,i_2)\in\bN^2_{\neq}}.$$
			In case $\cX=\{0,1\}$ arrays indexed by $\bN^2_{\neq}$ are exchangeable \emph{directed} graphs on nodes $\bN$ (without self-loops), $X_{(i_1,i_2)}\in\{0,1\}$ indicates the presence of a directed edge $i_1\rightarrow i_2$. 
		\end{itemize}
	
		The examples (E2) and (E3) have straightforward generalizations to indices being $k$-size subsets $I_{\bN} = \binom{\bN}{k}$ or $k$-length tuples with different entries $I_{\bN} = \bN^k_{\neq}$; of course one could also consider $I_{\bN} = \binom{\bN}{\leq k}, \bN^{\leq k}_{\neq}$ or $\bN^k$. In all these cases FRTs use randomization up to size $k$, that is involve variables $U_{a}, a\in\binom{\bN}{\leq k}$. FRTs for indices of unbounded size such as $I_{\bN}=\binom{\bN}{<\infty}$ (all finite subsets) or $\bN^*_{\neq}$ (all finite length tuples with different entries), use full randomization $U_a, a\in\finN$. The FRT in the latter case is due to Hoover, see Theorem~A in Section~\ref{sec:ingredients}.\\
		
		The Definitions~\ref{def:indexing-system} and~\ref{def:arrays} introduce indexing systems and the derived notion of array-type data structure, the latter being special types of BDS. This provides an abstract framework to capture examples of the previous types and Theorem~\ref{thm:frt-strong} gives a unified formulation of FRTs in such cases, which is later translated into an explicit low-level form in Corollary~\ref{cor:frt-array}.
		
		\begin{remark}[Graphons and Digraphons]
			Representations for exchangeable graphs, (E2) with $\cX=\{0,1\}$, are often presented using \emph{graphons}, which are symmetric measurable functions $W:[0,1]^2\rightarrow [0,1]$. Given a graphon one can define an exchangeable random graph as follows: given $U_{\{i\}}, i\in\bN$ let $X_{\{i_1,i_2\}}, \{i_1,i_2\}\in\binom{\bN}{2}$ be independent with $X_{\{i_1,i_2\}}\sim \Ber(W(U_{i_1}, U_{i_2}))$ (Bernoulli). The FRT in (E2) shows that, loosely speaking, every exchangeable random graph appears in this way if one allows the graphon to be picked at random in a first step experiment: define $W_u(x,y) = \bP[f(u,x,y,U)=1]$ with $U\sim\unif[0,1]$ and set $W = W_{U_{\emptyset}}$ (ignoring measureability details).\\
			Representations for exchangeable directed graphs, (E3) with $\cX=\{0,1\}$, are often presented using \emph{digraphons}; how the FRT in (E3) translates into a digraphon representation is explained in \cite{diaconis2007graph}, Proof of Theorem~9.1.\\
			Applications of such derived representations are, for example, in the context of Bayesian statistics, see \cite{cai2016priors} or \cite{orbanz2014bayesian}. 
		\end{remark}

		\begin{remark}[Other notions of exchangeability in arrays]
			This work mainly studies exchangeability in the sense of $\bSi$-invariance, motivated by the statistical philosophy in Section~\ref{sec:statisticalmotivation}. In the context of arrays, $\cS=\cX^{I_{\bN}}$, the term "exchangeability" is often also used for a probabilistic symmetry induced by a group action $G\times I_{\bN}\rightarrow I_{\bN}$ on indices in which $G$ is not necessarily $\bS_{\bN}$. Examples are separately exchangeable arrays, for instance $I_{\bN} = \bN^2$ and $G = \bS_{\bN}\times\bS_{\bN}$ acting on $\bN^2$ as $(\pi_1,\pi_2)(i_1,i_2) = (\pi_1(i_1), \pi_2(i_2))$. Considering only the diagonal action every separately exchangeable array is seen to be also (jointly) exchangeable in the sense of $\bS_{\bN}$-invariance; the converse fails in general. The statistical philosophy of "basic" notions of seperate exchangeability can be exploited to derive a notion of seperate exchangeability in the abstract setting of BDS discussed in this work, this is outlined in Section~\ref{sec:sep}. Functional representations for classical notions of seperatly exchangeable arrays are presented in Chapter~7 of \cite{kallenberg2006probabilistic}.\\
			Other types of actions on indices, giving generalizations of classical notions of exchangeability, are studied in \cite{austin2014hierarchical} (hierarchical exchangeability), \cite{jung2021generalization} and \cite{alea2022finetti} (DAG-exchangeability). Also see \cite{lloyd2013exchangeable} in which "exchangeability in databases" is discussed.
		\end{remark}
	
		\subsection{Other exchangeable random objects} 
		
		Exchangeability has long been studied in random structures different from, but not unrelated to, arrays. Many of such examples can be discussed within the BDS framework, to mention only a few:\\
		Relation-type examples are given by partitions (interpreted as equivalence relations, Kingman's Paintbox, see \cite{kallenberg2006probabilistic}, Section~7.8), posets \cite{janson2011poset}, strict weak orders \cite{gnedin1997representation} or total orders (folklore, see e.g. \cite{gerstenberg2020general} Section 3.2). Examples of this type fit into the frameworks (\ref{eq:groupaction}) and (\ref{eq:indices}) as follows: given is a action on indices $\bS_{\bN}\times I_{\bN}\rightarrow I_\bN$ and the space $\cS$ of interest (partitions, orders,$\dots$) can be encoded as a subspace $\cS\subseteq \cX^{I_{\bN}}$ such that $\pi\in\bS_{\bN}, x\in\cS\Rightarrow \pi x\in \cS$ and such that the notion of exchangeability on $\cS$ is inherited from the array-setting. The exchangeable random structures can thus be seen as exchangeable arrays for which FRTs are often known -- but mostly lead to unsatisfactory functional representations, as the additional structure given by $\cS$ is ignored. However, this approach can serve as an intermediate step to a satisfactory representation, for an example see \cite{evans2017doob} (exchangeable didendritic systems). The essence of these examples -- structures of interest being "embedded" in more general ones -- is later introduced within the BDS framework by considering \emph{natural embeddings} and \emph{sub-data structures}. In (hyper)graphs sub-data structures correspond to so-called \emph{hereditary (hyper)graph properties}, see the introduction in \cite{austin2010testability}.\\
		Another source of examples for exchangeability in random structures does not (directly) fit the array-framework: structures of \emph{set system-type}. Examples are total partitions (hierarchies) \cite{forman2018representation} or interval hypergraphs \cite{gerstenberg2020exchangeable}. It is not (directly) obvious how these structures could be encoded as an exchangeable array in a useful way. Later the (combinatorial) BDS of set systems is introduced and these examples are seen to be sub-data structures therein.
	
		\subsection{Statistical motivation}\label{sec:statisticalmotivation}
		
		Studying exchangeability in context (\ref{eq:groupaction}) can be motivated by thinking about how data is collected by a statistician: \emph{picking a small group of individuals from a large population and measuring information on that group}, the type of information could very well be about interactions between individuals, that is relational. For storing the measured information as data (on a piece of paper, on a computer,...) it is required to give unique \emph{identifiers (IDs)} to the individuals of the picked group, which are used to represent the individuals within stored data - a common choice of IDs for storing information of a finite group of $n$ individuals is $[n] = \{1,2,\dots,n\}$, at least in mathematical papers. When studying exchangeability theory it is assumed that the finite groups can be of arbitrary finite size - which pays to the idea that the underlying population is 'large'. Based on the idea of \emph{sampling consistency} one passes to model measurements on countable infinite group of individuals, usually identifying individuals using IDs $\bN = \{1,2,\dots\}$. Having this in mind, a group action $\bS_{\bN}\times\cS\rightarrow\cS$ can be interpreted as follows: $x\in\cS$ represent data measured on a countable infinite group of individuals represented via IDs $i\in\bN$ and $\pi x\in\cS$ represents the measurement \emph{on the same group}, but with IDs of individuals redistributed according to $i\mapsto \pi(i)$. Now suppose randomness is involved: first, a population is picked \emph{at random} and second, conditioned on the population being picked, the statistician "randomly" picks a countable infinite group of individuals and gives them IDs $i\in\bN$, also "randomly". Given that group of individuals represented by IDs $i\in\bN$, the statistician measures data on that group, which gives $X\in\cS$. The precise meaning of "randomly" is not specified (for good reasons), but it seems reasonable to model the final measurement a $\cS$-valued exchangeable random variable, that is $X\ed\pi X$ for all $\pi\in\bS_{\bN}$.\\
		Two thoughts about this:
		\begin{itemize}
			\item[(T1)] all a statistician will ever see in practice are measurements on \emph{finite} groups of individuals; Borel data structures model the treatment of finite measurements only and countable infinite measurements, which are of theoretical interest, are \emph{constructed} using sampling consistency,
			\item[(T2)] there is no reason to restrict IDs $i$ being elements $i\in\bN$, that is natural numbers -- IDs only serve to identify individuals within stored data, no information of interest should be encoded in IDs. Using category theory terminology provides a suitable language to handle arbitrary sets (of IDs). 
		\end{itemize}
		In search for a mathematical framework replacing $\bS_{\bN}\times\cS\rightarrow\cS$ by something that fits both the statisticians philosophy and also pays to (T1) and (T2) directly leads to the Definition of a Borel data structure and a notion of exchangeability therein, Section~\ref{sec:results}.
		
		\begin{remark}\label{rem:whatis}
			The philosophy behind IDs and exchangeability are closely related to the ideas presented in \cite{mccullagh2002statistical}, where the way more general question of what constitutes a statistical model is discussed. In that approach, the concept of an ID is replaced by \emph{statistical unit}, which can encode more structure but just to serve as an identifier. 
		\end{remark}

		\subsection{The main ingredients of the proofs}\label{sec:ingredients}
		
		The notion of exchangeability studied in Borel data structures turns out to be equivalent to $\bSi$-invariance. $\bSi$ is a countable amenable group, thus ergodic theory provides important theorems: relevant for this work are ergodic decomposition (Theorem~A1.4 in \cite{kallenberg2006probabilistic}) and \emph{pointwise convergence} (Theorem 1.2 in \cite{lindenstrauss2001pointwise}). Interesting for statistical applications: the convergence in the pointwise convergence theorem is known to be asymptotically normal under mild regularity assumptions, see \cite{austern2018limit}. An application of this is, for example, in the analysis of cross validations protocols, see Section~4.5 of \cite{austern2019limit}. Also, an application to "generalized $U$-statistics" is given later, see Remark~\ref{rem:asymptotic-u-statistics}.\\
		
		The most important ingredient to the proofs of FRTs in this work is a functional representation of exchangeable arrays fitting the framework (\ref{eq:array}) as follows: let $\cX$ be a Borel space and $I_{\bN} = \bN^*_{\neq}$ be the set of all finite-length tuples $\bbi = (i_1,\dots,i_k)$ with $k\geq 0, i_j\in\bN$ and $i_{j}\neq i_{j'}$ for all $j\neq j'$. The group $\bS_{\bN}$ acts on $I_{\bN}=\bN^*_{\neq}$ as $\pi \bbi = (\pi(i_1),\dots,\pi(i_k))$. The following theorem follows the exposition of Theorem~7.21 in \cite{kallenberg2006probabilistic} where the result is attributed to D.N.~Hoover~\cite{hoover1979relations}. 
		
		\begin{theoremA}[FRT for exchangeable arrays indexed by $\bN^*_{\neq}$, Hoover, Kallenberg]
			For every $\cX$-valued exchangeable array $X = (X_{\bbi})_{\bbi\in\bN^*_{\neq}}$ there exists a measurable function 
			$$f:\bigcup_{k\geq 0}[0,1]^{2^{[k]}}\rightarrow\cX,$$
			such that 
			$$X~\ed~\Big(f\big((U_{\pi_{\bbi}(e)})_{e\in 2^{[k]}}\big)\Big)_{\bbi=(i_1,\dots,i_k)\in\bN^*_{\neq}},$$
			where for $\bbi = (i_1,\dots,i_k)\in\bN^*_{\neq}$ it is $\pi_{\bbi}:\{1,\dots,k\}\rightarrow\{i_1,\dots,i_k\}, j\mapsto i_j$. 
		\end{theoremA}
	
	
		\subsection{Notations}\label{sec:notations}
		Let $M$ be a set, $|M|$ its cardinality and $2^M$ its power set. For $k\geq 0$ define subsets of $2^M$ 
		\begin{align*}
			\binom{M}{k} &= \{M'\in 2^M~:~|M'|=k\},\\
			\binom{M}{\leq k} &= \{M'\in 2^M~:~|M'|\leq k\},\\
			\binom{M}{<\infty} &= \{M'\in 2^M~:~|M'|<\infty\}.
		\end{align*}
		Let $M^* = \cup_{k\geq 0}M^k$ be the set of all finite-length tuples $(m_1,\dots,m_k), k\geq 0$ over $M$. Let $M^k_{\neq}$ be the set of all length-$k$ tuples $(m_1,\dots,m_k)\in M^k$ with $m_j\neq m_{j'}$ for $j\neq j'$. Let $M^*_{\neq} = \cup_{k\geq 0}M^k_{\neq} \subset M^*$ be the set of all finite-length tuples over $M$ with different entries.\\
		Let $N, M$ be two non-empty sets and $N^M$ the set of functions $f:M\rightarrow N$. Note that $N^{\emptyset}$ is also defined, even if $N$ is empty: there exists exactly one function $f:\emptyset\rightarrow N$, which is always injective and bijective iff $N=\emptyset$. In particular, $|N^{\emptyset}|=|\emptyset^{\emptyset}| = 1$.\\
		For any function $f:M\rightarrow N$ define functions:
		\begin{itemize}
			\item $\im(f):2^M\rightarrow 2^N$ sends $M'\subseteq M$ to the image $f(M')\subseteq N$,
			\item $\vec{f}:M^*\rightarrow N^*$ sends $(m_1,\dots,m_k)\in M^k$ to $(f(m_1),\dots,f(m_k))\in N^k$,
			\item $\hat f:M\rightarrow f(M), m\mapsto f(m)$, that is $\hat f$ is obtained from $f$ by restricting its range to its image. $\hat f$ is surjective.
		\end{itemize}
		For every $M'\subseteq M$ let
		$$\iota_{M',M}:M'\rightarrow M, m\mapsto m$$
		be the inclusion map and  
		$$\id\nolimits_M:M\rightarrow M, m\mapsto m$$
		for the identity on $M$. It is $\iota_{M',M}$ always injective and it is bijective iff $\iota_{M',M} = \id_M$, that is $M'=M$.\\
		Every function $f:M\rightarrow N$ has the representation
		$$f = \iota_{f(M),N}\circ \hat f,$$
		that is as a composition of a surjective map followed by an inclusion map. $f$ is injective iff $\hat f$ is bijective. $f$ is surjective iff $\iota_{f(M),N}$ is bijective, which implies $f(M)=N$ and $\hat f = f$.\\
	
		For a measurable space $\cX$ it is $\sP(\cX)$ the set of probability measures on $\cX$. The law of $\cX$-valued random variable $X$ is $\cL(X) = \bP[X\in\cdot]\in\sP(\cX)$. For random variables $\ed$ denotes equality in distribution and $\as$ almost sure equality. For a set $M$ it is $\cX^M$ a measurable space equipped with the product $\sigma$-field. For $M=\emptyset$ it is $\cX^{\emptyset}$ the discrete measurable space consisting of one point being the function $x:\emptyset\rightarrow\cX$, similar $\emptyset^{\emptyset}$ has the single element $x:\emptyset\rightarrow\emptyset$.
		
		\section{Main definitions and results}\label{sec:results}
		
		Arbitrary finite sets are denoted by $a, b, c$. They represent \emph{finite sets of IDs used by a statistician to identify individuals from a finite group}. An injection $\tau:b\rightarrow a$ corresponds to \emph{picking a subgroup from a group represented by IDs $a$ using IDs $b$}. In the subgroup obtained via $\tau$ individuals are assigned IDs $b$, individual $i'\in b$ corresponds to  $\tau(i')\in a$ in the larger group. Each injection $\tau:b\rightarrow a$ can be written as 
		$$\tau = \iota_{\tau(b),a}\circ\hat\tau,$$
		with 
		\begin{itemize}
			\item $\iota_{\tau(b),a}:\tau(b)\rightarrow a, i\mapsto i$ the inclusion map of $\tau(b)\subseteq a$,
			\item $\hat\tau:b\rightarrow\tau(b), i\mapsto \tau(i)$ the bijection obtained by restricting the range. 
		\end{itemize}
		Injection $\iota_{\tau(b),a}$ corresponds to restricting group $a$ to subgroup $\tau(b)\subseteq a$ and $\hat \tau:b\rightarrow\tau(b)$ to a redistribution of IDs on subgroup $\tau(b)$ via $\tau(i)\in\tau(b)\mapsto i\in b$.\\
		
		The following is an explicit definition of a \emph{contra}variant functor $\INJ\rightarrow\BOREL$, which is the same as a (covariant) functor $\INJo\rightarrow\BOREL$. 
		
		\begin{definition}[Borel data structure]\label{def:boreldatastructure}
			A Borel data structure (BDS) is a rule $D$ that maps 
			\begin{itemize}
				\item every finite set $a$ to a Borel space $\Da$,
				\item every injection $\tau:b\rightarrow a$ between finite sets to a measurable map $\Dtau:\Da\rightarrow\Db$,
			\end{itemize}
			such that 
			\begin{enumerate}
				\item[(i)] $D[\id_a] = \id_{\Da}$ for every finite set $a$,
				\item[(ii)] $D[\sigma\circ\tau] = D[\tau]\circ D[\sigma]$ for all composable injections $\sigma, \tau$ between finite sets.
			\end{enumerate}
		\end{definition}
	
		In case every $D_a$ is a non-empty finite discrete space $D$ is called \emph{combinatorial data structure}. Combinatorial data structures coincide with functors $D:\INJo\rightarrow\FIN_+$, where $\FIN_+$ is the category of maps between non-empty finite sets.\\
		
		One can interpret $D_a$ as the space of data a statistician can collect on a group of $n=|a|$ individuals using IDs $a$ to represent individuals. For every injection $\tau:b\rightarrow a$ the contravariance of $D$ gives
		$$D[\tau] = D[\hat\tau]\circ D[\iota_{\tau(b),a}],$$
		one can think of 
		\begin{itemize}
			\item $D[\iota_{\tau(b),a}]: D_a\rightarrow D_{\tau(b)}$ as restricting measurements to subgroups,
			\item $D[\hat\tau]:D_{\tau(b)}\rightarrow D_b$ as transforming IDs within stored data as $\tau(i)\in\tau(b)\mapsto i\in b$,
		\end{itemize}
		thus $D[\tau]$ combines both these operations.\\
		
		Now image a statistician picks a finite group of $n$ individuals from a large population and gives them IDs $i\in a$ with $|a|=n$ \emph{both at random}, then measures $D_a$-valued data on that group, modeled as a $D_a$-valued random variable $X_a$. What "at random" means here is not specified, but is seems obvious that for every injection $\tau:b\rightarrow a$ it should hold that
		$$X_b \ed D[\tau](X_a).$$
		Let $\mu_a = \cL(X_a) \in\sP(D_a)$ the law of $X_a$. In terms of laws the previous is equivalent to 
		$$\mu_b = \mu_a\circ D[\tau]^{-1},$$
		which leads to the following definition:
	
		\begin{definition}[Exchangeable law]\label{def:exchangeable_law}
			An exchangeable law on $D$ is a rule $\mu$ that sends every finite set $a$ to a probability measure $\mu_a\in\sP(D_a)$ such that for every injection $\tau:b\rightarrow a$ it holds that $\mu_a \circ D[\tau]^{-1} = \mu_b$. Let $\SYM(D)$ be the class of all exchangeable laws on $D$. 
		\end{definition}
	
		\begin{remark}
			In $D_{\emptyset}$ the statistician records information that is not about any individual, hence that information is about the population itself or more general about the environment the measurement takes place in.
		\end{remark}
	
		\begin{exmp}\label{exmp:exchangeable_processes}
			Let $\cX$ be a Borel space and define $D = \SEQ(\cX)$ (sequential data over $\cX$) by $D_a = \cX^a$ and $D[\tau](x) = x\circ\tau$. Let $X = (X_i)_{i\in\bN}$ be a $\cX$-valued exchangeable sequence. By exchangeability, for every finite set $a$ and any two injections $\tilde\tau, \tau:a\rightarrow\bN$ it holds $X\circ\tau \ed X\circ \tilde\tau$, which allows to define
			$$\mu_a = \cL(X\circ\tau) \in \sP(D_a)$$
			not depending on the choice of $\tau$. It is easily seen that this defines an exchangeable law $\mu = [a\mapsto \mu_a]\in\SYM(\SEQ(\cX))$ and that the construction $\cL(X)\mapsto \mu$ is a one-to-one correspondence between laws of exchangeable $\cX$-valued sequences and $\SYM(\SEQ(\cX))$; the inverse construction involves Kolmogorov consistency arguments.
		\end{exmp}
	
		The discussion in Section~\ref{sec:random_variables} shows that the previous example generalizes to any Borel data structure $D$, that is: $\SYM(D)$ can be naturally identified with the space of invariant probability measures for some measurable group action $\bSi\times\cS\rightarrow\cS$ on a Borel space $\cS$. In particular, $\SYM(D)$ is a set that comes equipped with a natural Borel space (and convexity) structure such that for every finite set $a$ and measurable $M\subseteq D_a$ the evaluation map $\mu\in\SYM(D)\mapsto \mu_a(M)\in[0,1]$ is measurable. 
		
		\begin{remark}[Exchangeable laws via category theory terminology]\label{rem:exch_via_cat}
			See \cite{mac2013categories} for category theory vocabulary used here, in particular Section~4. There are at least two equivalent ways to obtain $\SYM(D)$ using category theory constructions. Both involve the endofunctor $\sP:\BOREL\rightarrow\BOREL$ which sends a Borel space $\cX$ to the Borel space $\sP(X)$ and a measurable map $f:\cX\rightarrow\cY$ to the push-forward $\sP[f]:\sP(\cX)\rightarrow\sP(\cY), \nu\mapsto \nu\circ f^{-1}$. For every BDS $D:\INJo\rightarrow\BOREL$ it is $\sP\circ D:\INJo\rightarrow\BOREL$ defined by $(\sP\circ D)_a = \sP(D_a)$ and $(\sP\circ D)[\tau] = \sP[D[\tau]]$ a new BDS. Having this, $\SYM(D)$ can be identified with either
			\begin{itemize}
				\item the \emph{limit} of the functor $\sP\circ D$: cones over $\sP\circ D$ correspond to measurable parametrizations $\Theta\rightarrow \SYM(D), \theta\mapsto \mu^{\theta}$ (not necessarily injective or surjective) with the parameter space $\Theta$ being Borel. The limit $\SYM(D)$ corresponds to the parametrization of $\SYM(D)$ by itself. An example of a cone over $\sP\circ\SEQ(\bR)$ is $\Theta = \bR\times(0,\infty)\mapsto \mu^{\theta}$ with $\mu^{\theta}_a = \Norm(\theta_1,\theta_2)^{\otimes a}$ (the iid-normal-distribution model).
				\item the set of all natural transformations $\eta:\ptt\rightarrow \sP\circ D$, where $\ptt:\INJo\rightarrow\BOREL$ is the trivial data structure $\ptt_a = \{1\}$ and $\ptt[\tau] = \id_{\{1\}}$, compare to Equation~(29) in \cite{austin2010testability}. 
			\end{itemize}
		\end{remark}
	
		\begin{remark}[Combinatorial species, see \cite{bergeron1998combinatorial}]
			A combinatorial species is a (covariant) functor $C:\BIJ_+\rightarrow\BIJ_+$, where $\BIJ_+$ is the category of bijections between non-empty finite sets. Every combinatorial data structure $D:\INJo\rightarrow\FIN_+$ defines a species of structures $C$ by $C_a = D_a$ and $C[\pi] = D[\pi^{-1}]$. In this sense, combinatorial data structures can be seen as combinatorial species enriched with restrictions compatible with the relabeling mechanism. The restriction mechanism is of crucial importance to develop exchangeability theory.
		\end{remark}
	
		\subsection{Generalization of de Finetti's theorem}
		
		Let $\mu\in\SYM(D)$. If $\mu$ corresponds to the law of data obtained by picking individuals from a \emph{fixed} large population, it seems obvious that the measurements on disjoint subgroups should be independent, that is: if $a, b$ are disjoint, $a\cap b=\emptyset$, and $X_{a+b}\sim \mu_{a+b}$ then $D[\iota_{a,a+b}](X_{a+b})$ and $D[\iota_{b,a+b}](X_{a+b})$ should be independent. The following defines this property on the level of laws. 
		
		\begin{definition}[Independence property]\label{def:independence}
			$\mu\in\SYM(D)$ has the \emph{independence property} if for all finite sets $a, b$ with $a\cap b=\emptyset$ 
			$$\mu_{a+b}\circ (D[\iota_{a,a+b}],D[\iota_{b,a+b}])^{-1} = \mu_a\otimes\mu_b.$$
			Let $\eSYM(D)\subseteq\SYM(D)$ be the subset of exchangeable laws having this property.
		\end{definition}
	
		It is seen later that the laws having the independence property coincide with ergodic invariant laws, thus the notion $\eSYM$. Exchangeable laws are precisely the mixtures of exchangeable laws having the independence property:
	
		\begin{theorem}\label{thm:definetti}
			If $\SYM(D)\neq\emptyset$, then $\eSYM(D)$ a non-empty measurable subset of $\SYM(D)$ and the following map is a bijection: 
			$$\sP(\eSYM(D))~\longrightarrow~\SYM(D),~~~\Xi~\longmapsto~E[\Xi],$$
			where $E[\Xi]$ is the rule $a \mapsto E[\Xi]_a\in\sP(D_a)$ defined by 
			$$E[\Xi]_a(\cdot) = \int_{\eSYM(D)}\mu_a(\cdot)d\Xi(\mu).$$
		\end{theorem}

		\begin{exmp}[Exchangeable sequences are mixed iid]
			Let $D=\SEQ(\cX), \mu\in\SYM(D)$. Let $a\cap b=\emptyset$ and $(X_i)_{i\in a+b} \sim \mu_{a+b}$. In terms of random variables it is
			$$\mu_{a+b}\circ (D[\iota_{a,a+b}],D[\iota_{b,a+b}])^{-1} = \cL\big((X_i)_{i\in a}, (X_i)_{i\in b}\big),$$
			a joint distribution of two \emph{disjoint} sub-collections of RVs. If $\mu$ has the independence property it thus holds that
			$$\cL\big((X_i)_{i\in a}, (X_i)_{i\in b}\big) = \mu_a\otimes\mu_b = \cL\big((X_i)_{i\in a}\big)\otimes\cL\big((X_i)_{i\in b}\big).$$
			Applying this inductively down to singletons and using exchangeability shows every $\mu\in\eSYM(D)$ is of the form $\mu_a = \nu^{\otimes a}$ for some $\nu\in\sP(\cX)$; one can identify $\sP(\cX)$ with $\eSYM(D)$ and Theorem~\ref{thm:definetti} gives: exchangeable laws in $\SEQ(\cX)$ are precisely given by the rules $a\mapsto \int \nu^{\otimes a}(\cdot)d\Xi(\nu)$, bijectivity parameterized through $\Xi\in\sP(\sP(\cX))$.
		\end{exmp}
	
		In case $D=\SEQ(\cX)$ it was easily possible to use the independence property to give a perfect parametrization of $\eSYM(D)$. From a data structure point of view the reason for this is that for every disjoint sets $a\cap b=\emptyset$ the map $\cX^{a+b}\rightarrow \cX^{a}\times\cX^{b}, x\mapsto (x_{|a}, x_{|b})$ is a bijection. This is a very special property of sequential data $D=\SEQ(\cX)$ and fails in general. As a consequence, it is in general far from obvious how exchangeable laws having the independence property look like -- functional representations offer a different approach to understand the structure of exchangeable laws. 
	
		\subsection{A weak FRT for arbitrary Borel data structures}
		
		Borel data structures have been introduced as functors and a good notion for "functions between functors" is that of a natural transformation. Also an \emph{almost sure} version is introduced:
		
		\begin{definition}[(Almost sure) Natural transformations]\label{def:natural_transformation}
			Let $D, E:\INJo\rightarrow\BOREL$ be two Borel data structures and $\eta:D\rightarrow E$ be a rule that sends every $a$ to a measurable map $\eta_a:D_a\rightarrow E_a$.\\
			The rule $\eta$ is called
			\begin{itemize}
				\item \emph{natural transformation} if for all $\tau:b\rightarrow a$ 
				$$\eta_b\circ D[\tau] = E[\tau]\circ\eta_a,$$
				\item \emph{$\mu$-a.s. natural transformation}, with $\mu\in\SYM(D)$, if for all $\tau:b\rightarrow a$ 
				$$\eta_b\circ D[\tau] = E[\tau]\circ\eta_a~~~\text{$\mu_a$-almost surely}.$$
			\end{itemize}
		\end{definition}
	
		Of course, a natural transformation $\eta:D\rightarrow E$ is also a $\mu$-a.s. natural transformation for every $\mu\in\SYM(D)$. 
		
		\begin{exmp}
			Every measurable $f:\cX\rightarrow\cY$ gives a natural transformation $\eta^f:\SEQ(\cX)\rightarrow\SEQ(\cY)$ having components $\eta^f_a(x) = f\circ x$ and this is a one-to-one correspondence between measurable maps and natural transformations. This is generalized by Theorem~\ref{thm:nat-array} later.
		\end{exmp}
		
		A central observation is the following:
		
		\begin{proposition}\label{prop:push_forward}
			For every $\mu\in\SYM(D)$ and $\mu$-a.s. natural transformation $\eta:D\rightarrow E$ it is $\mu\circ\eta^{-1} \in \SYM(E)$, where $\mu\circ\eta^{-1}$ is the rule that sends $a$ to the push-forward of $\mu_a$ under $\eta_a$, that is to the probability measure $\mu_a\circ\eta^{-1}_a\in \sP(E_a)$.
		\end{proposition}
		\begin{proof}
			Let $\tau:b\rightarrow a$ be injective, $X_a\sim\mu_a$ and $Y_a = \eta_a(X_a)\sim \mu_a\circ\eta_a^{-1}$. It is $E[\tau](Y_a) = E[\tau]\circ \eta_a(X_a) \as \eta_b\circ D[\tau](X_a) \ed \eta_b(X_b)$.
		\end{proof}
		
		Four (parameterized) examples of Borel data structures are introduced to state the main results. All of these are array-type data structures, the general concept is in Definitions~\ref{def:indexing-system}~and~\ref{def:arrays}. In Section~\ref{sec:examples} many more examples of Borel data structures and ways of composing new ones from given ones are presented.
		
		\begin{definition}[First examples of BDS]\label{def:first-examples}
			Let $\cX$ be a Borel space and $k\geq 0$.
			\begin{itemize}
				\item $D= \SEQ(\cX) = \ARRAY(\cX,\square)$ with $D_a = \cX^a$ and $D[\tau](x) = x\circ\tau$. 
				\item $D = \ARRAY(\cX, 2^{\square})$ with $D_a = \cX^{2^a}$ and $D[\tau](x) = x\circ \im(\tau)$.
				\item $D = \ARRAY(\cX, \binom{\square}{\leq k}), k\geq 0$ with $D_a = \cX^{\binom{a}{\leq k}}$ and $D[\tau](x) = x\circ\im(\tau)$. 
				\item $D = \ARRAY(\cX, \square^*_{\neq})$ with $D_a = \cX^{a^*_{\neq}}$ and $D[\tau](x) = x\circ \vec{\tau}$.
			\end{itemize}
		\end{definition}	
	
		Iid uniform random variables $U_a, a\in\finN$, frequently used in FRTs, are mirrored in this framework by the following:
	
		\begin{definition}[Uniform randomizer]
			The following notations are used:
			$$R = \ARRAY\big([0,1], 2^{\square}\big)~~~\text{and}~~~\Rk = \ARRAY\Big([0,1],\binom{\square}{\leq k}\Big).$$
			The exchangeable laws $\ur\in\SYM(R)$ and $\urk\in\SYM(\Rk)$ are defined by 
			$$\ur_a = \unif[0,1]^{\otimes 2^a}~~~\text{and}~~~\urk_a = \unif[0,1]^{\otimes \binom{a}{\leq k}}.$$
			The letter "$R$" is used for "Randomization".
		\end{definition}
	
		A key result to our approach is:
	
		\begin{theorem}[Theorem~A via natural transformations]\label{thm:equivalence}
			The following are equivalent:
			\begin{itemize}
				\item[(i)] Theorem~A (representation of $\cX$-valued exchangeable arrays indexed by $\bN^*_{\neq}$),
				\item[(ii)] for every $\mu\in \SYM(\ARRAY(\cX,\square^*_{\neq}))$ exist a natural transformation $\eta:R\rightarrow \ARRAY(\cX,\square^*_{\neq})$ such that $\mu = \ur\circ\eta^{-1}$.
			\end{itemize}
		\end{theorem}
	
		The latter allows to translate Theorem~A into the language of natural transformations, which is used to prove the following: 
		
		\begin{theorem}[Weak FRT]\label{thm:frt-weak}
			For every Borel data structure $D$ and exchangeable law $\mu\in\SYM(D)$ exists a $\ur$-a.s. natural transformation $\eta:R\rightarrow D$ such that $\mu = \ur\circ\eta^{-1}$.
		\end{theorem}
	
		In Corollary~\ref{cor:frt-rvs} the result is presented using random variables. There are examples of BDS $D$ for which no exchangeable laws exists, that is $\SYM(D)=\emptyset$, see Example~\ref{exmp:nolaws} later. A direct consequence of Theorem~\ref{thm:frt-weak} is 
		\begin{equation*}
			\SYM(D)\neq\emptyset~~\Longleftrightarrow~~\text{there exists a $\ur$-a.s. natural transformation $\eta:R\rightarrow D$.}
		\end{equation*}	
		
		As seen in (E1)-(E3), known FRTs may not need randomization of arbitrary high level. This can be involved in the Theorem by defining the \emph{depth} of a BDS:
		
		\begin{definition}[Depth]\label{def:depth}
			A BDS $D$ is $k$-determined, $k\geq 0$, if for every finite set $a$ and $x,y\in D_a$ the following implication holds
			\begin{equation*}
			D[\iota_{a',a}](x) = D[\iota_{a',a}](y)~\text{for all $a'\subseteq a$ with $|a'|\leq k$}~~~\Longrightarrow~~~x=y.
			\end{equation*}
			Let $\depth(D) := \min\{k|~\text{$D$ is $k$-determined}\}$ with $\min\emptyset=\infty$.
		\end{definition}
		
		\begin{theorem}[Weak FRT for finite depth]\label{thm:frt-weak-depth}
			Let $D$ be a Borel data structure with $k = \depth(D) < \infty$. For every exchangeable law $\mu\in\SYM(D)$ there exists a $\urk$-a.s. natural transformation $\eta:\Rk\rightarrow D$ such that $\mu = \urk\circ\eta^{-1}$.
		\end{theorem}
		
		\begin{remark}[Weak FRT for ergodic laws]\label{rem:ergodicmod}
			Another refinement of the weak FRTs can be made for \emph{ergodic} exchangeable laws: define the BDS $R^{\circ}$ by $R^{\circ}_a = [0,1]^{2^a\setminus\{\emptyset\}}$, $R^{\circ}[\tau](x) = x\circ\im(\tau)$ and the exchangeable law $a\mapsto \unif(R^{\circ})_a = \unif[0,1]^{\otimes 2^a\setminus\{\emptyset\}}$. It can be shown that for every Borel data structure $D$ and every ergodic $\mu\in\eSYM(D)$ there exists a $\unif(R^{\circ})$-a.s. transformation $\eta:R^{\circ}\rightarrow D$ with $\mu = \unif(R^{\circ})\circ \eta^{-1}$. The same can be stated for finite depth by introducing $R^{k,\circ}$ and $\unif(R^{k,\circ})$ in an obvious analogue way.
		\end{remark}
	
		\begin{remark}[Global axiom of choice]\label{rmk:globalaxiomofchoice}
			The weak FRT is about the existence of a $\ur$-almost sure natural transformation. Such objects are "rules" that map \emph{any} finite set to a measurable map; from an axiomatic point of view, such rules are functions between proper classes. A suitable axiomatization of mathematics to work with proper classes are, for example, given by the NBG-axioms (Neumann-Bernays-Gödel). Often included in the NBG-axioms is the \emph{global axiom of choice}, which states that there exists a rule that \emph{simultaneously} picks an element from \emph{any} non-empty set. This axiom will be used several times in our proofs, which makes many results NBG-theorems. However, this is not problematic if one wishes to not leave the ZFC-world: all our NBG-theorems involving a quantifier "for all finite sets" (maybe within involved definitions) give an evenly interesting theorem by restricting the quantifier to "for all finite subsets of some fixed infinite set". Our NBG-Theorems obtained by this restriction talk about sets only. Now NBG is a \emph{conservative extension} of ZFC: every NBG-theorem talking about sets only also is a ZFC-theorem, that is could have been proved within ZFC alone, see \cite{felgner1971comparison}. An alternative approach to handle these foundational aspects is to postulate the existence of sufficiently rich Grothendieck universes and call "sets" only  elements of these, see Section I.6 in \cite{mac2013categories}. The global axiom of choice is used also in the index arithmetic being developed in Section~\ref{sec:arrays}, see the discussion in Example~\ref{exmp:skeleton} there.
		\end{remark}
	
		\subsection{A strong FRT for array-type data structures}
	
		The weak FRT is weak in the sense that it only guarantees the existence of an $\unif(R)$-almost sure natural transformation $\eta:R\rightarrow D$ to represent $\mu\in\SYM(D)$ via $\mu = \unif(R)\circ\eta^{-1}$. The question arises in what circumstances this can be strengthened to a "strong" form in which a true natural transformation can be used for a functional representation. The following shows that this can not always be the case, in fact there may exists no true natural transformations $R\rightarrow D$ at all:
		
		\begin{exmp}
			The existence of a true natural transformation $\eta:R\rightarrow D$ implies that for every $a$ there exists $x\in D_a$ with $x = D[\pi](x)$ for every bijection $\pi:a\rightarrow a$; choose $x = \eta_a(u)$ with $u\in[0,1]^{2^a}$ having a constant value $u(a')\equiv w\in[0,1]$. One example of $D$ in which there exists no true natural transformation $\eta:R\rightarrow D$ but exchangeable laws exist is given by the combinatorial data structure of total orders, see Example~\ref{exmp:binrel}.
		\end{exmp}
		
		A class of data structures where the weak FRT can be strengthened to a strong version are array-type data structures. Also, it is possible to give an explicit "low-level" description of the "high-level" concept natural transformations mapping into arrays, which allows to give low-level descriptions of the strong FRT in the usual style of such representation results.\\
		
		Indexing systems are defined as functors $I:\INJ\rightarrow\INJ$ satisfying additional axioms, in an explicit form:
	
		\begin{definition}[Indexing system]\label{def:indexing-system}
			An indexing system $I$ is a rule that maps 
			\begin{itemize}
				\item every finite set $b$ to a finite set $I_b$
				\item every injection $\tau:b\rightarrow a$ to an injection $I[\tau]:I_b\rightarrow I_a$
			\end{itemize} 
			such that the following hold
			\begin{enumerate}
				\item $I[\sigma\circ\tau] = I[\sigma]\circ I[\tau]$ for all composable injections $\sigma, \tau$,
				\item $I_b\cap I_a = I_{b\cap a}$ for all finite sets $b, a$,
				\item $I[\iota_{b',b}] = \iota_{I_{b'},I_b}$ for all finite sets $b'\subseteq b$.
			\end{enumerate}
		\end{definition}
	
		Indexing systems are introduced to define array-type data structures:
	
		\begin{definition}[Array-type data structure]\label{def:arrays}
			Let $\cX$ be a Borel space (data type) and $I$ an indexing system. The Borel data structure $D = \ARRAY(\cX, I)$ is defined by
			$$D_a = \cX^{I_a}~~~\text{and}~~~D[\tau](x) = x\circ I[\tau].$$
		\end{definition}
	
		The previous examples of array-type data structures used the indexing systems
		\begin{itemize}
			\item $I = \square$ with $I_b = b$ and $I[\tau] = \tau$,
			\item $I = 2^{\square}$ with $I_b = 2^b$ and $I[\tau] = \im(\tau)$, 
			\item $I = \binom{\square}{\leq k}$ with $I_b = \binom{b}{\leq k}$ and $I[\tau] = \im(\tau)$,
			\item $I = \square^*_{\neq}$ with $I_b = b^*_{\neq}$ and $I[\tau] = \vec{\tau}$. 
		\end{itemize}
	
		\begin{exmp}
			The indexing system axioms give that any index $\bbi$ from an indexing system $I$, that is $\bbi\in I_a$ for some $a$, has a \emph{unique minimal set of IDs used to build $\bbi$}: there exists a unique finite set $b$ with $\bbi\in I_{b}$ and $\bbi\in I_a\Rightarrow b\subseteq a$. Later we write $b = \dom(\bbi)$ (the domain of $\bbi$). Not every functor $I:\INJ\rightarrow\INJ$ is an indexing system, an example: let $k\geq 2$ and $I_b = b$ in case $|b|\geq k$ and $I_b = \emptyset$ in case $|b|<k$. For an injection $\tau:b\rightarrow a$ let $I[\tau] = \tau$ in case $|b|\geq k$ and $I[\tau]:\emptyset\rightarrow I_a$ the unique function on domain $\emptyset$ in case $|b|<k$. For two sets $a,b$ with $|a|, |b|\geq k$ and $1\leq |a\cap b|<k$ it is $I_a\cap I_b = a\cap b\neq\emptyset = I_{a\cap b}$. In this case no domains can be defined. 
		\end{exmp}
	
		Every array-type data structure $\ARRAY(\cX, I)$ has exchangeable laws: $\nu\in\sP(\cX)$ gives an exchangeable law $a\mapsto \mu_a = \nu^{\otimes I_a}$. In case $I=2^{\square}, \cX=[0,1]$ and $\nu = \unif[0,1]$ it is $\ARRAY(\cX, I) = R$ and the latter rule equals $\ur\in\SYM(R)$ used in the weak FRT, Theorem~\ref{thm:frt-weak}.
	
		\begin{definition}[Products of BDS]\label{def:products}
			For every countable family of Borel data structures $D^{(l)}, l\in L$ it is $D = \prod_{l\in L} D^{(l)}$ defined by 
			$$D_a = \prod_{l\in L}D^{(l)}_a~~~\text{and}~~~D[\tau]\big((x^{(l)})_{l\in L}\big) = \big(D^{(l)}[\tau](x^{(l)})\big)_{l\in L}$$
			a new Borel data structure. 
		\end{definition}
	
		More constructions such as coproducts, composition or sub-data structures are presented later.
	
		\begin{theorem}[Strong FRT for products of array-type data structures]\label{thm:frt-strong}
			For every countable product of array-type data structures $D = \prod_{l\in L}\ARRAY(\cX^{(l)}, I^{(l)})$ and every exchangeable law $\mu\in\SYM(D)$ there exists a (true) natural transformation $\eta:R\rightarrow D$ such that $\mu = \ur\circ \eta^{-1}$.\\
			In case $k=\depth(D)<\infty$ one can replace $(R, \ur)$ by $(\Rk, \urk)$. 
		\end{theorem}
	
		Theorem~\ref{thm:frt-strong} can be reformulated using category theory vocabulary: it shows the existence of a \emph{weak universal arrow} for the $\SYM$-functor defined on array-type data structures, see Remark~\ref{rem:ct-reformulation} for details. The strong FRT becomes particularly important combined with the following, which gives an \emph{explicit} description of natural transformations mapping into countable products of array-type data structures:
		
		\begin{theorem}[Characterization of natural transformations mapping into arrays]\label{thm:nat-in-array}
			For every Borel data structure $E$ there exists an explicit one-to-one correspondence between natural transformations $\eta:E\rightarrow \prod_{l\in L} \ARRAY(\cX^{(l)},I^{(l)})$ and certain countable families of \emph{kernel functions} $\cF$ in which for each $f\in \cF$ there is some $m\in M$, $k\geq 0$ and sub-group $G\subseteq \bS_k$ such that $f:E_{\{1,\dots,k\}}\rightarrow \cX^{(m)}$ is measurable with $f\circ E[\pi] = f$ for all $\pi\in G$.
		\end{theorem}
	
		Some prior work is needed to explicitly state the correspondence and how the set $\cF$ is constructed, the case $L=\{1\}$ is stated in Theorem~\ref{thm:nat-array}. In this case, the theorem characterizes natural transformations $\eta:E\rightarrow\ARRAY(\cX,I)$. The groups $G$ imposing symmetry restrictions on kernel functions depend on the indexing system $I$ and in this regard the indexing systems $I = 2^{\square}, \square^*_{\neq}$ represent the two extreme cases: the former leads to full subgroups $G = \bS_k$, the latter to trivial subgroups $G = \{\id_{[k]}\}\subsetneq \bS_k$. It is seen in Theorem~\ref{thm:everygroup} that, up to group-isomorphism, for every finite group $G$ there exists an indexing system $I$ such that $G$ can appear as a symmetry restriction on a kernel function. 
		
		\begin{remark}[Skew-products]\label{rem:skew}
			In \cite{austin2015exchangeable} the notions of \emph{skew-product tuples} and \emph{skew-product type functions} were introduced; in our terminology these concept are about natural transformations $\eta:\prod_{j=0}^k\ARRAY(\cX^{(j)},\binom{\square}{j})\rightarrow \prod_{j=0}^k\ARRAY(\cY^{(j)},\binom{\square}{j})$; loosely speaking, skew-product tuples correspond to kernel functions in the sense of Theorem~\ref{thm:nat-in-array} and the associated skew-product type function to the obtained natural transformation.
		\end{remark}
		
		\subsection{Universality of $\square^*_{\neq}$}
		
		The key importance of the indexing system $\square^*_{\neq}$, and hence Theorem~A, is that indices can be identified with injections mapping $[k] = \{1,\dots,k\}$ into finite sets: every index $\bbi = (i_1,\dots,i_k)\in a^*_{\neq}$ gives the injection $\tau_{\bbi,a}:[k]\rightarrow a, j\mapsto i_j$. The whole concept of Borel data structures is based on the Borel space assumption and on handling injective maps; and in fact, $\ARRAY([0,1],\square^*_{\neq})$ plays a crucial role in the theory. 

		\begin{definition}[Embedding and isomorphism]\label{def:embedding}
			Let $D, E$ be Borel data structures and $\eta:D\rightarrow E$ be a natural transformation. $\eta$ is called 
			\begin{itemize}
				\item embedding if all components $\eta_a:D_a\rightarrow E_a$ are injective, 
				\item isomorphism if all components $\eta_a:D_a\rightarrow E_a$ are bijective. 
			\end{itemize}
			$D$ and $E$ are called isomorphic of there exists an isomorphism between them. 
		\end{definition}
		
		It is easy to check that if $\eta:D\rightarrow E$ is an isomorphism then the rule $\eta^{-1}$ having as components the inverse functions $\eta_a^{-1}$ of $\eta_a$ is a natural transformation $\eta^{-1}:E\rightarrow D$ with $\eta\circ\eta^{-1} = \id_E$ and $\eta^{-1}\circ\eta = \id_D$, measureability is given by the Borel space assumption.
	
		\begin{definition}[Sub-data structures]\label{def:sub-data-structure}
			Let $D', D$ be Borel data structures. $D'$ is a sub-data structure of $D$, denoted with $D'\subseteq D$, if for every $a$ and injection $\tau:b\rightarrow a$
			\begin{itemize}
				\item $D'_a\subseteq D_a$ is a measurable subspace,
				\item $D'[\tau](x) = D[\tau](x)$ for all $x\in D'_a$. 
			\end{itemize}
		\end{definition}
	
		\begin{remark}
			In the context of $D=\GRAPH = \ARRAY(\{0,1\},\binom{\square}{2})$ sub-data structures $D'\subseteq D$ correspond to \emph{hereditary} graph properties: if $\cP$ is a hereditary graph property, $D'_a = \{x\in D_a|~\text{$x$ satisfies $\cP$}\}$ gives a sub-data structure $D'\subseteq\GRAPH$, see the introduction in \cite{austin2010testability} and the later Section~\ref{sec:examples}.
		\end{remark}
	
		\begin{proposition}\label{prop:embedding-structure}
			If $\eta:D\rightarrow E$ is an embedding then $E' = \eta D$ defined by $E'_a = \eta_a(D_a)$ and $E'[\tau](x) = E[\tau](x), x\in E'_a$ is a sub-data structure of $E$ isomorphic to $D$, an isomorphism is given by $\hat\eta$ with components $\hat\eta_a:D_a\rightarrow E'_a, x\mapsto \eta_a(x)$. 
		\end{proposition}
		\begin{proof}
			$\eta_a:D_a\rightarrow E_a$ is a measurable injection between Borel spaces, thus the image $E'_a=\eta_a(D_a)\subseteq E_a$ is a measurable subspace and hence a Borel space. For every $x\in E'_a$ there is a unique $y\in D_a$ with $x = \eta_a(y)$ and for $\tau:b\rightarrow a$ it is $E'[\tau](x) = E[\tau](x) = E[\tau]\circ \eta_a(y) = \eta_b\circ D[\tau](y) \in E'_b$, which shows that $E'=\eta D$ is a sub-data structure of $E$. The natural inverse of $\hat\eta$ has components $\hat\eta_a^{-1}$, which are measurable by Borel space assumptions, the naturality of $\hat\eta$ and $\hat\eta^{-1}$ is straightforward.  
		\end{proof}
		
		Note that the inverse of $\hat\eta$ is a natural transformation $\hat\eta^{-1}:E'\rightarrow D$ with $E' = \eta D\subseteq E$ and can in general not be extended to natural transformation defined on the whole BDS $E$. This is different from embeddings between Borel spaces: if $f:\cX\rightarrow\cY$ is a measurable injection between Borel spaces, then there exists a measurable left-inverse $g:\cY\rightarrow\cX$ of $f$, that is $g\circ f = \id_{\cX}$. In category theory terminology: in $\BOREL$ every monomorphism is a section, which is not the case in the functor category $[\INJo,\BOREL]$.\\
		
		In Theorem~\ref{thm:embedding} it is shown that every Borel data structure can be naturally embedded in $\ARRAY([0,1],\square^*_{\neq})$, the embedding being more or less explicit, but of little practical interest. However, together with Proposition~\ref{prop:embedding-structure} this yields:
		
		\begin{theorem}[Universality]\label{thm:universality}
			Every Borel data structure is naturally isomorphic to a sub-data structure of $\ARRAY([0,1],\square^*_{\neq})$. 
		\end{theorem}	
		
		\section{Examples and Constructions}\label{sec:examples}
		
		\begin{exmp}[Array-type data structures]\label{exmp:arrays}
			Examples of array-type data structures $D=\ARRAY(\cX,I)$ are obtained by giving examples of indexing systems $I$, that is specifying the finite set $I_b$ and for every $\bbi\in I_b$ and $\tau:b\rightarrow a$ the value $I[\tau](\bbi)\in I_a$. Note that in case $\cX$ is a finite set $\ARRAY(\cX,I)$ is a combinatorial data structure.\\
			Let $k\geq 0$. 
			\begin{itemize}
				\item $I=\square$ with $I_b = b$ and $I[\tau] =\tau$ is the indexing system in which IDs equal indices. 
				\item Set-type indexing systems are of the form $I_b\subseteq 2^b$ and $I[\tau] = \im(\tau)$. Examples are the indexing systems $I = 2^{\square}, \binom{\square}{k}, \binom{\square}{\leq k}$ having sets of indices $I_b = 2^b, \binom{b}{k}, \binom{b}{\leq k}$. Note that injectivity of $\tau$ gives $I[\tau](I_b)\subseteq I_a$ in all these cases. 
				\item Tuple-type indexing systems are of the form $I_b\subseteq b^* = \cup_{k\geq 0}b^k$ and $I[\tau]=\vec{\tau}$. Examples are the indexing systems $I=\square^*_{\neq}, \square^k_{\neq}, \square^k$ having sets of indices $I_b = b^*_{\neq}, b^k_{\neq}, b^k$, where the sup-script $\neq$ indicates that only tuples with distinct entries are considered.  
				\item Let $I, J$ be two indexing systems. New indexing systems are defined by 
				\begin{itemize}
					\item Products: $I\times J$ with $(I\times J)_a = I_a\times J_a$ and $(I\times J)[\tau](\bbi,\bbj) = (I[\tau]\bbi, I[\tau]\bbj)$,
					\item Coproducts: $I\sqcup J$ are defined analogously,
					\item Composition: $I\circ J$ with $(I\circ J)_a = I_{J_a}$ and $(I\circ J)[\tau] = I[J[\tau]]$. 
				\end{itemize}
				\item Every species of structures $C:\BIJ_+\rightarrow \BIJ_+$ can be turned into an indexing system $I = I(C)$: let $I_b = \sqcup_{b'\subseteq b}C_{b'}$ and for $b'\subseteq b, x\in C_{b'}$, that is $\bbi = (b,x)\in I_b$, let $I[\tau](\bbi) = (\tau(b'), C[\pi](x))$ with $\pi:b'\rightarrow \tau(b'), i\mapsto \tau(i)$. 
			\end{itemize}
		\end{exmp}

		\begin{definition}[Set systems]\label{def:setsystem}
			The combinatorial data structure $D = \SETSYSTEM$ is defined by $D_a = 2^{2^a}$, that is elements $x\in D_a$ are subsets $x\subseteq 2^a$, and for injective map $\tau:b\rightarrow a$ and $x\in D_a$ it is $D[\tau](x) = \{\tau^{-1}(a')|a'\in x\}$. 
		\end{definition}
	
		There is a canonical bijection between the set of set systems $2^{2^a}$ and the set of functions $\{0,1\}^{2^a}$ by mapping $x\subseteq 2^a$ to the indicator function $a'\subseteq a\mapsto 1(a'\in x)$. This is \emph{not} a natural isomorphism between $\SETSYSTEM$ and $\ARRAY(\{0,1\},2^{\square})$:
	
		\begin{proposition}\label{prop:setsystem}
			$\SETSYSTEM$ and $\ARRAY(\{0,1\},2^{\square})$ are not naturally isomorphic. 
		\end{proposition}
		\begin{proof}
			Let $D=\ARRAY(\{0,1\}, 2^{\square})$, $E=\SETSYSTEM$ and $b\subseteq a$ with $k=|b|\leq |a|=n$.\\
			Let $x\in D_b = \{0,1\}^{2^b}$ and consider the set
			$$\Big\{~y\in D_a = \{0,1\}^{2^a}~~\Big|~~D[\iota_{b,a}](y) = y\circ \iota_{2^b,2^a} = x~\Big\}.$$
			This set has cardinality $2^{2^n - 2^k}$, not depending on the concrete choice of $x$.\\
			Now let $x\in E_b$, that is $x\subseteq 2^b$, and consider the set 
			$$\Big\{~y\in E_a = 2^{2^a}~~\Big|~~E[\iota_{b,a}](y) = \{b\cap a'|a'\in y\} = x~\Big\}.$$
			If $D$ and $E$ would be naturally isomorphic, this set would have the same cardinality $2^{2^n - 2^k}$ independent on the concrete choice of $x$. But this does not hold: let $x = \{b\}\in E_b$. It is $\{b\cap a'|a'\in y\} = \{b\}$ if and only if for all $a'\in y$ it holds that $a'\supseteq b$. In particular, for this specific $x$ there are precisely $2^{2^{n-k}} - 1$ such $y$. Clearly, $2^{2^{n-k}} - 1 \neq 2^{2^n - 2^k}$ for $n>k$.
		\end{proof}	
	
		\begin{exmp}[Three implementations of graphs]\label{exmp:graphs}
			An undirected loop-free graph can be defined as either (1) a pair of vertices and edges, (2) an edge indicator function or (3) as an adjacency matrix. These "implementations" of graphs can be formalized using the BDS framework and are seen to be naturally isomorphic:		
			\begin{enumerate}
				\item Pairs of vertices and edges: $D=\GRAPH^{(1)}$ is defined by $D_a = \{(a, E)|E\subseteq \binom{a}{2}\}$ and $D[\tau](x) = D[\tau]((a,E)) = (b, \{e\in\binom{b}{2}|\tau(e)\in E\})$,
				\item Edge indicator functions: $D=\GRAPH^{(2)}$ is defined by $D = \ARRAY(\{0,1\},\binom{\square}{2})$,
				\item Adjacency matrices: $D=\GRAPH^{(3)}$ is defined as a sub-data structure $D\subseteq\ARRAY(\{0,1\},\square^2)$ with 
				$$D_a = \{x\in \{0,1\}^{a^2}|x(i,i)=0~\text{and}~x(i,i')=x(i',i)~\text{for all}~i,i'\in a\}.$$
			\end{enumerate}
			Natural isomorphisms between these implementations are
			\begin{itemize}
				\item $\eta:\GRAPH^{(1)}\rightarrow \GRAPH^{(2)}$ with $\eta_a((a,E)) = \big[e\in \binom{a}{2}\mapsto 1(e\in E)\big]$.
				\item $\eta:\GRAPH^{(2)}\rightarrow \GRAPH^{(3)}$ with $\eta_a(x) = \big[(i,i')\in a^2\mapsto 1(i\neq i')x(\{i,i'\})\big]$.
				\item $\eta:\GRAPH^{(3)}\rightarrow \GRAPH^{(1)}$ with $\eta_a(x)\mapsto (a, \{\{i,i'\}\in\binom{a}{2}|x(i,i')=x(i',i)=1\})$.
			\end{itemize}
		\end{exmp}
	
		\begin{definition}[Products, coproducts, composition]
			Let $D, E$ be Borel data structures and let $T:\BOREL\rightarrow \BOREL$ and $I:\INJ\rightarrow\INJ$ be endofunctors.
			\begin{itemize}
				\item $D\times E$ is defined by $(D\times E)_a = D_a\times E_a$ and $(D\times E)[\tau](x,y) = (D[\tau]x, E[\tau]y)$,
				\item $D\sqcup E$ is defined analogously,
				\item $T\circ D$ is defined by $(T\circ D)_a = T_{D_a}$ and $(T\circ D)[\tau] = T[D[\tau]]$. One important example is $T=\sP$, the probability measure endofunctor, see Remark~\ref{rem:exch_via_cat}. In case $D=\prod_{l=0}^k\ARRAY(\cX^{(l)}, \binom{\square}{l})$ exchangeable laws in $\sP\circ D$ have been studied in \cite{austin2015exchangeable}. Note that exchangeable laws in $\sP\circ D$ correspond to the limit of the functor $\sP\circ\sP\circ D$. The results in \cite{austin2015exchangeable} lead to a conjecture later, see Remark~\ref{rem:conj},
				\item $D\circ I$ is defined by $(D\circ I)_a = D_{I_a}$ and $(D\circ I)[\tau] = D[I[\tau]]$. In case $I$ is an indexing system it holds that $\ARRAY(\cX, I) = \SEQ(\cX)\circ I$. 
			\end{itemize}
		\end{definition}
	
		\begin{exmp}[Binary relations and hereditary properties therein]\label{exmp:binrel}
			A binary relation on a set $a$ can be seen a subset $x \subseteq a\times a = a^2$. If $\tau:b\rightarrow a$ is an injection then $\{(i,i')\in b\times b|(\tau(i),\tau(i'))\in x\}\subseteq b\times b$ defines a new binary relation on $b$ and this gives the combinatorial data structure $D = \BINREL$ of binary relations, which is naturally isomorphic to $\ARRAY(\{0,1\},\square^2)$ by mapping $x\subseteq a\times a$ to the indicator $(i,i')\in a\times a\mapsto 1((i,i')\in x)$.\\
			Many standard properties of binary relations are hereditary, that is stable under $D[\tau]$, such as: symmetry, transitivity, reflexivity, connectedness, anti-symmetry,$\dots$ and thus yield sub-data structure of $\BINREL\simeq\ARRAY(\{0,1\},\square^2)$.\\
			One example important for illustrative purposes: a binary relation $x$ on $a$, implemented as an array $x\in\{0,1\}^{a\times a}$, is a strict total order iff for all $i_1,i_2,i_3\in a$
			\begin{equation*}
				x(i_1,i_1)=0,~~~x(i_1,i_2) = 1 - x(i_2,i_1)~~~\text{and}~~~x(i_1,i_2)x(i_2,i_3)[1-x(i_1,i_3)] = 0.
			\end{equation*}
			Being a strict total order is hereditary and gives the data structure $\TOTAL\subset \ARRAY(\{0,1\},\square^2)$ with $\TOTAL_a$ the subset of strict total orders on $a$.
		\end{exmp}
	
		\begin{exmp}[Exchangeable total order]\label{rmk:exchtotord}
			The exchangeability theory of $\TOTAL$ is folklore, see for example \cite{gerstenberg2020general} or \cite{gerstenberg2020exchangeable}: there exists exactly one exchangeable law on $\TOTAL$ given by $a\mapsto \mu_a = \unif(\TOTAL_a)$, which is ergodic by uniqueness. It is $\depth(\TOTAL) = 2$, but a weak representation in the style of Theorem~\ref{thm:frt-weak-depth} only needs level $1$ randomization. For a finite set $a$ and $U_i, i\in a$ iid $\sim\unif[0,1]$ define a random strict total order $<_a$ on $a$ as $i_1<_a i_2:\Leftrightarrow U_{i_1}<U_{i_2}$. Note that this gives a strict total order with probability one and is equivalent to a weak representation $\mu_a = \eta_a^{-1}\circ\unif(\square)_a$ with $\unif(\square)\in\SYM(\SEQ([0,1]))$ being $\unif(\square)_a = \cL((U_i)_{i\in a}) =\unif[0,1]^{\otimes a}$ and $\mu$ the a.s. natural transformation $\eta:\SEQ([0,1])\rightarrow \TOTAL$ defined as $\mu_a((U_i)_{i\in a}) := <_a$ (and arbitrary on a set with $\unif(\square)_a$-probability zero). 
		\end{exmp}
	
		\begin{exmp}[Sub-data structures of $\SETSYSTEM$]
			Sub-data structures of $D = \SETSYSTEM$ correspond to hereditary set system properties, that is properties $\cP$ such that if a set system $x\subseteq 2^a$ fulfills $\cP$ and $\tau:b\rightarrow a$ is an injection, then $D[\tau](x) = \{\tau^{-1}(a')|a'\in x\}\subseteq 2^b$ also satisfies $\cP$, for every such property $D'_a = \{x\subseteq 2^a|~\text{$x$ satisfies $\cP$}\}$ gives a sub-data structure $D'\subseteq \SETSYSTEM$. Examples of such properties are
			\begin{itemize}
				\item $x\subseteq 2^a$ being a \emph{partition}: $\emptyset\in x, a_1,a_2\in x\Rightarrow a_1\cap a_2=\emptyset$ and $a = \cup_{a'\in x}a'$ (including the empty set in this case is a question of implementation and does not affect the essence of what makes a partition).
				\item $x\subseteq 2^a$ being a \emph{total partition}, also called \emph{hierarchy}: $\emptyset, a\in x, \{i\}\in x$ for all $i\in a$, $a_1,a_2\in x\Rightarrow a_1\cap a_2\in \{a_1,a_2,\emptyset\}$. 
				\item $x\subseteq 2^a$ being an \emph{interval hypergraph}: $\emptyset\in x, \{i\}\in x$ for all $i\in a$ and there exists a strict total order $y\in \TOTAL_a$ such that every $a'\in x$ is an interval with respect to $y$, that is: $i_1,i_2\in a'$ and $i_3\in a$ with $y(i_1,i_3)=y(i_3,i_2)=1$ then $i_3\in a'$. 
			\end{itemize}
			Exchangeability in partitions has a representation by Kingmans's paintbox construction, representations for exchangeable total partitions are by \cite{forman2018representation} and \cite{gerstenberg2020exchangeable} and for interval hypergraphs by \cite{gerstenberg2020exchangeable}. The functional representation in \cite{gerstenberg2020exchangeable} can be translated into the style of FRTs: for every exchangeable law $\mu$ over interval hypergraphs there exists a random compact subset $K$ of the triangle $\{(x,y)\in[0,1]^2|x\leq y\}$ such that $\mu_a \sim \{\{i\in a|x<U_{\{i\}}<y\}|(x,y)\in K\}$ for every finite set $a$, where $K, U_{\{i\}}, i\in a$ are independent. Letting $K \ed g(U_{\emptyset})$ and defining $\eta_a((u_e)_{e\subseteq \binom{a}{\leq 1}}) = \{\{i\in a| x<u_{\{i\}}<y\}|(x,y)\in g(u_{\emptyset})\}$ defines a $\unif(R^1)$-almost sure natural transformation mapping into interval hypergraphs such that $\mu = \unif(R^1)\circ \eta^{-1}$.
		\end{exmp}

		\begin{exmp}[Examples with $\SYM(D)=\emptyset$]\label{exmp:nolaws}
			Exchangeable laws always exist in array-type data structures (product measures) and in combinatorial data structures (by a compactness argument). Two examples of a BDS $D$ without exchangeable laws are:	
			\begin{itemize}
				\item $D_a = (0,1)$ for every $a$ (the open unit interval) and $D[\tau]:(0,1)\rightarrow (0,1), x\mapsto x/{2^{|a|-|b|}}$ for injection $\tau:b\rightarrow a$. Suppose $\mu\in \SYM(D)$ exists, write $X_a\sim \mu_a$. Applying exchangeability via $\tau = \iota_{\emptyset,[n]}, n\geq 0$ gives $X_{\emptyset} \ed X_{[n]}/2^n$, which converges to $0$ in probability as $n\rightarrow\infty$, thus $X_{\emptyset}\sim\delta_0$, which is a contradiction to $X_{\emptyset}$ taking values in $(0,1)$. 
				\item let $\cX$ be countable infinite and $D\subset \SEQ(\cX)$ the sub-data structure with $D_a = \cX^a_{\neq}$ the set of all injective functions $x:a\rightarrow \cX$. If there were $\mu\in\SYM(D)$ it would also be $\mu\in\SYM(\SEQ(\cX))$ such that $\mu_a(\cX^a_{\neq})=1$ for all $a$. By de~Finetti $\mu$ has to be a mixture over iid-laws, $\mu_a = \int \nu^{\otimes a} d\Xi(\nu)$, which implies for $|a|\geq 2$ that $\mu_a(\cX^a_{\neq})<1$ because $\cX$ is countable.
			\end{itemize}
		\end{exmp}
	
		\begin{exmp}\label{exmp:comp}
			Let $\GRAPH:\BIJ_+\rightarrow\BIJ_+$ be species of structures defining graphs. The previous discussion allows to consider the Borel data structure 
			$$D = \bigg[\sP\circ \Big\{\Big[\big(\ARRAY(\bR^3, 2^{\square}\circ\square^8_{\neq})\times \TOTAL\big)\sqcup \big(\sP\circ\SETSYSTEM\big)\Big]\circ I(\GRAPH)\Big\}\bigg]\circ \binom{\square}{\leq 10}.$$
		\end{exmp}

		\section{Extension, pointwise convergence and decomposition}\label{sec:random_variables}
		
		For this Section let $D:\INJo\rightarrow\BOREL$ be a fixed BDS.\\
		Let $A$ be countable infinite, e.g. $A=\bN$. Imagine a statistician picks a countable infinite group of individuals from a large population, uses IDs $i\in A$ to represent the individuals and then measures information $x_a\in D_a$ on each finite subgroup $a\in\finA$. The obtained measurements $(x_a)_{a\in\finA}$ should satisfy \emph{sampling consistency}
		\begin{equation*}
			x_{a'} = D[\iota_{a',a}](x_a)~~~\text{for all}~a'\subseteq a\in\finA.
		\end{equation*}
		
		If individuals are picked and IDs $i\in A$ distributed \emph{at random}, the obtained measurement $X = (\Xa)_{a\in\finA}$ should be an exchangeable random object in the following sense:
		
		\begin{definition}[Exchangeable $D$-measurement]\label{def:exchangeability}
			Let $A$ be countable infinite. An exchangeable $D$-measurement using IDs $A$ is a collection of random variables
			$$X = (X_a)_{a\in\finA}$$
			such that for every $a\in\finA$
			\begin{itemize}
				\item $\Xa$ takes values in $\Da$,
				\item $D[\iota_{a',a}](X_{a}) \as X_{a'}$ for every $a'\subseteq a$ (sampling consistency),
				\item $D[\pi](X_a)\ed X_a$ for every bijection $\pi:a\rightarrow a$ (exchangeability).
			\end{itemize}
			If only the first two hold $X$ is called random $D$-measurement using IDs $A$. Let
			$$\SYM(D; A) = \Big\{\cL(X)~|~\text{$X$ is an exchangeable $D$-measurement using IDs $A$}\Big\},$$
			that is
			$$\SYM(D; A)\subseteq \sP\Big(\prod\nolimits_{a\in\finA}D_a\Big).$$
		\end{definition}
	
		\begin{proposition}\label{prop:exchangeablelaws}
			Let $C$ be countable infinite and $\cL(X)\in\SYM(D;C)$. Let $a$ be a finite set. Choose $c\in\finC$ with $|c|=|a|$ and a bijection $\pi:a\rightarrow c$. Then $\cL(D[\pi](X_c))\in\sP(D_a)$ does not depend on the concrete choice of $c$ and $\pi$ which allows to define $\mu_a := \cL(D[\pi](X_c))\in\sP(D_a)$. The rule $\mu = [a\mapsto \mu_a]$ is element of $\SYM(D)$ and the map $\cL(X)\mapsto \mu$ is a one-to-one correspondence between $\SYM(D;C)$ and $\SYM(D)$. In particular, $\SYM(D)$ is a set.
		\end{proposition}
	
		The proof is based on Kolmogorov consistency arguments and placed in the Appendix. 
		
		\begin{definition}[Canonical extension to countable infinite sets of IDs]\label{def:extension}
			For $A$ countable infinite let 
			\begin{equation*}
				\pD_A = \bigg\{x=(x_a)_{a\in\finA}\in\prod\nolimits_{a\in\finA}D_a~\Big|~D[\iota_{a',a}](x_a)=x_{a'}~\text{for all}~a'\subseteq a\in\finA\bigg\}.
			\end{equation*}
			For any countable set $B$ (finite or infinite), injection $\tau:B\rightarrow A$ and $x = (x_a)_{a\in\finA}\in \pD_A$ let 
			\begin{equation*}
			D[\tau](x) = \begin{cases}
			D\big[\hat\tau\big](x_{\tau(b)}),&~\text{if $B=b$ is finite}\\
			\Big(D\big[\widehat{\tau\circ\iota_{b,B}}\big](x_{\tau(b)})\Big)_{b\in\finB},&~\text{if $B$ is infinite}.
			\end{cases}
			\end{equation*}
		\end{definition}
		
		It is easily seen that if $\tau:B\rightarrow A$ is an injection between two countable infinite sets then $x\in\pD_A$ implies $D[\tau](x)\in\pD_B$. In particular: $\pD_A\neq\emptyset$ for some countable infinite $A$ implies $\pD_B\neq\emptyset$ for every countable infinite $B$. The proof of the following is placed in the Appendix.
	
		\begin{proposition}\label{prop:translation}
			Assume $\SYM(D)\neq \emptyset$. Then
			\begin{itemize}
				\item[(1)] For every countable infinite $A$ it is $\pD_A$ a non-empty measurable subset of $\prod_{a\in\finA}D_a$ and hence a Borel space. In particular, random infinite $D$-measurements using IDs $A$ can be considered $\pD_A$-valued random variables. For two $D_A$-valued random variables $X, X'$ it holds $X\ed X'$ iff $X_a \ed X'_a$ for all $a\in\finA$,
				\item[(2)] The construction in Definition~\ref{def:extension} extends $D:\INJo\rightarrow\BOREL$ to a functor $D:\CINJo\rightarrow\BOREL$, where $\CINJ$ is the category of injections between countable sets, 
				\item[(3)] Let $X = (\Xa)_{a\in\finA}$ be a random $D$-measurement using IDs $A$, that is a $D_A$-valued random variable. The following are equivalent:
				\begin{itemize}
					\item[(i)] $X_a\ed D[\pi](X_a)$ for every $a\in\finA$ and bijection $\pi:a\rightarrow a$, that is: $X$ is an exchangeable $D$-measurement in the sense of Definition~\ref{def:exchangeability},
					\item[(ii)] $D[\pi](X) \ed X$ for every bijection $\pi:A\rightarrow A$ with $\pi(i)=i$ for all but finitely many $i\in A$,
					\item[(iii)] $D[\pi](X) \ed X$ for every bijection $\pi:A\rightarrow A$,
					\item[(iv)] $D[\tau](X) \ed X$ for every injection $\tau:A\rightarrow A$.
				\end{itemize}
				\item[(4)] For every injection $\tau:B\rightarrow A$ between countable infinite sets the map $\cL(X)\mapsto \cL(D[\tau](X))$ is a bijection $\SYM(D;A)\rightarrow \SYM(D;B)$.
			\end{itemize}
		\end{proposition}
	
		\begin{remark}
			Let $X$ be an exchangeable $D$-measurement using IDs $B$ whose law is represented by $\mu\in\SYM(D)$. Let $A\supseteq B$. Applying (4) to $\tau = \iota_{B,A}$ allows to represent $X$ as $X = D[\iota_{B,A}](\tilde X)$ with $\tilde X$ being an exchangeable $D$-measurement using IDs $A$, whose law is necessarily also represented by $\mu$. In case $\bZ\supseteq \bN$ such constructions are a basic approach to prove functional representation theorems for arrays, see \cite{aldous1982exchangeability}, \cite{aldous1985exchangeability} and \cite{austin2012exchangeable}.
		\end{remark}
	
		Combining the previous propositions with Theorems~\ref{thm:frt-weak}~and~\ref{thm:frt-weak-depth} gives the following reformulation of the FRTs:
		
		\begin{corollary}[Weak FRT for exchangeable random measurements]\label{cor:frt-rvs}
			For every exchangeable $D$-measurement $X = (\Xa)_{a\in\finA}$ there exists a $\ur$-almost sure natural transformation $\eta:R\rightarrow D$ such that 
			\begin{equation*}
			\big(\Xa\big)_{a\in\finA}~\ed~\Big(\eta_a\big((U_e)_{e\subseteq a}\big)\Big)_{a\in\finA},
			\end{equation*}
			where $U_e, e\in \finA$ are iid $\sim\unif[0,1]$. If $\depth(D) = k<\infty$ there is a $\urk$-a.s. natural transformation $\eta:\Rk\rightarrow D$ such that
			\begin{equation*}
			\big(\Xa\big)_{a\in\finA}~\ed~\Big(\eta_a\big((U_e)_{e\subseteq a, |e|\leq k}\big)\Big)_{a\in\finA}.
			\end{equation*}
		\end{corollary}
		\begin{proof}
			By Proposition~\ref{prop:exchangeablelaws} there is a unique $\mu\in\SYM(D)$ with $\mu_a\sim X_a$ for every $a$. By Theorem~\ref{thm:frt-weak} there is a $\ur$-a.s. natural transformation $\eta:R\rightarrow D$ with $\mu = \ur\circ\eta^{-1}$. This gives $X_a \ed \eta_a\big((U_e)_{e\subseteq a}\big)$ for every $a\in\finA$. Since $\eta$ is $\mu$-a.s. natural transformation 
			$(\eta_a\big((U_e)_{e\subseteq a}\big))_{a\in\finA}$ takes values in $D_A$ almost surely and the same is true for $(X_a)_{a\in\finA}$. The equality in distribution at each $a\in\finA$ implies equality in distribution of the whole $\finA$-indexed processes by (1) of Proposition~\ref{prop:translation}. The finite-depth case follows the same way by applying Theorem~\ref{thm:frt-weak-depth}.
		\end{proof}
	
		\subsection{Natural extensions of array-type data structures}\label{sec:extension-arrays}
	
			Let $D=\ARRAY(\cX,I)$. Since $\SYM(D)\neq\emptyset$ the functor $D:\INJo\rightarrow\BOREL$ can be extended to a functor $D:\CINJo\rightarrow\BOREL$ by the construction of Definition~\ref{def:extension}. A more \emph{natural} extension is possible for array-type data structures. The special case $D=\SEQ(\cX) = \ARRAY(\cX,\square)$ is very instructive: let $A$ be countable infinite, an element $x\in D_A$ in the canonical extension is of the form
			$$x = (x_a)_{a\in\finA}~~\text{with}~~x_a = (x_{a}(i))_{i\in a} \in \cX^a.$$
			Let $x_i := x_{\{i\}}(i), i\in A$. The defining property of $D_A$ (sampling consistency) gives 
			$$(x_a)_{a\in\finA} = ((x_i)_{i\in a})_{a\in\finA},$$
			which obviously can be represented more naturally as $(x_i)_{i\in A} \in \cX^A$. This works for every $D = \ARRAY(\cX,I)$: a natural extension is based on extending the indexing system, which is a functor $I:\INJ\rightarrow\INJ$ with additional properties, to a functor $I:\CINJ\rightarrow\CINJ$ and defining the natural extension of $D$ as $D_A = \cX^{I_A}$ and $D[\tau](x) = x\circ I[\tau]$. The extension of $I$ is as follows:\\
			Let $A$ be countable infinite. Define $I_A = \cup_{a\in\finA}I_a$ and for an injection $\tau:B\rightarrow A$, with $B$ finite or infinite, define 
			\begin{equation*}
			I[\tau]:I_B\rightarrow I_A, I[\tau](\bbi) = \begin{cases}
			I[\hat\tau](\bbi),&~\text{if $B=b$ is finite},\\
			I\big[\reallywidehat{\tau\circ\iota_{\dom(\bbi),B}}\big](\bbi),&~\text{if $B$ is infinite}.
			\end{cases}
			\end{equation*}
			Lemma~\ref{lemma:ind} later provides the main technical details to see that this extends $I$ to countable infinite sets, satisfying functor properties and also satisfying the indexing system axioms for countable infinite sets. Some examples: let $B, A$ be arbitrary countable and $\tau:B\rightarrow A$ be an injection:
			\begin{itemize}
				\item $D = \SEQ(\cX) = \ARRAY(\cX,\square)$ has natural extension $D_A = \cX^A$ and $D[\tau](x) = x\circ\tau$,
				\item $D = \ARRAY(\cX,\binom{\square}{k})$ has natural extension $D_A = \cX^{\binom{A}{k}}$ and $D[\tau](x) = x\circ\im(\tau)$,
				\item $D = \ARRAY(\cX,2^{\square})$ has natural extension $D_A = \cX^{\finA}$ and $D[\tau](x) = x\circ\im(\tau)$,
				\item $D = \ARRAY(\cX,\square^*_{\neq})$ has natural extension $D_A = \cX^{A^*_{\neq}}$ and $D[\tau](x) = x\circ\vec{\tau}$.
			\end{itemize}
			In particular, exchangeable random measurements using IDs $\bN$ now fit the framework (\ref{eq:array}) presented in the introduction: the group action on indices is $\bSi\times I_{\bN}\rightarrow I_{\bN}, (\pi, \bbi)\mapsto I[\pi](\bbi)$ and following Proposition~\ref{prop:exchangeablelaws} shows that laws of exchangeable processes $(X_{\bbi})_{\bbi\in I_{\bN}}$ can be identified with $\SYM(\ARRAY(\cX,I))$ (by passing from canonical to natural extension). An exchangeable array in natural extension $(X_{\bbi})_{\bbi\in I_{\bN}}$ corresponds to $\big((X_{\bbi})_{\bbi\in I_a}\big)_{a\in\finN}$ in canonical extension.
		
		\begin{remark}
			With $D=\SETSYSTEM$ it is not obvious if there is an extension that is any more "natural" than the canonical one from Definition~\ref{def:extension}. Note that both in \cite{forman2018representation} and \cite{gerstenberg2020exchangeable} exchangeable random objects of set system-type (hierarchies/interval hypergraphs) have been introduced as random sequences of finite growing exchangeable structures satisfying sampling consistency, that is in canonical extension. 
		\end{remark}
	
		\subsection{Pointwise convergence, $U$-statistics and the independence property}\label{sec:pointwise}
		
		Let $D$ be a BDS with $\SYM(D)\neq\emptyset$ and $D:\CINJo\rightarrow\BOREL$ be the canonical extension of $D$. Propositions~\ref{prop:exchangeablelaws} and \ref{prop:translation} show that studying $\SYM(D)$ falls into the framework (\ref{eq:groupaction}) presented in the introduction: a measurable group action $\bS_{\bN}\times\cS\rightarrow\cS, (\pi,x)\mapsto \pi x$ is derived by defining 
		$$\cS = D_{\bN}~~~\text{and}~~~\pi x = D[\pi^{-1}](x),$$
		and $\SYM(D)$ can be identified with $\SYM(D;\bN)$, that is with laws of $\cS=D_{\bN}$-valued random variables $X$ with $\pi X\ed X$ for all $\pi\in\bS_{\bN}$. Further, $X$ is exchangeable already iff $\pi X\ed X$ for all $\pi\in\bSi\subseteq\bS_{\bN}$.\\
		Ergodic theory results become directly applicable: let $\cI$ be the $\sigma$-field of measurable subsets $M\subseteq \cS = D_{\bN}$ with $D[\pi](M) = M$ for all $\pi\in\bSi$. An exchangeable $D_{\bN}$-valued $X$ is called ergodic iff $\bP[X\in M]\in\{0,1\}$ for all $M\in\cI$. Let $\eSYM(D;\bN)\subseteq \SYM(D;\bN)$ be the set ergodic exchangeable laws, which is non-empty measurable. Ergodic decomposition, Theorem~A1.4 in \cite{kallenberg1997foundations}, gives that the following two maps are bijections inverse to each other:
		\begin{align*}
			\sP\big(\eSYM(D;\bN)\big)\longrightarrow\SYM(D;\bN)&,~~~~\Xi\mapsto \int_{\eSYM(D;\bN)}\mu(\cdot)d\Xi(\mu)\\
			\SYM(D;\bN)\longrightarrow\sP\big(\eSYM(D;\bN)\big)&,~~~~\cL(X)~\mapsto~\cL\big(~\bP[X\in\cdot|X^{-1}(\cI)]~\big).
		\end{align*}
		The abstract de~Finetti theorem, Theorem~\ref{thm:definetti}, follows from this by identifying $\eSYM(D;\bN)$ with exchangeable laws having the independence property, that is with $\eSYM(D)$. This is shown in Theorem~\ref{thm:independence} below.
		
		\begin{remark}[Convex decomposition]
			Let $A$ be countable infinite and consider collections $(\mu_a)_{a\in\finA}$ with $\mu\in\SYM(D)$. A strict partial order on $\finA$ is given by comparing sets by cardinality, that is $b<a$ iff $|b|<|a|$. This strict partial order is directed to the right and countable at infinity. For $b<a$ let $T_{ba}:b\rightarrow a$ be a uniform random injection and define the probability kernel $p_{ba}:D_a\rightarrow \sP(D_b), x\mapsto \cL(D[T_{ba}](x))$. By combinatorial arguments, a collection $(\mu_a)_{a\in\finA}$ comes from some $\mu\in\SYM(D)$ iff $\mu_b = p_{ba}\mu_a$ for all $b<a$. Modulo topological assumptions: Proposition~1.1 in Chapter~IV of \cite{lauritzen1988extremal} gives a simplex decomposition for such collections $(\mu_a)_{a\in\finA}$.
		\end{remark}
	
		The proof for characterizing ergodicity via independence heavily relies on the following, for $\pi\in\bSi$ write $|\pi|\leq n$ iff $\pi(i)=i$ for all $i>n$:
		
		\begin{theoremB}[Pointwise convergence, Theorem 1.2 in \cite{lindenstrauss2001pointwise} applied to $\bSi$]
			For every $\cL(X)\in\SYM(D;\bN)$ and measurable $f:D_{\bN}\rightarrow\bR$ with $\bE[|f(X)|]<\infty$ it holds that 
			\begin{equation*}
			\frac{1}{n!}\sum_{\pi\in\bSi,|\pi|\leq n}f\circ D[\pi](X)~~\overset{n\rightarrow\infty}{\longrightarrow}~~\bE[f(X)|X^{-1}(\cI)]~~\text{almost surely.}
			\end{equation*}
		\end{theoremB}
		
		Theorem~B is applied to functions $f$ obtained from \emph{kernel functions} $g:D_{[k]}\rightarrow\bR, k\geq 0$ via $f = g\circ D[\iota_{[k],\bN}]$. For $n\geq k$ let 
		\begin{equation*}
			\avg(g,n):D_{[n]}\rightarrow \bR,~~~~\avg(g,n) = \frac{(n-k)!}{n!}\sum_{\tau:[k]\rightarrow[n]~\text{injective}}g\circ D[\tau].
		\end{equation*}
		For a uniform random injection $T_{k,n}:[k]\rightarrow[n]$ it is
		$$\avg(g,n)(x_n) = \bE\big[g\circ D[T_{k,n}](x_n)\big]~~~\text{for every}~x_n\in D_{[n]}$$
		and basic combinatorial arguments together with functorality of $D$ gives for every random $D$-measurement $X = (X_{a})_{a\in\finN}$ and $n\geq k$
		\begin{equation*}
			\avg(g,n)(X_{[n]}) = \frac{1}{n!}\sum_{\pi\in\bSi,|\pi|\leq n}f\circ D[\pi](X).
		\end{equation*}
		
		Theorem~B directly yields the following
		
		\begin{corollary}\label{cor:conv}
			For an exchangeable $D$-measurement $X=(\Xa)_{a\in\finN}$ and measurable $g:D_{[k]}\rightarrow\bR, k\geq 0$ with $\bE[g(X_{[k]})|]<\infty$ it is
			\begin{equation*}
			\avg(g,n)(X_{[n]})\overset{n\rightarrow\infty}{\longrightarrow} \bE[g(X_{[k]})|X^{-1}(\cI)]~~\text{almost surely.}
			\end{equation*}
		\end{corollary}
	
		\begin{remark}	
			An alternative approach to Corollary~\ref{cor:conv} is by backwards martingale convergence; however, the proof using pointwise convergence theorem is much more direct.
		\end{remark}
	
		Basic measure theoretic considerations give that $X = (\Xa)_a$ is ergodic iff for every $k\geq 0$ and bounded measurable kernel $g:D_{[k]}\rightarrow \bR$ it is $\bE[g(X_{[k]})|X^{-1}(\cI)]\as \bE[g(X_{[k]})]$ a.s. constant, which is equivalent to the variance of $\bE[g(X_{[k]})|X^{-1}(\cI)]$ being zero. For every exchangeable $X$, not necessarily ergodic, and every square integrable kernel $g:D_{[k]}\rightarrow\bR$, that is $\bE[g^2(X_{[k]})]<\infty$, simple calculations using exchangeability, sampling consistency and functorality of $D$ give for every $n\geq k$
		
		\begin{equation}\label{eq:variance}
			\bVa\big(\avg(g,n)(X_{[n]})\big) = \frac{(n-k)!}{n!}\sum_{a\in\binom{[n]}{k}}\sum_{\pi:[k]\rightarrow a~\text{bij.}}\Cov\Big(g(X_{[k]}), g\circ D[\pi](X_a)\Big).
		\end{equation}
		
		This is used to prove:
		
		\begin{theorem}\label{thm:independence}
			Let $X=(\Xa)_{a\in\finN}$ be an exchangeable $D$-measurement. Equivalent are:
			\begin{enumerate}
				\item[(i)] $X$ is ergodic, 
				\item[(ii)] $X$ has the independence property: $\Xa, \Xb$ are stochastically independent for all $a,b\in\finN$ with $a\cap b=\emptyset$,
				\item[(iii)] for every countable set $\cG\subseteq\bigcup_{k\geq 0}\bR^{D_{[k]}}$ of bounded measurable functions there exists a \emph{deterministic} sequence $(x_n)_{n\in\bN}$ with $x_n\in D_{[m_n]}$ such that $m_n\rightarrow\infty$ and for every $g\in\cG, g:D_{[k]}\rightarrow\bR$
				$$\bE[g(X_{[k]})] = \lim\limits_{n\rightarrow\infty}\avg(g, m_n)(x_n).$$ 
			\end{enumerate} 
		\end{theorem}
		\begin{proof}
			(i)$\Rightarrow$(iii)$\Rightarrow$(ii)$\Rightarrow$(i) is shown.\\
			(i)$\Rightarrow$(iii). By Corollary~\ref{cor:conv} $\avg(g,n)(X_{[n]})\rightarrow \bE[g(X_{[k]})]$ a.s. for every $g\in\cG$ defined on $D_{[k]}$. Since $\cG$ is countable the convergence almost surely holds simultaneously over $\cG$, take $x_n = X_n(\omega)$ for some $\omega$ from the corresponding probability-one event; $m_n=n$ in this case.\\
			(iii)$\Rightarrow$(ii). Let $a, b\in\finN$ with $a\cap b = \emptyset$ and $f:\Da\rightarrow\bR, g:\Db\rightarrow\bR$ be bounded measurable, so $\bE[f(X_a)g(X_b)] = \bE[f(X_a)]\bE[g(X_b)]$ is to be shown. Let $k\in\bN$ be such that $a\cup b\subseteq [k]$ and $h:D_{[k]}\rightarrow\bR$ be 
			$$h = \big(f\circ D[\iota_{a,[k]}]\big)\cdot\big(g\circ D[\iota_{b,[k]}]\big).$$ 
			Applying (iii) to the three-element set $\cG = \{f\circ D[\iota_{a,[k]}], g\circ D[\iota_{b,[k]}], h\}$ gives a deterministic sequence $(x_n)$ with $x_n\in D_{[m_n]}$, $m_n\rightarrow\infty$, such that for a uniform random injection $T_{k,m_n}:[k]\rightarrow[m_n]$ it holds that 
			\begin{align*}
			\bE[f(X_a)] = \bE\big[f\big(D[\iota_{a,[k]}](X_{[k]})\big)\big] &= \lim_{n\rightarrow\infty}\bE\Big[f\Big(D[T_{k,m_n}\circ\iota_{a,[k]}](x_n)\Big)\Big],\\
			\bE[g(X_b)] = \bE\big[g\big(D[\iota_{b,[k]}](X_{[k]})\big)\big] &= \lim_{n\rightarrow\infty}\bE\Big[g\Big(D[T_{k,m_n}\circ\iota_{b,[k]}](x_n)\Big)\Big],\\
			\bE[f(X_a)g(X_b)] &= \lim_{n\rightarrow\infty}\bE\Big[f\Big(D[T_{k,m_n}\circ\iota_{a,[k]}](x_n)\Big)g\Big(D[T_{k,m_n}\circ\iota_{b,[k]}](x_n)\Big)\Big].
			\end{align*}
			Let $T'_{k,m_n}$ be another random uniform injection $[k]\rightarrow[m_n]$ independent from $T_{k,m_n}$ and let $A_{k,m_n} = \{T_{k,m_n}(a)\cap T'_{k,m_n}(b) = \emptyset\}$. Elementary combinatorial calculations show that $\bP[A_{k,m_n}]\rightarrow 1$ as $n\rightarrow\infty$ and that for every fixed $n$ with $m_n\geq k$ the joint distribution of $(T_{k,m_n}\circ\iota_{a,[k]}, T_{k,m_n}\circ\iota_{b,[k]})$ is the same as that of $(T_{k,m_n}\circ\iota_{a,[k]}, T'_{k,m_n}\circ\iota_{b,[k]})$ conditioned on $A_{k,m_n}$. This gives
			\begin{align*}
			\bE[f(X_a)g(X_b)] &= \lim_{n\rightarrow\infty}\bE\Big[f\Big(D[T_{k,m_n}\circ\iota_{a,[k]}](x_n)\Big)g\Big(D[T_{k,m_n}\circ\iota_{b,[k]}](x_n)\Big)\Big]\\
			&= \lim_{n\rightarrow\infty}\bE\Big[f\Big(D[T_{k,m_n}\circ\iota_{a,[k]}](x_n)\Big)g\Big(D[T'_{k,m_n}\circ\iota_{b,[k]}](x_n)\Big)~\Big|~A_{k,m_n}~\Big]\\
			&= \lim_{n\rightarrow\infty}\bE\Big[f\Big(D[T_{k,m_n}\circ\iota_{a,[k]}](x_n)\Big)\Big]\bE\Big[g\Big(D[T'_{k,m_n}\circ\iota_{b,[k]}](x_n)\Big)\Big]\\
			&= \bE[f(X_a)]\bE[g(X_b)].
			\end{align*}
			(ii)$\Rightarrow$(i). Let $g:D_{[k]}\rightarrow\bR, k\geq 0$ be bounded measurable, it is shown that the variance of $\bE[g(X_{[k]})|X^{-1}(\cI)]$ is zero. By pointwise and dominated convergence
			\begin{equation*}
			\bVa(\bE[g(X_{[k]})|X^{-1}(\cI)]) = \lim\limits_{n\rightarrow\infty}\bVa\Big(\avg(g,n)\big(X_{[n]}\big)\Big).
			\end{equation*} 
			By (\ref{eq:variance}) the variance of $\avg(g,n)(X_{[n]})$ depends on covariances $\Cov(g(X_{[k]}), g\circ D[\pi](X_a))$ with $a\in\binom{[n]}{k}$ and $\pi:[k]\rightarrow a$ bijective. By assumption (ii) such a covariance is zero if $[k]\cap a = \emptyset$. For $n\geq 2k$ there are $\binom{n-k}{k}$ of such $a\in\binom{[n]}{k}$, bounding $|g|\leq C$ gives 
			\begin{equation*}
				\bVa\Big(\avg(g,n)\big(X_{[n]}\big)\Big) \leq \frac{k!(n-k)!}{n!}\big[\binom{n}{k} - \binom{n-k}{k}\big]C^2 = \big[1 - \binom{n-k}{k}/\binom{n}{k}\big]C^2,
			\end{equation*}
			for fixed $k$ the upper bound goes to zero as $n\rightarrow\infty$.
		\end{proof}

		\begin{remark}[Asymptotic of $U$-statistics]\label{rem:asymptotic-u-statistics}
			Let $g:D_{[k]}\rightarrow\bR$ be a symmetric measurable kernel, that is $g\circ D[\pi] = g$ for every bijection $\pi:[k]\rightarrow[k]$. In this case $g$ can be extended to have domain $D_a$ for every $a$ with $|a| = k$. If $X = (\Xa)_{a\in\finN}$ is exchangeable with $\bE[g^2(X_{[k]})]<\infty$ the variance formula (\ref{eq:variance}) can be further reduced: for $n\geq 2k$ it is		
			\begin{equation*}
				\bVa\big(\avg(g,n)(X_{[n]})\big) = \sum_{l=0}^k \frac{\binom{k}{l}\binom{n-k}{k-l}}{\binom{n}{k}}c_{l}~~~\text{with}~~~c_{l} = \Cov(g(X_{[k]}), g(X_{[l]+\{k+1,\dots,2k-l\}})).
			\end{equation*}
			In case $D=\SEQ(\cX)$ this follows from classical $U$-statistics theory, see \cite{korolyuk2013theory}, and in this case $c_l$ has a representation as the variance of a conditional expectation, which directly gives $c_l\geq 0$. This also holds for general $D$: by Corollary~\ref{cor:frt-rvs} there is a $\unif(R)$-a.s natural transformation $\eta:R\rightarrow D$ representing the law of $X$, for every $a\in\finN$ let $X_a := \eta_a((U_e)_{e\subseteq a})$ with $U_e, e\in\finN$ iid $\sim\unif[0,1]$. In this special construction of $X$, the same ideas as for the sequential case give
			\begin{equation*}
				c_l = \Cov\big(g(X_{[k]}), g(X_{[l]+\{k+1,\dots,2k-l\}})\big) = \bVa\Big(\bE\big[g(X_{[k]})\big|(U_e)_{e\subseteq[l]}\big]\Big)~\geq 0.
			\end{equation*}
			In case $X$ is ergodic it is $c_0 = 0$, which follows directly from the independence property and is also reflected in the previous formula noting that for ergodic laws no randomness from $U_{\emptyset}$ is needed in functional representations, see Remark~\ref{rem:ergodicmod}.\\
			Theorem~17 from \cite{austern2018limit} can be applied to consider the asymptotic distribution of $\avg(g,n)(X_{[n]})$: in case $X$ is ergodic it is
			\begin{equation*}
				\sqrt{n}~\Big[~\avg(g,n)(X_{[n]})~-~\bE[g(X_{[k]})]~\Big]~~\overset{n\rightarrow\infty}{\longrightarrow}~~\Norm(0,\sigma^2)~~\text{in distribution},
			\end{equation*}
			where the asymptotic variance $\sigma^2\geq 0$ can be found as 
			\begin{equation*}
				\sigma^2 = \lim_{n\rightarrow\infty} n\sum_{l=1}^k \frac{\binom{k}{l}\binom{n-k}{k-l}}{\binom{n}{k}}c_{l} = k^2 c_1,
			\end{equation*}
			with $c_1 = \Cov\big(g(X_{\{1,\dots,k\}}), g(X_{\{1,k+1,k+2,\dots,2k-1\}})\big) = \bVa\big(\bE\big[g(X_{[k]})\big|U_{\{1\}}\big]\big)$.
		\end{remark}

	\subsection{Limits of combinatorial structures}\label{sec:combinatorial}
			
			Let $D:\INJo\rightarrow\FIN_+$ be a combinatorial data structure. For simplicity assume the finite set $a$ can be recovered from $x\in D_a$, so one can define $|x|:=|a|$. If this is not the case replace $D$ by the isomorphic BDS $\tilde D$ defined as $\tilde D_a = \{(a,x)|x\in D_a\}$ and $\tilde D[\tau]((a,x)) = (b, D[\tau](x))$.\\
			For $x\in D_b, y\in D_a$ with $|x|\leq|y|$ let
			\begin{equation*}
				\density(x,y) = \frac{(|a|-|b|)!}{|a|!}\Big|\big\{\tau:b\rightarrow a~\text{injective}~:~D[\tau](y) = x~\big\}\Big|,
			\end{equation*}
			that is $\density(x,y) = \bP[D[T_{ba}](y) = x]$ for a uniform random injection $T_{ba}:b\rightarrow a$. The value $\density(x,y)\in[0,1]$ is interpreted as the (combinatorial) density of the smaller structure $x$ within the larger structure $y$.
			
			\begin{definition}[Limits of combinatorial structures]
				A sequence $\bbx = (x_n)_{n\geq 1}$ with $x_n \in D_{a_n}$ is said to be convergent iff $|x_n| = |a_n|\rightarrow\infty$ and for every $x\in D_b$ the limit
				$$\lim_{n\rightarrow\infty}\density(x, x_n) \in [0,1]$$
				exists. In this case, the limit of $\bbx$ is the rule that maps $x\in D_b$ to $\lim_n\density(x,x_n)\in [0,1]$.
			\end{definition}
		
			The following is an application of Theorem~\ref{thm:independence} for the combinatorial case, technical details are in the Appendix.
			
			\begin{theorem}\label{thm:correspondence-limits}
				Limits of convergent sequences coincide with $\eSYM(D)$: for every convergent sequence $\bbx = (x_n)_{n\geq 1}$ there is exactly one rule $\mu\in\eSYM(D)$ such that 
				\begin{equation}\label{eq:density}
					\mu_b(\{x\}) = \lim_{n\rightarrow\infty}\density(x,x_n)~~~\text{for every $b$ and}~x\in D_b
				\end{equation} 
				and conversely, every $\mu\in\eSYM(D)$ is of this form for some convergent sequence $\bbx$.
			\end{theorem}
		
			\begin{remark}[$\SYM(D)$ is a Bauer simplex for combinatorial data structures]
				Let $\COMP$ be the category of continuous maps between compact metrizable topological spaces. Every finite discrete space is compact metrizable and every map between finite discrete spaces  continuous, thus combinatorial data structures can be seen as functors $D:\INJo\rightarrow\COMP$. In this case extensions to $\CINJo\supset\INJo$ always exists and can be seen as functors $D:\CINJo\rightarrow\COMP$; the derived group action $\bSi\times D_{\bN}\rightarrow D_{\bN}$ is a topological group action on compact metrizable space, which are studied in ergodic theory, see \cite{glasner2003ergodic}. For any compact metrizable space $\cS$ and topological group action $\bSi\times\cS\rightarrow\cS$ the space of invariant laws, denoted with $\sPs(\cS)$, has the structure of a Choquet simplex. Our discussion shows that $\sPs(D_{\bN})$ has a closed set of extremal points -- either check that the independence property is closed or argue that extremal points coincide with a Martin boundary -- in particular, $\sPs(D_{\bN})$ is a Bauer simplex. This is not obvious from general theory: with $\cS = \{0,1\}^{\bSi}$ and $\pi (x_{\sigma})_{\sigma\in\bSi} = (x_{\pi^{-1}\sigma})_{\sigma\in\bSi}$ it is known that $\sPs(\{0,1\}^{\bSi})$ is the Poulsen simplex by the fact that $\bSi$ is amenable and countable infinite and thus does not have Kazhdan's property (T), see Theorem~13.15 in \cite{glasner2003ergodic}.
			\end{remark}
			
		\section{Weak FRT}
		
		To recall, the indexing system $I = \square^*_{\neq}$ is defined by $I_b = b^*_{\neq}$, that is the set of all tuples $\bbi = (i_1,\dots,i_k)\in b^k$ with $i_j\neq i_{j'}, j\neq j'$, and for an injection $\tau:b\rightarrow a$ it is $I[\tau](\bbi)\in a^*_{\neq}$ defined by 
		$$I[\tau](\bbi) = \vec{\tau}(\bbi) = \big(\tau i_1,\dots,\tau i_k\big).$$
		For a finite set $a$ and index $\bbi=(i_1,\dots,i_k)\in a^*_{\neq} = I_a$ let $\tau_{\bbi,a}:[k]\rightarrow a, j\mapsto i_j$, which is injective.\\
		
		The following result characterizes natural transformations $\eta:D\rightarrow\ARRAY(\cX,\square^*_{\neq})$, where $D$ is an arbitrary Borel data structure and $\cX$ an arbitrary Borel space. It arises as a special case of Theorem~\ref{thm:nat-array} later, but because $\square^*_{\neq}$ is of great importance to the general theory and the proof is especially tractable in this case, it is presented here separately. Note that the components $\eta_a$ of a natural transformation $\eta:D\rightarrow\ARRAY(\cX,I)$ are measurable maps $\eta_a:D_a\rightarrow \cX^{I_a}$ and thus have \emph{inner} component functions $\eta_{a,\bbi}:D_a\rightarrow\cX, \bbi\in I_a$ such that $\eta_a(\cdot) = (\eta_{a,\bbi}(\cdot))_{\bbi\in I_a}$.
		
		\begin{proposition}\label{thm:nat-trans-into-end}
			There is a one-to-one correspondence between 
			\begin{itemize}
				\item Natural Transformations $\eta:D\rightarrow\ARRAY(\cX,\square^*_{\neq})$
				\item Sequences of measurable maps $(f_k)_{k\geq 0}$ with $f_k:D_{[k]}\rightarrow\cX$
			\end{itemize}
			given by 
			\begin{itemize}
				\item $\eta \mapsto (f_k)_{k\geq 0}$ with $f_k = \eta_{[k],(1,\dots,k)}$
				\item $(f_k)_{k\geq 0}\mapsto \eta$ with $\eta_{a}(\cdot) = \big(f_k\circ D[\tau_{\bbi, a}](\cdot)\big)_{\bbi=(i_1,\dots,i_k)\in a^*_{\neq}}$.
			\end{itemize}
		\end{proposition}
		\begin{proof}
			Let $E=\ARRAY(\cX,\square^*_{\neq})$. For every rule $\eta:D\rightarrow E$ that maps a finite set $a$ to a measurable map $\eta_a:D_a\rightarrow E_a = \cX^{I_a}$ consider the "inner" component functions $\eta_{a,\bbi}:D_a\rightarrow\cX$ which are measurable and satisfy $\mu_a(\cdot) = (\mu_{a,\bbi}(\cdot))_{\bbi\in I_a}$. It is easily checked that $\eta$ is a natural transformation iff the inner components satisfy for every injection $\tau:b\rightarrow a$ and index $\bbi\in I_b$
			$$\eta_{b,\bbi}\circ D[\tau] = \eta_{a,I[\tau]\bbi}.$$
			Suppose $\eta$ is a natural transformation and let $f_k = \eta_{[k],(1,\dots,k)}$. For every $\bbi = (i_1,\dots,i_k)\in a^*_{\neq}$ it holds $I[\tau_{\bbi,a}]((1,\dots,k)) = \bbi$ and hence
			$$f_k\circ D[\tau_{\bbi,a}] = \eta_{[k],(1,\dots,k)}\circ D[\tau_{\bbi,a}] = \eta_{a,\bbi},$$
			that is: $\eta$ is determined by $(f_k)_{k\geq 0}$, hence the construction $\eta \mapsto (f_k)_{k\geq 0}$ is injective.\\
			On the other hand, let $(f_k)_{k\geq 0}$ be an arbitrary sequence of measurable functions $f_k:D_{[k]}\rightarrow \cX$ and for $\bbi=(i_1,\dots,i_k)\in a^*_{\neq}$ define the inner component $\eta_{a,\bbi}:D_a\rightarrow \cX$ by $\eta_{a,\bbi} = f_k\circ D[\tau_{\bbi,a}]$. Let $\tau:b\rightarrow a$ be an injection. For $\bbi = (i_1,\dots,i_k)\in b^*_{\neq}$ it holds
			$$\tau\circ \tau_{\bbi,b} = \tau_{I[\tau]\bbi, a}$$
			and hence
			\begin{equation*}
				\eta_{b,\bbi}\circ D[\tau] = f_k\circ D[\tau_{\bbi,b}]\circ D[\tau] = f_k\circ D[\tau\circ \tau_{\bbi,b}] = f_k\circ D[\tau_{I[\tau]\bbi, a}] = \eta_{a,I[\tau]\bbi},
			\end{equation*}
			so the construction $(f_k)_k\mapsto \eta$ defines a natural transformation. It is obvious that the constructions $\eta\mapsto (f_k)_k$ and $(f_k)_k\mapsto \eta$ are inverse to each other.  
		\end{proof}
	
		A first application of Proposition~\ref{thm:nat-trans-into-end} is in proving that Theorem~A has an equivalent formulation using natural transformations.
	
		\begin{proof}[Proof of Theorem~\ref{thm:equivalence}]
			Let $D = \ARRAY(\cX,\square^*_{\neq})$. The following are shown to be equivalent:
			\begin{itemize}
				\item[(i)] Theorem~A
				\item[(ii)] For every $\mu \in \SYM(D)$ exists a natural transformation $\eta:R\rightarrow D$ with $\mu = \ur\circ\eta^{-1}$. 
			\end{itemize}
			\underline{(i)$\Rightarrow$(ii)}. Let $\mu\in\SYM(D)$. By Kolmogorov consistency there exists a $\cX$-valued stochastic process $X = (X_{\bbi})_{\bbi\in \bN^*_{\neq}}$ such that for every finite set $a\in\finN$ it is $(X_{\bbi})_{\bbi\in a^*_{\neq}} \sim \mu_a$. Let $U_a, a\in\finN$ be iid $\sim\unif[0,1]$ and for every $a$ let $Y_a = (U_{a'})_{a'\in 2^a}$. By Theorem~A there is a measurable function $f:\cup_k[0,1]^{2^{[k]}}\rightarrow\cX$ such that $X \ed \big(f\big((U_{\pi_{\bbi}(e)})_{e\in 2^{[k]}}\big)\big)_{\bbi=(i_1,\dots,i_k)\in\bN^*_{\neq}}$, where for $\bbi = (i_1,\dots,i_k)$ it is $\pi_{\bbi}:[k]\rightarrow\{i_1,\dots,i_k\}, j\mapsto i_j$ bijective and it holds
			$$R[\tau_{\bbi,a}](Y_a) = (U_{\pi_{\bbi}(e)})_{e\in 2^{[k]}}.$$
			For every $k\geq 0$ let $f_k:[0,1]^{2^{[k]}}\rightarrow\cX$ be the restriction of $f$ to $[0,1]^{2^{[k]}}$. The functions $(f_k)_k$ give a natural transformation $\eta:R\rightarrow D$ by the construction in Proposition~\ref{thm:nat-trans-into-end}. For every finite subset $a\in \finN$ it then holds that 
			\begin{align*}
				\mu_a \sim (X_{\bbi})_{\bbi\in a^*_{\neq}} &\ed \big(f\big((U_{\pi_{\bbi}(e)})_{e\in 2^{[k]}}\big)\big)_{\bbi=(i_1,\dots,i_k)\in a^*_{\neq}}\\
					  &= \big(f_k\circ R[\tau_{\bbi, a}](Y_a)\big)_{\bbi=(i_1,\dots,i_k)\in a^*_{\neq}} = \eta_a(Y_a) \sim \ur_a\circ\eta_a^{-1},
			\end{align*}
			that is $\mu = \ur\circ\eta^{-1}$.\\
			\underline{(ii)$\Rightarrow$(i)}. Let $X = (X_{\bbi})_{\bbi\in\bN^*_{\neq}}$ be exchangeable $\cX$-valued. For any injection $\tau:a\rightarrow\bN$ define $\tilde\tau:a^*_{\neq}\rightarrow \bN^*_{\neq}, (i_1,\dots,i_k)\mapsto (\tau i_1,\dots,\tau i_k)$. The law of $X\circ \tilde\tau$ does not depend on a concrete choice of $\tau$ and hence allows to define
			$$\mu_a = \cL\big(X\circ \tilde\tau\big) \in \sP(\cX^{a^*_{\neq}}) = \sP(D_a),$$
			independent on the choice of $\tau$. It is easy to check that this defines an exchangeable law $\mu\in \SYM(D)$. By (ii) there is a natural transformation $\eta:R\rightarrow D$ such that $\mu_a = \ur_a\circ\eta^{-1}_a$ for every finite $a$. By Proposition~\ref{thm:nat-trans-into-end} there is a sequence of measurable functions $(f_k)_{k\geq 0}$ representing $\eta$ which glued together yield $f:\cup_k[0,1]^{2^{[k]}}\rightarrow\cX$. Let $U_a, a\in\finN$ be iid $\sim\unif[0,1]$ and $Y_a = (U_{a'})_{a'\subseteq a}$. It is $Y_a\sim\ur_a$ and hence 
			\begin{align*}
				(X_{\bbi})_{\bbi\in a^*_{\neq}} \sim \mu_a = \ur_a\circ\eta_a^{-1}  \ed \eta_a(Y_a) = \big(f\big((U_{\pi_{\bbi}(e)})_{e\in 2^{[k]}}\big)\big)_{\bbi=(i_1,\dots,i_k)\in a^*_{\neq}}.
			\end{align*}
	        Note that for every $b\subseteq a$ by naturality $D[\iota_{b,a}](\eta_a(Y_a)) = \eta_b(R[\iota_{b,a}](Y_a)) = \eta_b(Y_b)$ and that $\bN^*_{\neq} = \cup_{n\geq 0}[n]^*_{\neq}$. By Kolmogorov consistency the distributional equations $(X_{\bbi})_{\bbi\in a^*_{\neq}}\ed \eta_a(Y_a)$ holding for every finite set $a\subseteq \bN$ can thus be lifted to the whole process:
	        $$(X_{\bbi})_{\bbi\in \bN^*_{\neq}}\ed \big(f\big((U_{\pi_{\bbi}(e)})_{e\in 2^{[k]}}\big)\big)_{\bbi=(i_1,\dots,i_k)\in \bN^*_{\neq}},$$
	        giving (i).
		\end{proof}
		
		Theorems~A~+~\ref{thm:equivalence} give:
		
		\begin{corollary}\label{thm:frt-kallenberg}
			Let $D = \ARRAY(\cX,\square^*_{\neq})$. For every $\mu \in \SYM(D)$ exists a natural transformation $\eta:R\rightarrow D$ with $\mu = \ur\circ\eta^{-1}$. 
		\end{corollary}
	
		A second application of Proposition~\ref{thm:nat-trans-into-end} is:
	
		\begin{theorem}\label{thm:embedding}
			Every Borel data structure $D$ can be naturally embedded in $\ARRAY([0,1],\square^*_{\neq})$. 
		\end{theorem}
		
		\begin{proof}
			A natural transformation $\phi:D\rightarrow\ARRAY([0,1],\square^*_{\neq})$ is constructed such that every component $\phi_a:D_a\rightarrow [0,1]^{a^*_{\neq}}$ is injective.\\
			For every $k\geq 0$ it is $D_{[k]}$ a Borel space, hence there exists a measurable injection 
			$$f_k:D_{[k]}\rightarrow [0,1].$$ 
			By Proposition~\ref{thm:nat-trans-into-end} the rule $\phi:D\rightarrow\ARRAY([0,1],\square^*_{\neq})$ having components $\phi_a = (\phi_{a,\bbi})_{\bbi\in a^*_{\neq}}$ with $\phi_{a,\bbi} = f_k\circ D[\tau_{\bbi,a}], \bbi = (i_1,\dots,i_k)\in a^*_{\neq}$ is a natural transformation. For every $a$ a tuple $\bbi = (i_1,\dots,i_k)\in a^*_{\neq}$ having maximal length $k=|a|$ is an enumeration of all elements of $a$ and hence $\tau_{\bbi,a}:[k]\rightarrow a, j\mapsto i_j$ is a bijection, so $D[\tau_{\bbi,a}]$ is a bijection and hence $\phi_{a,\bbi}$ is an injection (as a composition of injection $f_k$ and bijection $D[\tau_{\bbi,a}]$). Now $\phi_a = (\phi_{a,\bbi})_{\bbi\in a^*_{\neq}}$ is injective as already some of its component functions are.
		\end{proof}
		
		\begin{remark}
			In case $\depth(D) = k <\infty$ one can construct an embedding $\phi:D\rightarrow \ARRAY([0,1],\square^k_{\neq})$ in an analog way, thus $D$ is naturally isomorphic to a sub-data structure of $\ARRAY([0,1],\square^k_{\neq})$.
		\end{remark}
		
		The embedding constructed in the proof is highly redundant and unpractical for applications: every entry in the array $\phi_a(x)\in[0,1]^{a^*_{\neq}}$ that is indexed by a full-length tuple $\bbi\in a^*_{\neq}$, of which there are $|a|!$ many, contains all information about $x$ and thus also the information about all other entries. But this information can in general \textbf{not} be recovered using a true natural transformation defined on the whole of $\ARRAY([0,1],\square^*_{\neq})$; for some BDS $D$ there do not even exist a single true natural transformation $\ARRAY([0,1],\square^*_{\neq})\rightarrow D$. 
		
		\begin{lemma}\label{lemma:chaining}
			Let $D, E, F$ be Borel data structures, $\mu\in\SYM(D)$, $\eta:D\rightarrow E$ a $\mu$-a.s. natural transformation and $\phi:E\rightarrow F$ a rule that maps every finite set $a$ to a measurable function $\phi_a:E_a\rightarrow F_a$. Then the following are equivalent:
			\begin{itemize}
				\item[(i)] $\phi$ is a $\mu\circ\eta^{-1}$-a.s. natural transformation, 
				\item[(ii)] $\phi\circ\eta$ is a $\mu$-a.s. natural transformation.
			\end{itemize}
		\end{lemma}
		\begin{proof}
			Let $\tau:b\rightarrow a$ be an injection, $X_a\sim \mu_a$ and $Y_a = \eta_a(X_a)$, that is $Y_a\sim (\mu\circ\eta^{-1})_a = \mu_a\circ\eta^{-1}_a$.
			(i)$\Rightarrow$(ii): It is $F[\tau]\circ\phi_a\circ\eta_a(X_a) = F[\tau]\circ\phi_a(Y_a) \as \phi_b\circ E[\tau](Y_a)$ because $\phi$ is $\mu\circ\eta^{-1}$-a.s. natural transformation by assumption (i). It is $\phi_b\circ E[\tau](Y_a) = \phi_b\circ E[\tau]\circ \eta_a(X_a) \as \phi_b\circ \eta_b\circ D[\tau](X_a)$ because $\eta$ is $\mu$-a.s. natural transformation by assumption. Hence $F[\tau]\circ\phi_a\circ\eta_a(X_a) \as \phi_b\circ \eta_b\circ D[\tau](X_a)$, so (ii).\\
			(ii)$\Rightarrow$(i). It is $F[\tau]\circ \phi_a(Y_a) = F[\tau]\circ\phi_a\circ \eta_a(X_a) \as \phi_b\circ \eta_b\circ D[\tau](X_a)$ since $\phi\circ\eta$ is $\mu$-a.s. natural transformation by (ii). It is $\phi_b\circ \eta_b\circ D[\tau](X_a) \as \phi_b\circ E[\tau]\circ \eta_a(X_a) = \phi_b\circ E[\tau](Y_a)$ since $\eta$ is a $\mu$-a.s. natural transformation, hence (i). 
		\end{proof}
		
		\begin{proposition}\label{prop:inverse}
			Let $\phi:D\rightarrow E$ be an embedding. Then there exists a rule $\theta:E\rightarrow D$ that sends every finite set $a$ to measurable map $\theta_a:E_a\rightarrow D_a$ such that
			\begin{itemize}
				\item $\theta\circ \phi = \id_D$, that is $\theta_a\circ\phi_a = \id_{D_a}$ for every $a$,
				\item $\theta$ is a $\mu\circ\phi^{-1}$-a.s natural transformation for every $\mu\in\SYM(D)$.
			\end{itemize}
		\end{proposition}
		\begin{proof}
			Every component $\phi_a:D_a\rightarrow E_a$ is a measurable injection between Borel spaces, hence has a measurable left-inverse. Applying the global axiom of choice gives a rule $\theta:E\rightarrow D$ that picks measurable left-inverses, so $\theta\circ\phi = \id_D$. Since both $\phi$ and $\theta\circ\phi = \id_D$ are natural transformations, they are also $\mu$-a.s. natural transformations for every $\mu\in\SYM(D)$. By Lemma~\ref{lemma:chaining} $\theta$ is a $\mu\circ\phi^{-1}$-a.s. natural transformation.
		\end{proof}
	
		Given the previous results it is now easy to prove the weak FRT (without depth):
	
		\begin{proof}[Proof of Theorem~\ref{thm:frt-weak}]
			It is shown that for every BDS $D$ and $\mu\in\SYM(D)$ there exists a $\ur$-a.s. natural transformation $\eta:R\rightarrow D$ with $\mu = \ur\circ\eta^{-1}$. 
			Let $E = \ARRAY([0,1],\square^*_{\neq})$ and $\phi:D\rightarrow E$ be an embedding, which exists due to Theorem~\ref{thm:embedding}. It is $\mu\circ\phi^{-1}\in\SYM(E)$ and by Corollary~\ref{thm:frt-kallenberg} there is a natural transformation $\psi:R\rightarrow E$ such that $\mu\circ\phi^{-1} = \ur\circ\psi^{-1}$.\\
			By Proposition~\ref{prop:inverse} there is a $\mu\circ\phi^{-1}$-a.s. natural transformation $\theta:E\rightarrow D$ such that $\theta\circ\phi = \id_D$. Let $\eta = \theta\circ\psi$. It is $\psi:R\rightarrow E$ a natural transformation and $\theta:E\rightarrow D$ a $\mu\circ\phi^{-1} = \ur\circ\psi^{-1}$-a.s. natural transformation. Applying Lemma~\ref{lemma:chaining} gives that $\eta$ is a $\ur$-a.s. natural transformation. Because $\theta\circ\phi = \id_D$ it holds that
			$$\eta = \eta\circ(\theta\circ\phi)^{-1} = \eta\circ\phi^{-1}\circ\theta^{-1} = \ur\circ\psi^{-1}\circ\theta^{-1} = \ur\circ (\theta\circ\psi)^{-1} = \ur\circ\eta^{-1},$$
			so $\eta$ gives the desired functional representation of $\mu$. 
		\end{proof}
	
		Versions of (weak) FRTs for finite depth can be obtained from unbounded depth case by the following; the proof is placed in the appendix, it is very technical.
	
		\begin{proposition}\label{prop:depth-restriction}
			Let $D$ be a Borel data structure with $k = \depth(D)<\infty$ and let 
			$$r:R\rightarrow \Rk$$ 
			be the rule that has components 
			$$r_a:R_a\rightarrow \Rk_a, u\mapsto u\circ \iota_{\binom{a}{\leq k},2^a}.$$
			Then the following holds:
			\begin{itemize}
				\item[(i)] $r$ is a natural transformation with $\urk = \ur\circ r^{-1}$.
				\item[(ii)] for every natural transformation $\eta:R\rightarrow D$ exists a natural transformation $\tilde\eta:\Rk\rightarrow D$ with $\eta = \tilde\eta\circ r$.
				\item[(iii)] for every $\ur$-a.s. natural transformation $\eta:R\rightarrow D$ exists a $\urk$-a.s. natural transformation $\tilde\eta:\Rk\rightarrow D$ with $\eta = \tilde\eta\circ r$ $\ur$-almost surely, that is for every $a$ it holds that $\eta_a(u) = \tilde\eta_a\circ r_a(u)$ for $\ur_a$-almost all $u\in R_a$. 
			\end{itemize}
		\end{proposition}
	
		The weak FRT for bounded depth follows easily:
	
		\begin{proof}[Proof of Theorem~\ref{thm:frt-weak-depth}]
			Let $D$ have depth $k=\depth(D)<\infty$ and let $\mu\in\SYM(D)$. By Theorem~\ref{thm:frt-weak} there exists a $\ur$-a.s. natural transformation $\eta:R\rightarrow D$ with $\mu = \ur\circ\eta^{-1}$. Let $r:R\rightarrow\Rk$ be as in Proposition~\ref{prop:depth-restriction}, which gives the existence of a $\urk$-a.s. natural transformation $\tilde\eta:\Rk\rightarrow D$ with $\eta = \tilde\eta\circ r$ $\ur$-almost surely and such that $\urk = \ur\circ r^{-1}$. Combined: 
			$$\mu = \ur\circ\eta^{-1} = \ur\circ (\tilde\eta\circ r)^{-1} = \ur\circ r^{-1}\circ\tilde\eta^{-1} = \urk\circ\tilde\eta^{-1}.$$
		\end{proof}

		\section{Array-type data structures}\label{sec:arrays}
		
		Let $I:\INJ\rightarrow\INJ$ be an indexing system, see Definition~\ref{def:indexing-system}. 	
		
		\begin{definition}\label{def:indexing-system-constructions}
			Let $b$ be a finite set and $\bbi\in I_b$. Define
			\begin{itemize}
				\item $\dom(\bbi) = \bigcap_{b'\subseteq b, \bbi\in I_{b'}}b'$ (domain of $\bbi$ = IDs used to build $\bbi$),
				\item $|\bbi| = |\dom(\bbi)|$ the size of $\bbi$,
				\item $\stab(\bbi) = \{\pi~|~\pi:\dom(\bbi)\rightarrow\dom(\bbi)~\text{bijective with}~I[\pi](\bbi)=\bbi\}$,
				\item for any other index $\bbi'$ write $\bbi\sim\bbi'$ iff there exists an injection $\tau$ such that $I[\tau](\bbi) = \bbi'$.
			\end{itemize}	
		\end{definition}
	
		Using functorality of $I$ shows that $\stab(\bbi)$ is a finite group. The indexing system axioms give the following, a proof is given in the Appendix.
		
		\begin{lemma}\label{lemma:ind}
			Let $\bbi\in\Ib$ and $\tau:b\rightarrow a$ be injective.
			\begin{enumerate}
				\item $\dom(\bbi)$ does not depend on $b$,
				\item $I[\tau](\bbi) = I[\hat\tau](\bbi)$,
				\item $\dom(I[\tau](\bbi)) = \tau(\dom(\bbi))$,
				\item for two injections $\tau:\dom(\bbi)\rightarrow a, \sigma:\dom(\bbi)\rightarrow b$ it holds $I[\tau](\bbi) = I[\sigma](\bbi)$ if and only if there exists $\pi\in\stab(\bbi)$ with $\tau\circ\pi(i) = \sigma(i)$ for all $i\in\dom(\bbi)$,
				\item $\sim$ is an equivalence relation on indices. 
			\end{enumerate}
		\end{lemma}
		
		\begin{exmp}
			Some examples for Definition~\ref{def:indexing-system-constructions}:
			\begin{itemize}
				\item $I=\square$: for $\bbi = i \in b$ it is $\dom(\bbi) = \{i\}$, $|\bbi|=1$, $\stab(\bbi) = \{\id_{\{i\}}\}$. All indices are equivalent.
				\item $I = \binom{\square}{k}$: for $\bbi = \{i_1,\dots,i_k\}\in\binom{b}{k}$ it is $\dom(\bbi)=\bbi$, $|\bbi|=k$, $\stab(\bbi) = \{\pi|\pi:\bbi\rightarrow\bbi~\text{bijective}\}$. All indices are equivalent.
				\item $I = 2^{\square}$: for $\bbi\in 2^b$ let $k\geq 0$ be such that $\bbi\in \binom{b}{k}$. For this index everything is as in the previous example. Two indices in $2^{\square}$ are equivalent iff they have the same size.
				\item $I = \square^k_{\neq}$: for $\bbi = (i_1,\dots,i_k)\in b^k_{\neq}$ it is $\dom(\bbi) = \{i_1,\dots,i_k\}$, $|\bbi|=k$ and $\stab(\bbi) = \{\id_{\{\dom(\bbi)\}}\}$. All indices are equivalent. 
				\item $I = \square^*_{\neq}$: for $\bbi\in b^*_{\neq}$ let $k\geq 0$ be such that $\bbi\in b^k_{\neq}$. For this index everything is as in the previous example. Two indices in $\square^*_{\neq}$ are equivalent iff they have the same size.
				\item $I = \square^k$: for $\bbi = (i_1,\dots,i_k)\in b^k$ it is $\dom(\bbi) = \{i_1,\dots,i_k\}$ the \emph{set} of different entries and $|\bbi|$ the number of different entries.  $\stab(\bbi)$ has only one element, the identity on $\dom(\bbi)$. Every index $\bbi = (i_1,i_2,\dots,i_k)$ defines a set-partition of $[k]$ by declaring that $j,j'\in[k]$ fall in the same block iff $i_j=i_{j'}$. The indices $\bbi, \bbi'$ are equivalent iff they induce the same partition.
				\item $I = \PAIR^{(k)}$ defined by $I_b = b\sqcup\cdots\sqcup b = \{(l,i)|1\leq l\leq k, i\in b\}$. For $\bbi = (l,i)\in I_b$ it is $I[\tau](\bbi) = (l, \tau(i))$,  $\dom(\bbi) = \{i\}$, $|\bbi|=1$ and $\stab(\bbi)$ has one element (identity). Two indices $\bbi =(l,i), \bbi'=(l',i')$ are equivalent iff $l=l'$.
			\end{itemize}
		\end{exmp}

		More complex examples can emerge from composing indexing systems:
		
		\begin{theorem}\label{thm:everygroup}
			Let $I = 2^{\square}\circ \square^*_{\neq}$. For every finite group $G$ there is an index $\bbi$ in $I$ such that $\stab(\bbi)$ and $G$ are isomorphic as groups.
		\end{theorem}
		\begin{proof}
			Wlog assume $G$ is a subgroup $G\subseteq\bS_k$. An index $\bbi\in I_b = 2^{b^*_{\neq}}$ is a set $\bbi = \{\bbi_1,\dots,\bbi_l\}$ with for each $1\leq j\leq l$ it is $\bbi_j = (i_{j1},\dots,i_{jk_j})\in b^{*}_{\neq}$ and for injection $\tau:b\rightarrow a$ it is
			$$I[\tau](\bbi) = \{(\tau i_{j1},\dots \tau i_{jk_j})~|~(i_{j1},\dots,i_{jk_j})\in \bbi\}.$$
			The index $\bbi = \{(\pi 1,\dots,\pi k)|\pi\in G\}$ has $\dom(\bbi) = [k]$ and $\stab(\bbi) = G$. 
		\end{proof}
	
		The following is very useful for characterizing natural transformations $\eta:D\rightarrow\ARRAY(\cX,I)$.
	
		\begin{definition}[Skeleton of an indexing system]
			Let $I$ be an indexing system. A \emph{skeleton} for $I$ is a triple $(\Irep, \rep, \pirep)$ in which
			\begin{itemize}
				\item $\Irep$ is a set of normalized representative indices, that is for every index $\bbi$ there is exactly one $\bbia\in\Irep$ with $\bbi\sim\bbia$ and for every $\bbia\in\Ir$ it is $\dom(\bbia) = [k]$ with $k=|\bbia|$,
				\item $\rep$ is the rule that maps every index $\bbi$ to the unique $r(\bbi)\in\Irep$ with $\bbi\sim r(\bbi)$,
				\item $\pirep$ is a rule that maps every index $\bbi$ with $k=|\bbi|$ to a bijection $\pii:[k]\rightarrow \dom(\bbi)$ satisfying $I[\pii](r(\bbi)) = \bbi$.
			\end{itemize}
		\end{definition}
	
		\begin{exmp}\label{exmp:skeleton}
			Consider two minimal examples, related to (E2) and (E3) from the introduction:
			\begin{itemize}
				\item $I=\square^2_{\neq}$ has a skeleton given by $\Irep = \{(1,2)\}$ and for $\bbi=(i_1,i_2)$ it is $r(\bbi) = (1,2)$ and $\pi_{\bbi}:\{1,2\}\rightarrow\{i_1,i_2\}, j\mapsto i_j$.
				\item $I=\binom{\square}{2}$ has $\Irep = \{\{1,2\}\}$ and for $\bbi = \{i_1,i_2\}$ (with $i_1\neq i_2$) it is $r(\bbi) = \{1,2\}$. Now a problem arises: $\pirep$ should be a rule that maps \emph{any} two-element set $\bbi = \{i_1,i_2\}$ to a bijection $\pi_{\bbi}:\{1,2\}\rightarrow\{i_1,i_2\}$ with $I[\pi_{\bbi}]([2])=\bbi$; but both bijections $[2]\rightarrow\{i_1,i_2\}$ have this property. To justify the existence of a rule $\pirep$ requires
				\begin{itemize}
					\item Global Axiom of Choice if IDs are arbitrary,
					\item (Usual) Axiom of Choice if IDs are elements only of some fixed but arbitrary uncountable set,
					\item Countable Axiom of Choice if IDs are elements only of some fixed but arbitrary countable set.
				\end{itemize} 
				A choice axiom is not needed when IDs are elements of some fixed but arbitrary set that comes equipped with a total order: in this case one can choose $\pi_{\bbi}:[2]\rightarrow\{i_1,i_2\}$ to be the strictly increasing function. This was done in the index arithmetic of Chapter~7 in \cite{kallenberg2006probabilistic}, where IDs are always from $\bN\subseteq\bZ$. We shortly see that the concrete choice of $\pirep$ does not really matter; but it is pleasant to have one available.
			\end{itemize}
		\end{exmp}
	
		As seen in the example, the following requires the global axiom of choice.
	
		\begin{proposition}[Existence of a skeleton]
			Every indexing system $I$ has a skeleton $(\Irep,\rep,\pirep)$. 
		\end{proposition}
		\begin{proof}
			It is $\sim$ an equivalence relation on indices. Let $\cT = \{\bbi|\dom(\bbi)=[k]~\text{for some}~k\geq 0\}$ which is a countable set. It is easy to see that every index $\bbi$ is equivalent to some index from $\cT$. Restricting $\sim$ to $\cT$ one can apply axiom of countable choice and obtain $\Ir$ together with a choice function $r':\cT\rightarrow\Ir$. Since for every index $\bbi$ it is $\{\bbi'\in \cT|\bbi'\sim\bbi\}$ a non-empty set, applying global choice gives a rule $r''$ that maps every index $\bbi$ to some index $r''(\bbi)\in\cT$ with $r''(\bbi)\sim\bbi$. The rule $r$ is defined as $r = r'\circ r''$. Obviously, $r$ is uniquely determined given $\Irep$.\\
			For every index $\bbi$ with $k=|\bbi|$ it is $\bbi\sim r(\bbi)\in\Irep$, hence $\dom(r(\bbi)) = [k]$ and by definition of $\sim$ it is
			$$\cA_{\bbi} = \{\pi:[k]\rightarrow\dom(\bbi)|~\text{$\pi$ is bijective with $I[\pi](r(\bbi))=\bbi$}\}$$
			a non-empty set. Applying global choice again gives the rule $\pirep$ which maps every index $\bbi$ to an element $\pi_{\bbi}\in \cA_{\bbi}$. 
		\end{proof}

		\begin{remark}
			For $D = \ARRAY(\cX,I)$ with $|\cX|\geq 2$ it is straightforward to check that $\depth(D) = \max\{|\bbia|~|~\bbia\in\Ir\}$, where $\Irep$ is an arbitrary choice of representative indices. 
		\end{remark}
	
		Let $D$ be a BDS. As before, for a rule $\eta:D\rightarrow\ARRAY(\cX,I)$ that maps every finite set $a$ to a measurable map $\eta_a:D_a\rightarrow \cX^{I_a}$ it is $\eta_{a,\bbi}:D_a\rightarrow\cX, \bbi\in I_a$ the $\bbi$-th component function of $\eta_a$, that is $\eta_a(\cdot) = (\eta_{a,\bbi}(\cdot))_{\bbi\in I_a}$.
		
		\begin{theorem}\label{thm:nat-array}
			Let $D$ be a BDS, $\cX$ a Borel space and $I$ an indexing system with skeleton $(\Irep,\rep,\pirep)$. A one-to-one correspondence between 
			\begin{enumerate}
				\item natural transformations $\eta:D\rightarrow\ARRAY(\cX,I)$ and
				\item sequences $(f_{\bbia})_{\bbia\in\Ir}$ such that for every $\bbia\in\Ir$ with $k=|\bbia|$ it is 
				$$f_{\bbia}:D_{\{1,\dots,k\}}\rightarrow \cX$$ 
				measurable with $f_{\bbia}=f_{\bbia}\circ D[\pi]$ for every $\pi\in\stab(\bbia)\subseteq\bS_{[k]}$
			\end{enumerate}
			is given by
			\begin{itemize}
				\item $\eta \mapsto (f_{\bbia})_{\bbia\in\Irep}$ with $f_{\bbia} = \eta_{[k], \bbia}, k=|\bbia|$,
				\item $(f_{\bbia})_{\bbia\in\Irep}\mapsto \eta$ with $\eta_{a,\bbi} = f_{r(\bbi)}\circ D[\pi_{\bbi}]\circ D[\iota_{\dom(\bbi),a}]$.
			\end{itemize}
			Further, the construction $(f_{\bbia})_{\bbia\in\Irep}\mapsto \eta$ does not depend on a concrete choice of $\pirep$. 
		\end{theorem}
	
		\begin{proof}
			Let $E=\ARRAY(\cX,I)$. For a rule $\eta:D\rightarrow E$ that maps finite sets $a$ to measurable functions $\eta_a:D_a\rightarrow E_a = \cX^{I_a}$ let $\eta_{a,\bbi}:D_a\rightarrow \cX, \bbi\in I_a$ be the components of $\eta_a$. As already noted in the proof of Theorem~\ref{thm:nat-trans-into-end}, the following are equivalent:
			\begin{itemize}
				\item[(i)] $\eta_a$ are the components of a natural transformation $\eta:D\rightarrow E$,
				\item[(ii)] for every injection $\tau:b\rightarrow a$ and index $\bbi\in I_b$
				\begin{equation}\label{eq:key}
				\eta_{b,\bbi}\circ D[\tau] = \eta_{a, I[\tau](\bbi)}.
				\end{equation}
			\end{itemize}
		
			\underline{$\eta \mapsto f$}. It is shown that this construction gives a sequences of kernels as in (2). Let $\bbia\in\Irep$ with $k=|\bbia|$. The measureability of $f_{\bbia}$ is clear, as it is the inner component of the measurable function $\eta_{[k]}$. Applying (\ref{eq:key}) to $a = b = [k]$ and $\tau = \pi\in\stab(\bbia)$ gives
			\begin{equation*}
				f_{\bbia}\circ D[\pi] = \eta_{[k],\bbia}\circ D[\pi] = \eta_{[k], I[\pi](\bbia)} = \eta_{[k], \bbia} = f_{\bbia},
			\end{equation*}
			that is the construction gives sequences of kernels as in (2).

			\underline{$f \mapsto \eta$}. It is shown that this construction gives a natural transformation. Let $\eta_a = (\eta_{a,\bbia})_{\bbia\in I_a}$ with $\eta_{a,\bbi} = f_{r(\bbi)}\circ D[\pi_{\bbi}]\circ D[\iota_{\dom(\bbi),a}]$. Let $\tau:b\rightarrow a$ and $\bbi\in I_b$. Property (\ref{eq:key}) needs to be verified. Consider both sides of that equation by plugging in the definitions and write $\bbi' = I[\tau](\bbi)$ for short:
			\begin{align*}
				\eta_{b,\bbi}\circ D[\tau] &= f_{r(\bbi)}\circ D[\pi_{\bbi}]\circ D[\iota_{\dom(\bbi),b}]\circ D[\tau],\\
				\eta_{a,I[\tau](\bbi)} & = f_{r(\bbi')}\circ D[\pi_{\bbi'}]\circ D[\iota_{\dom(\bbi'),a}].
			\end{align*}
			Now $\bbi\sim\bbi'$ and hence $r(\bbi)=r(\bbi')=:\bbia\in\Irep$. Calculation on the first term give
			\begin{align*}
				\eta_{b,\bbi}\circ D[\tau] &= f_{\bbia}\circ D[\pi_{\bbi}]\circ D[\iota_{\dom(\bbi),b}]\circ D[\tau]\\
										   &= f_{\bbia}\circ D[\tau\circ\iota_{\dom(\bbi),b}\circ \pi_{\bbi}]\\
										   &= f_{\bbia}\circ D[\iota_{\tau(\dom(\bbi)),a}\circ \tau^* \circ \pi_{\bbi}],
			\end{align*}
			with bijection $\tau^*:\dom(\bbi)\rightarrow\tau(\dom(\bbi)), i\mapsto \tau(i)$. Now consider the second term. It holds $\dom(\bbi') = \tau(\dom(\bbi))$. Using the symmetry of $f_{\bbia}$, for every $\pi\in\stab(\bbia)$ it follows
			\begin{align*}
			\eta_{a,I[\tau](\bbi)} &= f_{\bbia}\circ D[\pi_{\bbi'}]\circ D[\iota_{\tau(\dom(\bbi)),a}]\\
								   &= f_{\bbia}\circ D[\pi]\circ D[\pi_{\bbi'}]\circ D[\iota_{\tau(\dom(\bbi)),a}]\\
								   &= f_{\bbia}\circ D[\iota_{\tau(\dom(\bbi)),a}\circ \pi_{\bbi'}\circ \pi].
			\end{align*}
			Comparing the final calculations for both sides show that equality, hence naturality of $\eta$, follows, if there exists $\pi\in\stab(\bbia)$ such that $\pi_{\bbi'}\circ \pi = \tau^* \circ \pi_{\bbi}$, which is simply given by $\pi := (\pi_{\bbi'})^{-1}\circ \tau^*\circ \pi_{\bbi}$, one can check $\pi\in\stab(\bbia)$ noticing $I[\tau^*](\bbi) = \bbi'$.\\
			
			\underline{$\eta\mapsto f\mapsto \eta'$ implies $\eta=\eta'$.} Let $f_{\bbia} = \eta_{[k],\bbia}$ with $k=|\bbia|$. Let $a$ be finite and $\bbi\in I_a$. It is $\eta'_{a,\bbi} = f_{r(\bbi)}\circ D[\pi_{\bbi}]\circ D[\iota_{\dom(\bbi),a}]$. Write $\bbi = I[\tau](r(\bbi))$ with $\tau = \iota_{\dom(\bbi),a}\circ \pi_{\bbi}$ and apply (\ref{eq:key}) to the inner components of $\eta$:
			\begin{equation*}
				\eta_{a,\bbi} = \eta_{a,I[\tau](r(\bbi))} = \eta_{[k],r(\bbi)}\circ D[\pi_{\bbi}]\circ D[\iota_{\dom(\bbi),a}] = \eta'_{a,\bbi}.
			\end{equation*}
			
			\underline{$f\mapsto \eta\mapsto f'$ implies $f=f'$.} Let $a$ be finite and $\bbi\in I_a$. It is $\eta_{a,\bbi} = f_{r(\bbi)}\circ D[\pi_{\bbi}]\circ D[\iota_{\dom(\bbi),a}]$. For $\bbia\in\Ir$ with $|\bbia|=k$ it is 
			$$f'_{\bbia} = \eta_{[k],\bbia} = f_{r(\bbia)}\circ D[\pi_{\bbia}]\circ D[\iota_{\dom(\bbia),[k]}].$$
			Now $r(\bbia) = \bbia$ and $\dom(\bbia) = [k]$, hence $\iota_{\dom(\bbia),[k]} = \id_{[k]}$, which gives
			$$f'_{\bbia} = f_{\bbia}\circ D[\pi_{\bbia}].$$
			Now it is $\pi_{\bbia}$ such that $\bbia = I[\pi_{\bbia}](r(\bbia)) = I[\pi_{\bbia}](\bbia)$, that is $\pi_{\bbia}\in\stab(\bbia)$. Since $f_{\bbia}$ is symmetric it follows that $f'_{\bbia}=f_{\bbia}$.\\
			The one-to-one correspondence is thus shown. Only thing left to do:\\
			
			\underline{The construction $f\mapsto\eta$ does not depend on a concrete choice of $\pirep$}. Let $(\Irep,\rep)$ be a fixed choice of representative indices and let $\pirep, \pirep'$ be two rules that map an index $\bbi$ to bijections $\pi_{\bbi}, \pi'_{\bbi}:[k]\rightarrow \dom(\bbi)$ such that $I[\pi_{\bbi}](r(\bbi)) = \bbi = I[\pi'_{\bbi}](r(\bbi))$. Applying (4) from Lemma~\ref{lemma:ind} gives a $\pi\in\stab(r(\bbi))$ with $\pi_{\bbi} = \pi'_{\bbi}\circ\pi$. Suppose $\eta$ is defined using $\pi_{\bullet}$ and $\eta'$ using $\pi'_{\bullet}$. For a finite set $a$ and index $\bbi\in I_a$ the invariance of the kernel functions give
			\begin{align*}
				\eta_{a,\bbi} &= f_{r(\bbi)}\circ D[\pi_{\bbi}]\circ D[\iota_{\dom(\bbi),a}]\\
							  &= f_{r(\bbi)}\circ D[\pi'_{\bbi}\circ \pi]\circ D[\iota_{\dom(\bbi),a}]\\
							  &= f_{r(\bbi)}\circ D[\pi]\circ D[\pi'_{\bbi}]\circ D[\iota_{\dom(\bbi),a}]\\
							  &= f_{r(\bbi)}\circ D[\pi'_{\bbi}]\circ D[\iota_{\dom(\bbi),a}]\\
							  &= \eta'_{a,\bbi}.
			\end{align*}
		\end{proof}
	
		\begin{exmp}
			Natural transformations $\eta:\SEQ(\cX)\rightarrow\GRAPH$, with $\GRAPH = \ARRAY(\{0,1\},\binom{\square}{2})$, are determined by symmetric measurable maps $f:\cX\times\cX\rightarrow\{0,1\}$ and the corresponding natural transformation has components $\eta_a:\cX^a\rightarrow \{0,1\}^{\binom{a}{2}}, \eta_a((x_i)_{i\in a}) = (f(x_i, x_{i'}))_{\{i,i'\}\in\binom{a}{2}}$. The previous theorem shows: it does not matter in which order $i$ and $i'$ are picked from $\{i,i'\}$ and plugged into $f$, because of symmetry. However, a concrete choice is made in that theorem via $\pirep$. 
		\end{exmp}

		\begin{exmp}[Local modification rules]\label{exmp:local_modification_rule}
			Following Definition~1.27 in \cite{austin2010testability} the concept of a \emph{local modification rule} is introduced: let $D$ be an arbitrary BDS and $e$ be a finite set representing "extra individuals from the outside". A new BDS $D^{(e)}$ is defined by $D^{(e)}_a = D_{a\sqcup e}$ and $D^{(e)}[\tau] = D[\tau\sqcup\id_e]$, where for $\tau:b\rightarrow a$ it is $\tau\sqcup\id_e:b\sqcup e\rightarrow a\sqcup e$ the injection that operates as $\tau$ on $b$ and as $\id_e$ on $e$. A local modification rule on $D$ (using $e$) is a natural transformation $\eta:D^{(e)}\rightarrow D$. In case $D=\ARRAY(\cX,I)$ Theorem~\ref{thm:nat-array} gives an explicit description of local modification rules using kernel functions. 
		\end{exmp}
		
		A characterization of natural transformations $\eta:E\rightarrow \prod^L_l\ARRAY(\cX^{(l)}, I^{(l)})$ can be obtained easily given prior results:
		
		\begin{proof}[Proof of Theorem~\ref{thm:nat-in-array}]
			For any countable collections of BDS $E, D^{(l)}, l\in L$ there is an obvious one-to-one correspondence between natural transformations $\eta:E\rightarrow D:=\prod_lD^{(l)}$ and sequences of natural transformations $(\eta^{(l)})_l$ with $\eta^{(l)}:E\rightarrow D^{(l)}$ a natural transformation for every $l$: for every such sequence it is $\eta_a(x):=(\eta^{(l)}_a(x))_{l\in L}$ the component of a n.t. $E\rightarrow D$ and this construction is one-to-one. Hence Theorem~\ref{thm:nat-in-array} directly follows from Theorem~\ref{thm:nat-array}.
		\end{proof}
	
		It remains to show the strong FRT for array-type data structure, Theorem~\ref{thm:frt-strong}, which is obtained from the weak version by \emph{modifying} almost sure natural transformations to true ones. 
		
		\begin{lemma}\label{lemma:mod}
			Let $\cX, \cY$ be Borel spaces, $f:\cX\rightarrow\cY$ a measurable map, $G$ a countable group, $G\times\cX\rightarrow\cX$ a measurable group action and $\nu\in\sP(\cX)$ a probability measure such that for every $\pi\in G$ it it $f(\pi x) = f(x)$ for $\nu$-almost all $x\in\cX$. Then there exits a $G$-invariant measurable function $\tilde f:\cX\rightarrow\cY$ such that $\tilde f(x) = f(x)$ for $\nu$-almost all $x$. 
		\end{lemma}
		\begin{proof}
			For each $\pi$ the set $\{x\in\cX|f(\pi x)=f(x)\}\subseteq \cX$ is measurable with $\nu$-probability one. Since $G$ is countable the same is true for $\cX_0 = \{x\in\cX|f(\pi x)=f(x)~\text{for all}~\pi\in G\} = \cap_{\pi\in G}\{x\in\cX|f(\pi x)=f(x)\}$. In particular $\cX_0\neq\emptyset$. If $\cX_0=\cX$ choose $\tilde f = f$, otherwise choose $y_0\in\cY$ arbitrary and define
			$$\tilde f:\cX\rightarrow\cY,~~\hat f(x) = \begin{cases}
			f(x),&x\in\cX_0\\
			y_0,&x\in \cX\setminus \cX_0.
			\end{cases}$$
			$\tilde f$ is measurable which satisfies $\tilde f(x) = f(x)$ $\nu$-almost because $\nu(\cX_0)=1$. The $G$-invariance of $\tilde f$ follows because for every $\pi\in G$ the equivalence $x\in\cX_0\Leftrightarrow \pi x\in \cX_0$ holds.
		\end{proof}
		
		\begin{proposition}[Modification]\label{prop:modification}
			Let $E$ be a BDS, $\mu\in\SYM(E)$ and $D = \prod_{l}^L\ARRAY(\cX^{(l)},I^{(l)})$ a countable product of array-type data structures. Every $\mu$-a.s. natural transformation $\eta:E\rightarrow D$ has a modification to a true natural transformation $\tilde\eta:E\rightarrow D$ such that for every finite $a$ it holds that $\eta_a = \tilde\eta_a$ $\mu_a$-almost surely.
		\end{proposition}
		\begin{proof}
			Let $D^{(l)} = \ARRAY(\cX^{(l)},I^{(l)})$. For every $a$ it is $\eta_a:E_a\rightarrow D_a = \prod_l D^{(l)}_a$, let $\eta^{(l)}_a:E_a\rightarrow D^{(l)}_a$ be the $l$-th component function of $\eta_a$. It is $\eta^{(l)}$ a $\mu$-a.s. natural transformation. If every $\eta^{(l)}$ can be modified to a true natural transformation $\tilde\eta^{(l)}:E\rightarrow D^{(l)}$ then, because countable intersections of events with probability one have probability one, the rule $a\mapsto\tilde \eta_a = (\tilde \eta^{(l)}_a)_l$ defines the components of the desired modification $\tilde\eta$ of $\eta$. Hence one can restrict to the case $L=\{1\}$: showing that every $\mu$-a.s. natural transformation $\eta:E\rightarrow D = \ARRAY(\cX,I)$ has a modification, where $\cX$ is an arbitrary Borel space an $I$ an arbitrary indexing system. Let $(\Irep, \rep, \pirep)$ be a skeleton of $I$. For $a$ and $\bbi\in I_a$ let $\eta_{a,\bbi}:E_a\rightarrow\cX$ be the $\bbi$-th component of $\eta_a$. Let $\tau:b\rightarrow a$ be injective and $X_a\sim\mu_a$. Since $\eta$ is a $\mu$-a.s. natural transformation it holds that $\eta_b\circ E[\tau](X_a) \as D[\tau]\circ \eta_a(X_a)$. This is an almost surely equality in $\cX^{I_b}$ and hence for every index $\bbi\in I_b$ it follows that
			\begin{equation}\label{eq:comp-nt-as}
			\eta_{b,\bbi}\circ E[\tau](X_a) = \Big(\eta_b\circ E[\tau](X_a)\Big)(\bbi) \as \Big(D[\tau]\circ \eta_a(X_a)\Big)(\bbi) = \Big(\eta_a(X_a)\circ I[\tau]\Big)(\bbi) = \eta_{a,I[\tau](\bbi)}(X_a).
			\end{equation}
			For $\bbia\in \Ir$ with $\dom(\bbia) = [k], k\geq 0$ define 
			$$f_{\bbia}:E_{[k]}\rightarrow \cX,~~~f_{\bbia} = \eta_{[k],\bbia}.$$
			For $\pi\in\stab(\bbia)$ applying (\ref{eq:comp-nt-as}) to $a=b=[k], \bbi=\bbia$ and $\tau=\pi\in\stab(\bbia)$ gives
			\begin{equation*}
			f_{\bbia}\circ E[\pi](X_{[k]}) = \eta_{[k],\bbia}\circ E[\tau](X_{[k]}) \as \eta_{[k],I[\tau](\bbi)}(X_{[k]}) = f_{\bbia}(X_{[k]}).
			\end{equation*}
			By Lemma~\ref{lemma:mod} one can modify $f_{\bbia}$ to a measurable function $\tilde f_{\bbia}:E_{[k]}\rightarrow \cX$ such that $\tilde f_{\bbia}\circ E[\pi] = \tilde f_{\bbia}$ for all $\pi\in\stab(\bbia)$ (pointwise) and $\tilde f_{\bbia}(X_{[k]}) \as f_{\bbia}(X_{[k]})$. By Theorem~\ref{thm:nat-array} one can use $(\tilde f_{\bbia})_{\bbia\in \Ir}$ to construct a true natural transformation $\tilde \eta:E\rightarrow D = \ARRAY(\cX,I)$ which has components $\tilde\eta_{a,\bbi} = \tilde f_{r(\bbi)}\circ E[\pi_{\bbi}]\circ E[\iota_{\dom(\bbi),a}]$. This gives a $\mu$-a.s. modification of $\eta$: for a finite set $a$ and $\bbi\in I_a$ let $\tau = \iota_{\dom(\bbi),a}\circ \pi_{\bbi}$, which is an injection $[k]\rightarrow a$ such that $\bbi = I[\tau](r(\bbi))$ and $\tilde\eta_{a,\bbi} = \tilde f_{r(\bbi)}\circ E[\tau]$. Noticing $E[\tau](X_a)\sim X_{[k]}$ gives the calculation
			\begin{align*}
			\eta_{a,\bbi}(X_a) = \eta_{a,I[\tau](r(\bbi))}(X_a) \as \eta_{[k],r(\bbi)}\circ E[\tau](X_a) \as \tilde f_{\bbia}\circ E[\tau](X_a) = \tilde\eta_{a,\bbi}(X_a)
			\end{align*}
			and hence $\eta_a(X_a) \as \tilde\eta_a(X_a)$ (finite intersection of events with probability one).
		\end{proof}
		
		\begin{proof}[Proof of Theorem~\ref{thm:frt-strong}]
			Let $D = \prod_{l}^L\ARRAY(\cX^{(l)},I^{(l)})$ and $\mu\in\SYM(D)$. The weak FRT, Theorem~\ref{thm:frt-weak}, shows that there exists a $\ur$-a.s. natural transformation $\eta:R\rightarrow D$ with $\mu = \ur\circ \eta^{-1}$. Proposition~\ref{prop:modification} gives that $\eta$ can be modified to a true natural transformation $\tilde\eta$ with $\eta_a(u) = \tilde\eta_a(u)$ for $\ur_a$-almost all $u \in [0,1]^{2^a}$, hence $\ur\circ\eta^{-1} = \ur\circ\tilde\eta^{-1}$. In case $k = \depth(D)<\infty$ applying Proposition~\ref{prop:depth-restriction} to the true natural transformation $\tilde\eta$ gives a true natural transformation $\hat\eta:\Rk\rightarrow D$ with $\tilde\eta = \hat\eta\circ r$ and hence $\mu = \ur\circ\tilde\eta^{-1} = \ur\circ(\hat\eta\circ r)^{-1} = \ur\circ r^{-1}\circ \hat\eta^{-1} = \urk\circ\hat\eta^{-1}$.
		\end{proof}
		
		\subsection{Explicit FRT for array-type data structures} Let $I$ be an indexing system with skeleton $(\Irep,\rep,\pirep)$. Define $I_{\bN} = \cup_{n\geq 0}I_{[n]}$ and the action $\bS_{\bN}\times I_{\bN}\rightarrow I_{\bN}$ as 
		$$(\pi, \bbi)\mapsto \pi\bbi := I[\tilde \pi](\bbi)~~~\text{with}~~~\tilde\pi:\dom(\bbi)\rightarrow \pi(\dom(\bbi)), i\mapsto \pi(i).$$
		This gives a notion of exchangeability in arrays as in (\ref{eq:array}). For every bijection $\pi:b\rightarrow a$ it is $\im(\pi):2^b\rightarrow 2^{a}, b'\subseteq b\mapsto \pi(b')\subseteq a$. The following is a consequence of Theorems~\ref{thm:frt-strong}, Theorem~\ref{thm:nat-array} and formulated in terms of natural extensions of arrays, see Section~\ref{sec:extension-arrays}.
		
		\begin{corollary}\label{cor:frt-array}
			Let $\cX$ be a Borel space. For every exchangeable $\cX$-valued process $X = (X_{\bbi})_{\bbi\in I_{\bN}}$ there exist kernel functions $(f_{\bbia})_{\bbia\in \Ir}$ such that for every $\bbia\in \Ir$ with $\dom(\bbia) = [k], k\geq 0$ it is
			\begin{itemize}
				\item $f_{\bbia}:[0,1]^{2^{[k]}}\rightarrow \cX$ measurable,
				\item $f_{\bbia}(u) = f_{\bbia}(u\circ \im(\pi))$ for every $u\in [0,1]^{2^{[k]}}$ and $\pi\in\stab(\bbia)\subseteq \bS_{[k]}$
			\end{itemize}
			and such that 
			\begin{equation*}
				\big(X_{\bbi}\big)_{\bbi\in I_{\bN}} \ed \Big(f_{r(\bbi)}\Big((U_{e})_{e\subseteq\dom(\bbi)}\circ \im(\pi_{\bbi})\Big)\Big)_{\bbi\in I_{\bN}},
			\end{equation*}
			with $U_a, a\in\finN$ iid $\sim\unif[0,1]$. The representation does not depend on the concrete choice of $\pirep$ by symmetry of the kernels. 
		\end{corollary}
	
		It is directly seen that the FRT needs randomization up to order $k = \max\{|\bbia|:\bbia\in\Irep\} = \depth(\ARRAY(\cX,I))$. Applying the representation to $I=\square, \binom{\square}{2}, \square^2_{\neq}$ gives the examples (E1)-(E3), applying it to $I=\square^*_{\neq}$ gives back Theorem~A (from which everything started). It is noted that deriving a FRT for a particular indexing system $I$ from Hoover's (or any other known) FRT may often be more or less easy by "elementary" arguments - which then often depend on the concrete indexing system $I$ considered. The result above has worked these arguments out simultaneously for any indexing system.

		\subsection{Atomic indexing systems} To understand what indexing systems are about it is insightful to consider \emph{atomic} indexing systems. $I$ is called \emph{atomic} if there exists a unique representative index, that is: if $(\Irep,\rep,\pirep)$ is a skeleton, then $\Irep = \{\bbia\}$ has one single element $\bbia$ with $\dom(\bbia) = [k]$ for some $k\geq 0$. It follows that for every index $\bbi$ from $I$ it is $|\bbi|=k$, $r(\bbi)=\bbia$ and $I[\pi_{\bbi}](\bbia) = \bbi$. 
		
		\begin{exmp}
			Atomic indexing systems are $\square$ with representative index $\bbia = 1$, $\binom{\square}{k}$ with $\bbia = \{1,\dots,k\}$ and $\square^k_{\neq}$ with $\bbia = (1,\dots,k)$. Examples of non-atomic indexing systems are $\binom{\square}{\leq k}$ in case $k\geq 1$, $\square^k$ in case $k\geq 2$, $2^{\square}$ or $\square^*_{\neq}$. 
		\end{exmp}
		
		Using Lemma~\ref{lemma:ind} it is straightforward to show that an atomic indexing systems $I$ with $|\bbia|=k$ is always "in between" $\binom{\square}{k}$ and $\square^k_{\neq}$:
		for every finite set $a$ it is
		\begin{equation*}
			|I_a| = \frac{k!}{|\stab(\bbia)|} \cdot \binom{|a|}{k},~~\text{so $|\binom{a}{k}| \leq |I_a| \leq |a^k_{\neq}|$}.
		\end{equation*}
		Further, Theorem~\ref{thm:nat-array} can be used to justify that natural embeddings $\phi^1, \phi^2$
		\begin{equation*}
		\ARRAY(\cX,\binom{\square}{k}) \overset{\phi^1}{\longrightarrow} \ARRAY(\cX,I) \overset{\phi^2}{\longrightarrow} \ARRAY(\cX,\square^k_{\neq})
		\end{equation*}
		are given by 
		\begin{itemize}
			\item $\phi^1_a(x) = (x(\dom(\bbi)))_{\bbi\in I_a}$; the kernel is the identity function $f:\cX^{\binom{[k]}{k}}\rightarrow\cX, v\mapsto v\equiv v([k])$. Injectivity of $\phi^1_a$ follows because $\dom(\bbi)$ ranges over $\binom{a}{k}$ as $\bbi$ ranges over $I_a$,
			\item $\phi^2_a(x) = (x(I[\tau_{\bbj,a}](\bbia)))_{\bbj\in a^k_{\neq}}$; the kernel is $f:\cX^{I_{[k]}}\rightarrow\cX, v\mapsto v(\bbia)$. Injectivity of $\phi^2_a$ follows because $I[\tau_{\bbj,a}](\bbia)$ ranges over all elements from $I_a$ when $\bbj$ ranges over $a^k_{\neq}$ (because then, $\tau_{\bbj,a}$ ranges over all injections $[k]\rightarrow a$).
		\end{itemize}
		
		The term "atomic" is justified by the fact that every indexing system $I$ decomposes into atomic indexing systems: if $(\Irep, \rep, \pirep)$ is a skeleton for $I$ and the representatives are enumerated as $\Irep = \{\bbia_m|m\in M\}$, $M$ a countable set, then for every $m\in M$ an atomic indexing system is given by $I^{(m)}$ defined as $I^{(m)}_b = \{\bbi\in I_b|\bbi\sim\bbia_m\}\subseteq I_b$ and $I^{(m)}[\tau](\bbi) = I[\tau](\bbi)$ for $\bbi\in I^{(m)}_b$. For every finite set $a$ it is $I_a = I^{(1)}_a + I^{(2)}_a + \dots$ a disjoint union because $\sim$ is an equivalence relation on indices; a natural isomorphism $\phi:\ARRAY(\cX,I)\rightarrow\prod_{m\in M}\ARRAY(\cX,I^{(m)})$ is given by components
		\begin{equation*}
			\phi_a:\cX^{I_a}\rightarrow \prod_{m\in M}\cX^{I^{(m)}_a},~~~x \mapsto \Big(x\circ \iota_{I^{(m)}_a, I_a}\Big)_{m\in M}.
		\end{equation*}
		A formal remark on this: if $m\in M$ and $a$ are such that $I^{(m)}_a=\emptyset$, then $\cX^{I^{(m)}_a} = \cX^{\emptyset}$ is the discrete one-point Borel space consisting of the unique function $\emptyset\rightarrow\cX$, which for every $x\in\cX^{I_a}$ equals $x\circ \iota_{\emptyset,I_a}$. In case $M$ is countable infinite, for every finite set $a$ it is $I^{(m)}_a = \emptyset$ for all but finitely many $m$.

		\section{Outlook to seperate exchangeability}\label{sec:sep}
		
		Let $k\geq 1$ be fixed. The statistical philosophy behind (classical) notions of seperate exchangeability is that there are $k$ large populations and a statistician picks from any of the $k$ populations a finite set of individuals, representing individuals from population $l\in[k]$ via IDs from some finite set $a_l$. The complete sample of individuals is represented by the tuple $(a_1,\dots,a_k)$. Picking subgroups is performed \emph{separately} on each group, that is via tuples of injections $(\tau_1,\dots,\tau_k)$ such that $\tau_l:b_l\rightarrow a_l$ is injective. Composition with $(\sigma_1,\dots,\sigma_k)$, with $\sigma_l:c_l\rightarrow b_l$ injective, is $(\tau_1,\dots,\tau_k)\circ (\sigma_1,\dots,\sigma_k) = (\tau_1\circ\sigma_1,\dots,\tau_k\circ\sigma_l)$. The same ideas leading to study BDS ($k=1$) can be extended to $k\geq 1$ and lead to consider functors 
		\begin{equation*}
			G:(\INJo)^k\rightarrow\BOREL,
		\end{equation*}
		where $(\INJo)^k$ is the $k$-fold product category of $\INJo$. A functor $G$ gives the Borel spaces $G_{(a_1,\dots,a_k)}$ representing spaces of measurements on a group of individuals represented by $(a_1,\dots,a_k)$ and for every way of (separately) picking subgroups $(\tau_1,\dots\tau_k)$ a measurable map $G[(\tau_1,\dots,\tau_k)]:G_{(a_1,\dots,a_k)}\rightarrow G_{(b_1,\dots,b_k)}$ which explains how picking subgroups transforms measured data. Imagine the statistician picks individuals and assigns IDs "randomly" and then measures data. As with Borel data structures it is straightforward to model the distribution of such a random measurement by a rule $\rho$ mapping every $(a_1,\dots,a_k)$ to some $\rho_{(a_1,\dots,a_k)}\in \sP(G_{(a_1,\dots,a_k)})$ such that for any  $(\tau_1,\dots,\tau_k):(b_1,\dots,b_k)\rightarrow (a_1,\dots,a_k)$ it holds
		\begin{equation*}
		\rho_{(b_1,\dots,b_k)} = \rho_{(a_1,\dots,a_k)}\circ G[(\tau_1,\dots,\tau_k)]^{-1}.
		\end{equation*}
		Let $\SYM(G)$ be the space of such $\rho$, which are called \emph{symmetric laws on $G$}. Any symmetric law $\rho \in \SYM(G)$ is determined on its \emph{diagonal}, that is by the values $\rho_{(a,\dots,a)}$ ranging over finite sets $a$: for every $(a_1,\dots,a_k)$ let $a=\cup_l a_l$, it is
		\begin{equation*}
		\rho_{(a_1,\dots,a_k)} = \rho_{(a,\dots,a)}\circ G[(\iota_{a_1,a},\dots,\iota_{a_k,a})]^{-1}.
		\end{equation*}
		Let $\Delta:\INJo\rightarrow (\INJo)^k$ be the diagonal functor that sends $a$ to $(a,\dots,a)$ and $\tau$ to $(\tau,\dots,\tau)$. It is 
		$$G\circ \Delta:\INJo\rightarrow\BOREL$$ 
		a Borel data structure and for $\rho\in\SYM(G)$ the rule $\rho\circ \Delta:=[a\mapsto \rho_{(a,\dots,a)}]$ is element $\rho\circ \Delta\in\SYM(G\circ \Delta)$. The map $\rho\mapsto \rho\circ\Delta$ is injective. Let
		$$\SEP(G\circ \Delta) := \{\rho\circ\Delta|\rho\in\SYM(G)\}~~~\subseteq~~~\SYM(G\circ\Delta).$$
		In this context it is reasonable to call $\mu=\rho\circ\Delta\in\SEP(G\circ\Delta)$ a \emph{seperate exchangeable} law and $\mu\in\SYM(G\circ\Delta)$ a \emph{jointly exchangeable} law on the Borel data structure $G\circ\Delta$. The statistical interpretation of the BDS $G\circ\Delta$ is as follows: a statistician picks $n\geq 0$ individuals from each of the $k$ populations, obtaining $k$ distinct groups of individuals each of size $n$, and uses a single finite set of IDs $a$ with $|a|=n$ to identify individuals within each of the $k$ groups, that is every $i\in a$ points to an individual in each of the $k$ groups. Storing (joint) information about the picked groups as data gives a value from $(G\circ\Delta)_a = G_{(a,\dots,a)}$. Picking subgroups in $G\circ\Delta$ is performed such that for every $i\in a$ the $k$ individuals represented by $i$ are treated as "linked together". Loosely speaking, this results in the following statistical interpretation of seperate and jointly exchangeable laws:
		
		\begin{itemize}
			\item Jointly exchangeable laws $\mu\in\SYM(G\circ\Delta)$ arise as follows: $\mu_a\in\sP(G_{(a,\dots,a)})$ with $n=|a|$ is the law of a measurement in which individuals are picked with an arbitrary coupling, that is every pick $i$ represents a simultaneous pick of $k$ individuals, exactly one from each population. 
			\item Seperate exchangeable laws $\mu\in\SEP(G\circ\Delta)$ correspond to the coupling being "independent", that is for every $i$ and every $l\in[k]$ an inidividual from population $l$ is picked randomly and assigned ID $i$. Of course, $\SEP(G\circ\Delta)\subseteq \SYM(G\circ\Delta)$.
		\end{itemize}
		
		To summarize: in any BDS $D$ represented in the form $D = G\circ\Delta$ there is a canonical notion of seperate exchangeability $\SEP(D)$ being stronger than (joint) exchangeability $\SYM(D)$, that is $\SEP(D)\subseteq\SYM(D)$. Of course, when $k=1$ it is $D = G\circ\Delta = G$ and $\SEP(D) = \SYM(D)$. Next, two ways are given to \emph{construct} a BDS $D$ satisfying $D = G\circ \Delta$ for some $G$ constructed from a "base" BDS $D^*:\INJo\rightarrow\BOREL$.\\
		
		For injections $\tau_1,\dots,\tau_k$ with $\tau_l:b_l\rightarrow a_l$ let $\tau_1\times\cdots\times\tau_k:b_1\times\cdots\times b_k\rightarrow a_1\times\cdots\times a_k$ act as $(i_1,\dots,i_k)\mapsto (\tau_1 i_1,\dots,\tau_k i_k)$. The coproduct version is the map $\tau_1\sqcup\cdots\sqcup\tau_k:b_1\sqcup\cdots\sqcup b_k\rightarrow a_1\sqcup\cdots\sqcup a_k$ acing on $b_l$ as $\tau_l$. Considering only the "diagonal" version of these constructions leads to the indexing systems $\square^k$, which sends $b$ to $b^k = b\times\cdots\times b$ and $\tau$ to $\vec{\tau} = \tau\times\tau\cdots\times\tau$, and $\PAIRk$, which sends $b$ to $b\sqcup\cdots\sqcup b$ ($k$-times) and $\tau$ to $\tau\sqcup\cdots\sqcup\tau$.\\
		
		Let $D^*:\INJo\rightarrow\BOREL$ be an arbitrary BDS. 
		
		\begin{itemize}
			\item[(C1)] $D := D^*\circ \square^k$ satisfies $D = G\circ\Delta$ with $G$ being 
			\begin{equation*}
			G_{(a_1,\dots,a_k)} = D^*_{a_1\times\cdots\times a_k}~~~\text{and}~~~G[(\tau_1,\dots,\tau_k)] = D^*[\tau_1\times\cdots\times\tau_k].
			\end{equation*}
			\item[(C2)] $D := D^*\circ \PAIRk$ satisfies $D = G\circ\Delta$ with $G$ being 
			\begin{equation*}
			G_{(a_1,\dots,a_k)} = D^*_{a_1\sqcup\cdots\sqcup a_k}~~~\text{and}~~~G[(\tau_1,\dots,\tau_k)] = D^*[\tau_1\sqcup\cdots\sqcup\tau_k].
			\end{equation*}
		\end{itemize}
	
		The statistical interpretation of these constructions is straightforward: let the $k$ picked groups, each of size $n=|a|$, be represented by $(a,\dots,a)$. Data of the form $D = D^*\circ\square^k$ is measured by building all pairs $(i_1,\dots,i_k)\in a^k$ and using these pairs as new "individuals" on which data is measured according to $D^*$. Data of the form $D = D^*\circ\PAIRk$ is measured by pooling the individuals from the different groups together (in an identifiable way), which gives new IDs $(l,i), i\in a, l\in[k]$ (the elements of $\PAIRk_a$), and using $D^*$ to measure data on the pooled group.

		\begin{exmp}
			The classical notion of seperate exchangeability is about arrays indexed by $\bN^k$. This notion can be derived from the previous construction as follows:\\
			It is $D = \ARRAY(\cX,\square^k) = D^*\circ\square^k$ with $D^*=\SEQ(\cX)$. As seen before (natural extension of arrays + correspondence with laws of random measurements using a countable infinite set of IDs), jointly exchangeable laws $\mu\in\SYM(\ARRAY(\cX,\square^k))$ correspond to laws of $\cX$-valued arrays $X = (X_{\bbi})_{\bbi\in\bN^k}$ satisfying for every bijection $\pi:\bN\rightarrow\bN$
			$$X ~~\ed~~\big(X_{(\pi i_1,\dots,\pi i_k)}\big)_{\bbi = (i_1,\dots,i_k)\in\bN^k}.$$
			The derived notion of seperate exchangeability is the classical one: the law of $X$ is represented by a seperate exchangeable law $\mu\in\SEP(\ARRAY(\cX,\square^k))$ iff 
			$$X ~~\ed~~\big(X_{(\pi_1 i_1,\dots,\pi_k i_k)}\big)_{\bbi = (i_1,\dots,i_k)\in\bN^k}$$
			holds for any $k$ bijections $\pi_1,\dots,\pi_k:\bN\rightarrow\bN$. 
		\end{exmp}
	
		\begin{exmp}
			Let $D^*=\TOTAL$ be the data structure of strict total orders. Seperate exchangeable laws in $D = \TOTAL\circ\PAIR^{(2)}$ appeared in \cite{choi2017doob} in the context of identifying the Doob-Martin boundary of a specific combinatorial Markov chain producing randomly growing words over a $k=2$-letter alphabet; a (functional) representation of seperate exchangeable laws in this case is given by their Theorem~6.12 (where "exchangeable" instead of "seperate exchangeable" is used).
		\end{exmp}
	
		\begin{remark}
			The constructions (C1), (C2) also have \emph{outer} versions, details are only given for the product: let $D^1,\dots,D^k$ be BDS and consider the product $D=D^1\times\cdots\times D^k$. It is $D = G\circ\Delta$ with $G_{(a_1,\dots,a_k)} = \prod_l D^l_{a_l}$ and $G[(\tau_1,\dots,\tau_k)] = \prod_l D^l[\tau_l]$. The statistical interpretation is that on each of the $k$ groups data is measured separately, on group $l$ according to $D^l$, and recorded in a tuple. Seperate exchangeable laws can be easily identified: $\mu\in \SYM(D)$ is seperate exchangeable iff for every finite set $a$
			$$\mu_a(\cdot) = \int_{\eSYM(D^1)\times\cdots\times\eSYM(D^k)}\mu^1_a(\cdot)\otimes\cdots\otimes\mu^k_a(\cdot)d\Xi(\mu^1,\dots,\mu^k)$$
			for a uniquely defined probability measures $\Xi$ on $\eSYM(D^1)\times\cdots\times\eSYM(D^k)$. Note the coincidence that for $D^l = D^* = \SEQ(\cX)$ for all $l$ it is $D = D^*\times\cdots\times D^* \simeq D^*\circ\PAIRk \simeq \SEQ(\cX^k)$, which is special to sequential data.
		\end{remark}
		
		In future work the abstract notion of seperate exchangeability should be investigated further. For that, many of the results derived for functors $D:\INJo\rightarrow\BOREL$ and their exchangeable=symmetric laws should have a straightforward generalization to functors $G:(\INJo)^k\rightarrow\BOREL$ for arbitrary $k\geq 1$. Studying functional representations for seperate exchangeable laws should be particularly fruitful for BDS of the form $D = D^*\circ \square^k$ with $D^* = \ARRAY(\cX,I)$, as in this case $D = \ARRAY(\cX,I)\circ\square^k = \ARRAY(\cX,I\circ \square^k)$ is of array-type again, for which general results have been presented. The same holds for $D = D^*\circ\PAIRk = \ARRAY(\cX, I\circ\PAIRk)$.

		\section{Concluding remarks/outlook}
		
		\begin{remark}[Kernels as morphisms]\label{rem:kernels}
			Let $\KBOREL$ be the category that has Borel spaces as objects and probability kernels as morphisms, that is: a morphism from $\cX$ to $\cY$ in $\KBOREL$ is a measurable map $k:\cX\rightarrow \sP(\cY)$ and composition with $k':\cY\rightarrow\sP(\cZ)$ is defined by disintegration:
			$$(k'\circ k)(x,\cdot) = \int_{\cY}k'(y,\cdot)k(x,dy), x\in\cX.$$ 
			A good part of our definitions and results should also hold when $\BOREL$ is replaced by $\KBOREL$, that is the initial object of study would be functors $D:\INJo\rightarrow\KBOREL$; however, the focus of this work was on \emph{functional} aspects of exchangeability based on the statistical interpretation of "manipulating measurements in a deterministic way". A possible benefit on extending the theory from $\BOREL$ to $\KBOREL$ is to be investigated. It is noted that the results of this work would embed nicely into the more general framework: the category $\KBOREL$ is obtained as the Kleisli category induced by the Giry monad, see \cite{giry1982categorical}, which has a version on $\BOREL$. The results about BDS and natural transformations between BDS would embed in the $\KBOREL$-setting by identifying a function $\cX\rightarrow\cY, x\mapsto f(x)$ with the kernels $\cX\rightarrow\sP(\cY), x\mapsto \delta_{f(x)}$. 
		\end{remark}
	
		\begin{remark}[Conjecture about generalized Noise-Outsourcing Lemma]\label{rem:conj}
			\cite{austin2015exchangeable} studied exchangeable laws in $\sP\circ D$ with $D = \prod_{j=0}^k\ARRAY(\cX^{(j)},\binom{\square}{j})$. As noted in Remark~\ref{rem:skew}, the notions of natural transformations and kernel functions implicitly appeared in that context as skew-product type functions and skew-product tuples. The results obtained there lead to the following conjecture, formulated in a "weak" form for arbitrary BDS of arbitrary depth:
			\begin{conjecture}
				Let $E$ and $D$ be Borel data structures. 
				\begin{itemize}
					\item For every $\mu\in\SYM(E)$ and $\mu$-a.s. natural transformation $\eta:E\rightarrow\sP\circ D$ there exists a $\mu\otimes\ur$-a.s. natural transformation $\tilde\eta:E\times R\rightarrow D$ such that for every finite set $a$ it is $\eta_a(x) = \ur_a\circ\tilde\eta_a(x,\cdot)^{-1}$ for $\mu_a$-almost all $x\in E_a$,
					\item Abstract Noise-Outsourcing Lemma: for every $\rho\in\SYM(E\times D)$ with first marginal $\mu\in\SYM(E)$ there exists a $\mu\otimes\ur$-a.s. natural transformation $\eta:E\times R\rightarrow D$ such that $\rho = \mu\otimes\ur \circ (1_E\otimes \eta)^{-1}$ with $1_E\otimes \eta:E\times R\rightarrow E\times D$ having components $E_a\times R_a\rightarrow E_a\times D_a, (x,y)\mapsto (x,\eta_a(x,y))$.
				\end{itemize}
			\end{conjecture}
		\end{remark}
		
		\begin{remark}[Topological assumptions]
			Topological assumptions may be needed to obtain further results, in particular for sub-data structures, which is reflected by the topological assumptions being made in \cite{austin2010testability} for studying hereditary properties. For that, one could replace $\BOREL$ with $\POLISH$ or $\COMP$ (continuous maps between polish/compact metrizable spaces), both of which have a version of the Giry monad (equipping probability measures with the topology of weak convergence). An interesting question arises: it is known that for every measurable group action $\bSi\times\cS\rightarrow\cS$ on a Borel space $\cS$ there exists a polish topology on $\cS$ generating its $\sigma$-field and such that $\bSi\times\cS\rightarrow\cS$ becomes a continuous group action, see \cite{kechris2000descriptive}. Let $L:\POLISH\rightarrow\BOREL$ be the forgetful functor mapping a polish space to the obtained Borel space and a continuous map to itself; is it true that for every BDS $D:\INJo\rightarrow\BOREL$ there exists a functor $D^*:\INJo\rightarrow\POLISH$ such that $D = L\circ D^*$? Such a functor $D^*$ would correspond to a rule that maps every finite $a$ to a polish topology $\cT_a$ on $D_a$ generating its $\sigma$-field and making all maps $D[\tau]$ continuous. 
		\end{remark}
		
		\begin{remark}[Functors in \cite{austin2010testability}]\label{rem:subcantor}
			In Definition~3.5 of \cite{austin2010testability} contravariant functors $D:\CINJo\rightarrow\SUBCANTOR$ have been introduced, where $\SUBCANTOR$ has sub-Cantor spaces as objects (topological spaces homeomorphic to a compact subsets of the standard Cantor space) and probability kernels as morphisms. Restricting such a functor to $\INJo\subset\CINJo$ and keeping only the measureability structure gives a functor $\INJo\rightarrow\KBOREL$, see Remark~\ref{rem:kernels}. The derived functor obtained from a sub-Cantor palette $(\cZ_j)_j$ (Definition~3.7) corresponds to $\prod_{j=0}^{\infty}\ARRAY(\cZ_j, \square^j_{\neq})$ in our notation.
		\end{remark}
		
		\begin{remark}[Quasi-Borel spaces]
			\cite{heunen2017convenient} introduces the category of \emph{quasi-Borel spaces}, $\QBOREL$, aiming to provide a more solid mathematical foundation to applications in stochastic programming motivated from the (unpleasant) observation that $\BOREL$ is not Cartesian closed. In particular, a de~Finetti-type representation theorem for quasi-Borel-spaced exchangeable sequences is shown; given that and their statistical motivation, it seems interesting to investigate if and in what sense definitions and results in the BDS context translate to functors $D:\INJo\rightarrow\QBOREL$. 
		\end{remark}
		
		\begin{remark}[Using category theory terminology]\label{rem:ct-reformulation}
			It should be possible to translate definitions and results using more category theory terminology, which would give the opportunity to search for further abstractions. For example, Theorem~\ref{thm:frt-strong} (strong FRT for products of arrays) can be formulated as follows: Let $\ARRAY$ be the category that has countable products of array-type data structures as objects and natural transformations as morphisms. Consider the functor $\SYM:\ARRAY\rightarrow\BOREL$ that sends $D$ to the Borel space of exchangeable laws $\SYM(D) = \lim \sP\circ D$ and a natural transformation $\eta:D\rightarrow E$ to the push-forward $\eta^*:\SYM(D)\rightarrow\SYM(E), \mu\mapsto \mu\circ\eta^{-1}$. Let $\ptt$ be the one-point Borel space. Theorem~\ref{thm:frt-strong} is equivalent to the existence of a \emph{weak universal arrow from $\ptt$ to $\SYM$} witnessed by the pair $\langle R, \unif(R)\rangle$, where $\unif(R)$ is viewed as a function $\ptt\rightarrow\SYM(R)$, see \cite{mac2013categories} Section~X.2.
		\end{remark}
		
		\begin{remark}[Shift-invariance and contractability]
			Suppose $D:\INJo\rightarrow\BOREL$ is a BDS having an extension $D:\CINJo\rightarrow\BOREL$. Let $\INJ(\bN,\bN)$ the set of injections $\tau:\bN\rightarrow \bN$. It was seen that exchangeable laws $\SYM(D)$ corresponds to laws of $D_{\bN}$-valued random variables $X$ satisfying $D[\tau](X)\ed X$ for all $\tau\in\INJ(\bN,\bN)$. Any subset $G\subseteq \INJ(\bN, \bN)$ introduces a \emph{weaker} notion of invariance: (The law of) A $D_{\bN}$-valued random variable $X$ is called $G$-invariant iff $D[\tau]X\ed X$ for all $\tau\in G$; of course exchangeability induces $G$-invariance. To study $G$-invariance one can assume wlog that $G$ is closed under composition and contains $\id_{\bN}$, that is $G$ being a monoid under composition. Two classical examples fall into this frame:
			\begin{itemize}
				\item Shift-invariance: $G = \{\tau_k|k\in\bN_0\}$ with $\tau_k(i) = i+k$,
				\item Contractability/Spreadability: $G = \{\tau|\tau~\text{is strictly increasing}\}$. Note that $\tau\in G$ \emph{spreads} IDs and so, by contravariance, $D[\tau]$ \emph{contracts} measurements. 
			\end{itemize}
			Both these invariances are based on additional mathematical structure on the concrete choice of IDs $\bN$: addition for shift-invariance and a total order in case of contractability. How to invoke  additional structure on IDs into an abstract category theory framework remains open for future research, but should give interesting insights: comparing Theorems 7.15 and 7.22 in \cite{kallenberg2006probabilistic} shows a deep connection between contractability in $\ARRAY(\cX,2^{\square})$ and exchangeability in $\ARRAY(\cX,\square^*_{\neq})$.
		\end{remark}

		\section{Appendix}
		
		\subsection{Borel spaces}\label{appendix:borel_spaces}
		
		In \cite{kallenberg1997foundations} \emph{Borel spaces} are introduced as measurable spaces $\cX$ for which there exists a Borel subset $B\subseteq [0,1]$ and a bi-measurable bijection $f:\cX\rightarrow B$. Borel spaces coincide with \emph{standard Borel spaces} which are typically introduced as measurable spaces $\cX$ on which the $\sigma$-field is generated from a polish topology on $\cX$. The theory of (standard) Borel spaces is presented, for example, in \cite{kechris2012classical}.\\
		Borel spaces enjoy the following closure properties:
		
		\begin{itemize}
			\item Countable products and coproducts of Borel spaces are Borel,
			\item Measurable sub-spaces of Borel spaces are Borel,
			\item For a measurable space $\cX$ let $\sP(\cX)$ be the space of probability measures on $\cX$ equipped with the $\sigma$-field generated by the evaluation maps $\nu\in\sP(\cX)\mapsto \nu(M)\in[0,1], M\subseteq \cX$ measurable. If $\cX$ is Borel, so is $\sP(\cX)$, see Theorem~1.5 in \cite{kallenberg2017random}.
		\end{itemize}
	
		Let $\cX, \cY$ be Borel spaces and $f:\cX\rightarrow\cY$ measurable. 
		\begin{itemize}
			\item If $f$ is bijective its inverse $f^{-1}:\cY\rightarrow\cX$ is measurable,
			\item If $f$ is injective and $M\subseteq \cX$ measurable, then the image $f(M)\subseteq\cY$ is measurable and in case $M\neq\emptyset$ it is $M\rightarrow f(M), x\mapsto f(x)$ a bi-measurable bijection between the Borel spaces $M$ and $f(M)$, see Corollary~(15.2) in \cite{kechris2012classical}, 
			\item If $f$ is injective then there exists a measurable $g:\cY\rightarrow\cX$ with $g\circ f = \id_{\cX}$.
		\end{itemize}

		Borel spaces make the concept of conditional distributions well-behaved, see for example Lemma~3.1 in \cite{austin2012exchangeable}:
		
		\begin{theorem*}[Noise-Outsourcing]
			Let $\cX$ be a Borel space and $\cY$ an arbitrary measurable space. Let $(X,Y)$ be a $\cX\times\cY$-valued random variable. Then there exists a measurable function $f:\cY\times[0,1]\rightarrow\cX$ such that $(X,Y) \ed (f(Y,U), Y)$ with $U\sim\unif[0,1]$ independent from $Y$. 
		\end{theorem*}

		\subsection{Some proofs}
		
		\begin{proof}[Proof of Proposition~\ref{prop:exchangeablelaws}]
			First we check that the construction $\cL(X)\in\SYM(D;C) \mapsto \mu$ is well-defined: let $a$ be finite and let $c,c'\in\finC$ and $\pi:a\rightarrow c, \pi':a\rightarrow c'$ be two bijections. Then there exists a bijection $\sigma:c\rightarrow c'$ with $\sigma\circ\pi = \pi'$ and hence $D[\pi'](X_{c'}) = D[\sigma\circ\pi](X_{c'}) = D[\pi](D[\sigma](X_{c'}))$. Now let $d\in\finC$ be with $c\cup c'\subseteq d$. There exists a bijection $\tilde \sigma:d\rightarrow d$ such that $\tilde\sigma\circ \iota_{c,d} = \iota_{c',d}\circ \sigma$. With this $\tilde\sigma$ the functorality of $D$ and sampling consistency and exchangeability of $X$ gives
			\begin{align*}
			D[\sigma](X_{c'}) \as D[\sigma](D[\iota_{c',d}](X_d)) = D[\iota_{c',d}\circ\sigma](X_d) = D[\tilde\sigma\circ\iota_{c,d}](X_d) &= D[\iota_{c,d}](D[\tilde\sigma](X_d))\\
			& \ed D[\iota_{c,d}](X_d) \as X_c,
			\end{align*}
			which gives $D[\pi](X_c)\ed D[\pi](D[\sigma](X_{c'})) \as D[\pi'](X_{c'})$ and hence that the definition $\mu_a = \cL(D[\pi](X_c))$ does not depend on the concrete choice of $c, \pi$.\\
			Next check $\mu\in\SYM(D)$: let $\tau:b\rightarrow a$ be an injection and $\mu_a = \cL(D[\pi](X_c))$. Then $\mu_a\circ D[\tau]^{-1} = \cL\big(D[\tau](D[\pi](X_c))\big)$. Let $c' = \pi(\tau(b))\subseteq c$ and $\pi' = \widehat{\pi\circ\tau}:b\rightarrow c', i\mapsto \pi(\tau(i))$, which is bijection. It holds that $\pi\circ\tau = \iota_{c',c}\circ\pi'$. By functorality of $D$ and sampling consistency of $X$ it is 
			$$D[\tau](D[\pi](X_c)) = D[\pi\circ \tau](X_c) \as D[\pi'](X_{c'}),$$
			hence $\mu_a\circ D[\tau]^{-1} = \mu_b$.\\
			Next check that the construction $\SYM(D;C)\rightarrow \SYM(D)$ is a bijection. It is injective: let $\cL(X)\in\SYM(D;C)$ with constructed rule $\mu\in\SYM(D)$. For $c\in\finC$ it is $\mu_c = \cL(X_c)$. By sampling consistency the law of $X = (\Xc)_{c\in\finC}$ is determined by $(\mu_c)_{c\in\finC}$ hence the construction is injective.\\
			Next check that the construction is surjective, that is for every rule $\eta\in\SYM(D)$ there exists $\cL(X)\in\SYM(D;C)$ with $X_c\sim \mu_c$ for all $c\in\finC$. The Borel space assumption is needed to apply Kolmogorov extension theorem: let $(c_n)_{n\geq 1}\subseteq C$ be an increasing sequence of finite sets with $C = \cup_n c_n$. Applying Theorem 8.21 in \cite{kallenberg1997foundations} gives the existence of a stochastic process $(X_{c_n})_{n\geq 1}$ such that $X_{c_n}\sim \mu_{c_n}$ for all $n$ and $D[\iota_{c_m,c_n}](X_{c_n}) = X_{c_m}$ almost surely for all $m\leq n$. For any finite set $c\in\finC$ let $c_n$ be the smallest set with $c\subseteq c_n$ and define $X_c = D[\iota_{c,c_n}](X_{c_n})$. It can easily be checked that $X = (X_c)_{c\in\finC}$ is an exchangeable $D$-measurement, that is $\cL(X)\in \SYM(D;C)$, whose law gives back the rule $\mu$. 
		\end{proof}
	
		\begin{proof}[Proof of Proposition~\ref{prop:translation}]
			(1)~Let $A$ be countable infinite. $D_A$ is a measurable subset of $\prod_{a\in\finA}D_a$ because it is the countable intersection of sets $\{(x_a)_a|D[\iota_{c,b}](x_b)=x_c\}$ over $c\subseteq b\in\finA$, the latter are measurable because $D[\tau]$ is for every $\tau$. By assumption $\SYM(D)\neq\emptyset$, let $\mu\in\SYM(D)$. By Proposition~\ref{prop:exchangeablelaws} there exist an exchangeable $D_A$-measurement $X=(X_a)_{a\in\finA}$ with $X_a\sim\mu_a$ for every $a$, it holds that $\bP[X\in D_A]=1$ and hence $D_A\neq\emptyset$. The property $X\ed X'$ iff $X_a \ed X'_a$ for all finite $a$ follows from $D[\iota_{a',a}](X_a)= X_{a'}$ for all $a'\subseteq a$ together with laws of processes $X = (X_a)_{a}$ being determined by finite dimensional margins.\\
			(2)~Let $A$ be countable infinite. By definition for every $a\in\finA$ it is $D[\iota_{a,A}]:D_A\rightarrow D_a, (x_{a'})_{a'\in\finA}\mapsto x_a$. The $\sigma$-field on $D_A$ inherited of $\prod_aD_a$ is also generated by these projections, in particular $D[\iota_{a,A}]$ is measurable. It is easily checked that the extension of $D$ to arbitrary countable sets satisfies functorality, that is for all composable injections $\tau, \sigma$ between countable sets its holds $D[\tau\circ\sigma] = D[\sigma]\circ D[\tau]$ and $D[\id_A] = \id_{D_A}$. Only the measureability of $D[\tau]:D_A\rightarrow D_B$ needs to be checked: let $\tau:B\rightarrow A$ be injective. If $B=b$ is finite then $D[\tau] = D[\hat\tau]\circ D[\iota_{\tau(b),A}]$ is measurable by composition. If $B$ is also infinite, then $D[\tau]:D_A\rightarrow D_B$ is measurable iff $D[\iota_{b,B}]\circ D[\tau]$ is measurable for every $b\in\finB$. By functorality $D[\iota_{b,B}]\circ D[\tau] = D[\tau\circ\iota_{b,B}]:D_A\rightarrow D_b$ which was seen to measurable before.\\
			(3)~Let $X = (\Xa)_{a\in\finA}$ be an exchangeable $D$-measurement using IDs $A$. (iv)$\Rightarrow$(iii)$\Rightarrow$(ii) is clear. Assume (ii) and let $a\in\finA$ and $\pi:a\rightarrow a$ bijective. Extend $\pi$ to a bijection $\tilde\pi:A\rightarrow A$ via $\tilde\pi = \pi$ on $a$ and $\tilde \pi(i)=i$ on $A\setminus a$. By (ii) it is $D[\tilde \pi]X\ed X$ and hence $D[\iota_{a,A}]D[\tilde \pi]X\ed D[\iota_{a,A}]X \as X_a$. Let $\tau = \tilde\pi\circ\iota_{a,A}:a\rightarrow A$. It is $D[\iota_{a,A}]D[\tilde\pi]X = D[\tau]X$ and the latter equals $D[\hat\tau]X_{\tau(a)} = D[\pi](X_a)$ by definition, hence $X_a \ed D[\pi](X_a)$ and (i) follows. Now assume (i) and show (iv). By Proposition~\ref{prop:exchangeablelaws} there is $\mu\in\SYM(D)$ with $X_a\sim \mu_a$ for every finite set $a$. Let $\tau:A\rightarrow A$ be an arbitrary injection. It is 
			$$D[\tau](X) = \big(D[\widehat{\tau\circ\iota_{a,A}}]X_{\tau(a)}\big)_{a\in \finA}.$$
			Because laws on $D_A$ are determined by one-dimensional margins, see (1), only $D[\widehat{\tau\circ\iota_{a,A}}]X_{\tau(a)} \ed X_a$ needs to be shown. Now it is 
			$\widehat{\tau\circ\iota_{a,A}} = \tilde\tau$ an injection $a\rightarrow\tau(a)$, hence $D[\tilde\tau]X_{\tau(a)} \ed X_a$ follows from $\mu\in\SYM(D)$.\\
			(4)~Let $\tau:B\rightarrow A$ be injective between countable infinite sets and $\cL(X)\in \SYM(D;A)$. It is $D[\tau](X)$ a $D_B$-valued random variable. For every bijection $\pi:B\rightarrow B$ choose a bijection $\tilde \pi:A\rightarrow A$ with $\tilde\pi\circ\tau = \tau\circ\pi$. It holds $D[\pi]D[\tau]X = D[\tau\circ\pi](X) = D[\tilde\pi\circ\tau](X) = D[\tau]D[\tilde\pi]X \ed D[\tau]X$, that is $D[\tau]X$ is exchangeable and $\cL(X)\mapsto \cL(D[\tau]X)$ is a map $\SYM(D;A)\rightarrow \SYM(D;B)$. This map is an isomorphism due to Proposition~\ref{prop:exchangeablelaws}, which shows that both $\SYM(D;A)$ and $\SYM(D;B)$ can be identified with $\SYM(D)$ by the rule constructed there.
		\end{proof}
	
		\begin{proof}[Proof of Theorem~\ref{thm:correspondence-limits}]
			For all $x\in D_b, y\in D_a$ and bijections $\pi:a'\rightarrow a, \sigma:b'\rightarrow b$ it holds that 
			\begin{equation*}
			\density(x,y) = \density(D[\sigma](x), D[\pi](y)),
			\end{equation*}
			it is thus no restriction to consider only elements $x$ with $x\in D_{[k]}$ for some $k\geq 0$ when investigating limits. 
			Thus, only finite subsets $b, a\in\finN$ are considered and laws $\mu\in\eSYM(D)$ are identified with laws of ergodic exchangeable $D$-measurements $X = (\Xa)_{a\in\finN}$.\\
			
			Let $S = \cup_{k\geq 0}D_{[k]}$ and for $x\in S$ with $k=|x|$ let $1_{\{x\}}:D_{[k]}\rightarrow \{0,1\}$ be the indicator of $\{x\}\subseteq D_{[k]}$. Let $\cG = \{1_{\{x\}}|x\in S\}$. The law of any exchangeable $D$-measurement $X=(\Xa)_{a\in\finN}$ is determined by the expectations over $\cG$, that is by $\bE[1_{\{x\}}(X_{[k]})] = \bP[X_{[|x|]}=x], x\in S$.\\
			
			Applying (i)$\Rightarrow$(iii) of Theorem~\ref{thm:independence} to $\cG$ gives that for every ergodic $X\sim\mu$ there exists a convergent sequence $\bbx \subseteq S$ such that (\ref{eq:density}) holds.\\
			
			On the other hand it is easy to check that a limit of a convergent sequence $\bbs = (x_n)_n\subseteq S$ with $m_n=|x_n|\rightarrow\infty$ gives a rule $\mu\in\SYM(D)$
			via 
			$$\mu_a(M) = \sum_{x\in M\subseteq D_a}\lim_{n\rightarrow\infty}\density(x,x_n).$$
			This works because $D_a$ and hence $M\subseteq D_a$ are assumed to be finite.\\
			It only needs to be checked that $\mu$ is ergodic. Let $X\sim\mu$ using IDs $\bN$. Because the characterization of ergodicity via independence, Theorem~\ref{thm:independence}, check that for every $a,b\in\finN$ with $a\cap b=\emptyset$ and $x\in D_a, x'\in D_b$ it holds that $1(X_a=x), 1(X_b=x')$ are independent.\\
			
			For $a\cup b\subseteq [k]$, by sampling consistency, that probability for $\{X_{[k]}=x\}$ are represented by limits and that $D_{[k]}$ is finite one obtains:
			\begin{align*}
			\bP\big[X_a=x, X_b=x'\big] &= \sum_{y\in D_{[k]}}\bP[X_{[k]}=y]\bP[X_a=x, X_b=x'|X_{[k]}=y]\\
			&= \sum_{y\in D_{[k]}}1(D[\iota_{a,[k]}](y)=x, D[\iota_{b,[k]}](y)=x')\bP[X_{[k]}=y]\\
			&= \lim_{n\rightarrow\infty}\sum_{y\in D_{[k]}}1(D[\iota_{a,[k]}](y)=x, D[\iota_{b,[k]}](y)=x')\bP\Big[D[T_{k,m_n}](x_n) = y\Big]\\
			&= \lim_{n\rightarrow\infty}\bP\Big[D[T_{k,m_n}\circ \iota_{a,[k]}](x_n)=x, D[T_{k,m_n}\circ \iota_{b,[k]}](x_n)=x'\Big].
			\end{align*}
			The argument that the latter equals $\bP[X_a=x]\cdot \bP[X_b=x']$ is the same as in the proof of (iii)$\Rightarrow$(ii) from Theorem~\ref{thm:independence}.
		\end{proof}
	
		\begin{proof}[Proof of Proposition~\ref{prop:depth-restriction}]
			(i)~This is straightforward to check.\\
			
			For both (i) and (ii) some preparing observations. It is easy to check that $\tilde D$ defined by 
			$$\tilde D_a = \prod_{a'\in\binom{a}{\leq k}}D_{a'}$$
			and for $\tau:b\rightarrow a$ and $\tilde x = (x_{a'})_{a'\in\binom{a}{\leq k}} \in \tilde D_a$ 
			$$\tilde D[\tau](\tilde x) = \Big(D\big[\widetilde{\tau\circ\iota_{b',b}}\big](x_{\tau(b')})\Big)_{b'\in\binom{b}{\leq k}}$$
			defines a new Borel data structure. Again, it is straightforward to check that 
			$$\phi:D\rightarrow \tilde D~~~\phi_a(x) = \big(D[\iota_{a',a}](x)\big)_{a'\in\binom{a}{\leq k}}$$
			is a natural transformation such that every component $\phi_a$ is injective due to $\depth(D)=k$, that is $\phi:D\rightarrow\tilde D$ is an embedding.\\
			
			(ii) By Proposition~\ref{prop:embedding-structure} it is $\Dk = \phi D\subseteq \tilde D$ a Borel data structure naturally isomorphic to $D$. Let $\hat\phi:D\rightarrow \Dk$ be the natural isomorphism obtained from $\phi$ by restricting the range of its components and let $\hat\phi^{-1}:\Dk\rightarrow D$ be the natural inverse of $\hat\phi$.\\
			For every $a'\in\binom{a}{\leq k}$ it is $2^{a'} = \binom{a'}{\leq k}$ and hence 
			$$u\in \Rk_a \Longrightarrow u\circ\iota_{2^{a'},\binom{a}{\leq k}}\in R_{a'}.$$
			For every finite set $a$ and $u\in\Rk_a$ this allows to define
			\begin{equation}\label{eq:bareta}
			\bar\eta_a(u) = \big(\eta_{a'}\big(u\circ \iota_{2^{a'},\binom{a}{\leq k}}\big)\big)_{a'\in\binom{a}{\leq k}},
			\end{equation}
			Check that $\bar\eta_a(u)\in \Dk_a\subseteq \tilde D_a$: for every $u\in \Rk_a$ one can choose $v\in R_a$ with 
			$$u = r_a(v) = v\circ\iota_{\binom{a}{\leq k}, 2^a}.$$
			Note that for every $a'\in\binom{a}{\leq k}$ it holds that $2^{a'}\subseteq \binom{a}{\leq k} \subseteq 2^a$ and hence
			$$v\circ \iota_{2^{a'},2^a} = v\circ \iota_{\binom{a}{\leq k}, 2^a}\circ \iota_{2^{a'},\binom{a}{\leq k}} = u\circ \iota_{2^{a'},\binom{a}{\leq k}}.$$
			Applying naturality of $\eta$ gives
			\begin{align*}
			\hat\phi_a\circ\eta_a(v) &= (D[\iota_{a',a}]\circ\eta_a(v))_{a'\in\binom{a}{\leq k}}\\
			&= (\eta_{a'}\circ R[\iota_{a',a}](v))_{a'\in\binom{a}{\leq k}}\\
			&= (\eta_{a'}\big(v\circ \iota_{2^{a'},2^a}\big))_{a'\in\binom{a}{\leq k}}\\
			&= (\eta_{a'}\big(u\circ \iota_{2^{a'},\binom{a}{\leq k}}\big))_{a'\in\binom{a}{\leq k}}\\
			&= \bar\eta_a(u)\\
			&= \bar\eta_a\circ r_a(v).
			\end{align*}
			It is $\hat\phi_a\circ\eta_a(v)\in \Dk_a$ and hence $\bar\eta_a(u) = \hat\phi_a\circ\eta_a(v)\in \Dk_a$. That is, $\hat\phi$ is a measurable rule $\Rk\rightarrow\Dk$ and the previous calculation also showed that
			$$\hat\phi\circ\eta = \bar\eta\circ r.$$
			Applying $\hat\phi^{-1}$ to the left gives $\eta = \hat\phi^{-1}\circ\bar\eta\circ r$, so the candidate for $\tilde\eta$ is the rule $\tilde\eta = \hat\phi^{-1}\circ\bar\eta:\Rk\rightarrow D$.  All left to check is that this $\tilde\eta=\hat\phi^{-1}\circ\bar\eta$ is a natural transformation. Since $\hat\phi^{-1}$ is it suffices to show that $\bar\eta:\Rk\rightarrow \Dk$ is. Let $u\in \Rk_a$ and choose $v\in R_a$ with $u = r_a(v)$. Let $\tau:b\rightarrow a$ be injective.
			\begin{align*}
			\bar\eta_b\circ \Rk[\tau](u) &= \bar\eta_b\circ \Rk[\tau]\circ r_a(v) \\
			&= \bar\eta_b\circ r_b\circ R[\tau](v) \\
			&= \hat\phi_b\circ\eta_b\circ R[\tau](v) \\
			&= \Dk[\tau]\circ\hat\phi_a\circ\eta_a(v) \\
			&= \Dk[\tau]\circ\bar\eta_a\circ r_a(v) \\
			&= \Dk[\tau]\circ\bar\eta_a(u),
			\end{align*}
			that is $\bar\eta_b\circ \Rk[\tau] = \Dk[\tau]\circ\bar\eta_a$ as needed.\\
			
			(iii) The idea is the same as for (ii), but the technical details are a little more subtle. As before, let $\phi:D\rightarrow\tilde D$ be the embedding and let $\theta:\tilde D\rightarrow D$ be a left-inverse that is a $\mu\circ\phi^{-1}$-a.s. natural transformation for every $\mu\in\SYM(D)$, which exists due to Proposition~\ref{prop:inverse}. In particular, it holds that $\theta\circ\phi = \id_D$ and hence $\mu = \mu\circ(\theta\circ\phi)^{-1} = \mu\circ\phi^{-1}\circ\theta^{-1}$ for every $\mu\in\SYM(D)$.\\
			For every $u\in \Rk_a$ define the value $\bar\eta_a(u)\in \tilde D_a$ as in (\ref{eq:bareta}), which gives a rule $\bar\eta:\Rk\rightarrow\tilde D$.\\
			Let $V_a\sim\ur_a$ and define $U_a = r_a(V_a)$, so $U_a\sim\urk_a$. The $\ur$-a.s. naturality of $\eta$ gives 
			\begin{align*}
			\phi_a\circ\eta_a(V_a) &= (D[\iota_{a',a}]\circ\eta_a(V_a))_{a'\in\binom{a}{\leq k}}\\
			&\as (\eta_{a'}\circ R[\iota_{a',a}](V_a))_{a'\in\binom{a}{\leq k}}\\
			&= (\eta_{a'}\big(V_a\circ \iota_{2^{a'},2^a}\big))_{a'\in\binom{a}{\leq k}}\\
			&= (\eta_{a'}\big(V_a\circ \iota_{2^{a'},\binom{a}{\leq k}}\big))_{a'\in\binom{a}{\leq k}}\\
			&= \bar\eta_a(U_a)\\
			&= \bar\eta_a\circ r_a(V_a),
			\end{align*}
			that is the $\ur$-a.s. equality of the rules $\phi\circ\eta$ and $\bar\eta\circ r$. Applying $\theta$ on the left gives $\eta = \theta\circ\bar\eta\circ r$ $\ur$-almost surely. The desired candidate for $\tilde\eta:\Rk\rightarrow D$ is thus $\tilde\eta = \theta\circ\bar\eta$ and all left to check is that this is a $\urk$-a.s. natural transformation.\\
			First check that $\bar\eta:\Rk\rightarrow\tilde D$ is a $\urk$-a.s. natural transformation. Let $U_a = r_a(V_a)$ with $V_a\sim\ur_a$ and $\tau:b\rightarrow a$ be injective.
			\begin{align*}
			\bar\eta_b\circ \Rk[\tau](U_a) &= \bar\eta_b\circ \Rk[\tau]\circ r_a(V_a) \\
			&= \bar\eta_b\circ r_b\circ R[\tau](V_a) \\
			&\as \phi_b\circ\eta_b\circ R[\tau](V_a) \\
			&\as \phi_b\circ D[\tau]\circ\eta_a(V_a) \\
			&= \tilde D[\tau]\circ\phi_a\circ\eta_a(V_a) \\
			&\as \tilde D[\tau]\circ\bar\eta_a\circ r_a(V_a) \\
			&= \tilde D[\tau]\circ\bar\eta_a(U_a).
			\end{align*}
			So $\bar\eta:\Rk\rightarrow\tilde D$ is a $\urk$-a.s. natural transformation.\\
			Check that $\tilde\eta = \theta\circ\bar\eta:\Rk\rightarrow D$ is a $\urk$-a.s. natural transformation: it is $\theta$ a $\mu\circ\phi^{-1}$-a.s. natural transformation for every $\mu\in\SYM(D)$. Let $\mu = \ur\circ\eta^{-1}$. Because $\phi\circ\eta = \bar\eta\circ r$ $\ur$-almost surely and $\urk = \ur\circ r^{-1}$ it holds that
			$$\mu\circ\phi^{-1} = \ur\circ\eta^{-1}\circ\phi^{-1} = \ur\circ(\phi\circ\eta)^{-1} = \ur\circ (\bar\eta\circ r)^{-1} = \urk\circ\bar\eta^{-1}.$$
			Hence $\theta$ is $\urk\circ\bar\eta^{-1}$-a.s. natural transformation and $\bar\eta$ is a $\urk$-a.s. natural transformation. Lemma~\ref{lemma:chaining} gives that $\tilde\eta = \theta\circ\bar\eta$ is a $\urk$-a.s. natural transformation. 
		\end{proof}
	
		\begin{proof}[Proof of Lemma~\ref{lemma:ind}]
			(1)~For the moment write $\dom_b(\bbi) = \bigcap_{b'\subseteq b, \bbi\in I_{b'}}b'$. Let $c$ be another set with $\bbi\in I_c$. Check that $\dom_b(\bbi)=\dom_c(\bbi)$: since $\bbi\in I_b$ and $\bbi\in I_c$ it is $\bbi\in I_b\cap I_c = I_{b\cap c}$. Because $b\cap c\subseteq b$ and $b\cap c\subseteq c$ 
			$$\dom_b(\bbi) = \cap_{b'\subseteq b, \bbi\in I_{b'}} b' =  \cap_{d'\subseteq b\cap c, \bbi\in I_{d'}} d' = \cap_{c'\subseteq c, \bbi\in I_{c'}} c' = \dom_c(\bbi).$$
			(2)~Write $\tau = \iota_{\tau(b),a}\circ\hat\tau$ so that $I[\tau](\bbi) = I[\iota_{\tau(b),a}]\circ I[\hat\tau](\bbi)$. Because $I[\iota_{\tau(b),a}] = \iota_{I_{\tau(b)}, I_a}$ it is 
			$$I[\tau](\bbi) = \iota_{I_{\tau(b)}, I_a}(I[\hat\tau](\bbi)),$$
			that is $I[\hat\tau](\bbi) \in I_{\tau(b)} \subseteq I_a$ equals $I[\tau](\bbi)\in I_a$.\\
			(3)~Let $\bbi'=I[\tau](\bbi)$ and $c=\dom(\bbi)$ and $c'=\dom(\bbi')$. Check $c'=\tau(c)$: since $\bbi\in I_c$ it holds that $\bbi = \iota_{I_c,I_b}(\bbi)$, hence $I[\tau](\bbi) = I[\tau\circ\iota_{c,b}](\bbi)$ and so with $\pi = \widehat{\tau\circ\iota_{c,b}}$ by (2) $\bbi' = I[\pi](\bbi)$, which is element of $I_{\pi(c)} = I_{\tau(c)}$, thus $c'\subseteq \tau(c)$. Applying the inverse $\pi^{-1}$ to the equation $\bbi'=I[\pi](\bbi)$ gives $\bbi = I[\pi^{-1}](\bbi')$ and the same reasoning as before yields $c\subseteq \pi^{-1}(c')$. Because $\pi^{-1}$ is bijective it holds that $|c|\leq |c'|$. From $c'\subseteq \tau(c)$ and injectivity of $\tau$ it follows $|c'|\leq |c| = |\tau(c)|$ and hence $|c'|=|\tau(c)|$. Because both $c', \tau(c)$ are finite $c'\subseteq\tau(c)$ together with $|c'|=|\tau(c)|$ implies $c'=\tau(c)$.\\
			(4) Assume there is $\pi\in\stab(\bbi)$ with $\tau\circ \pi(i) = \sigma(i)$ for all $i\in\dom(\bbi)$. This implies $\widehat{\tau\circ\pi} = \hat \sigma$ and with (2) it follows 
			$$I[\sigma](\bbi) = I[\hat\sigma](\bbi) = I[\widehat{\tau\circ\pi}](\bbi) = I[\tau\circ\pi](\bbi) = I[\tau](I[\pi](\bbi)) = I[\tau](\bbi).$$
			Now assume $I[\sigma](\bbi) = I[\tau](\bbi)$. By (3) it is $\sigma(\dom(\bbi)) = \tau(\dom(\bbi))$, hence both $\hat\sigma, \hat\tau$ are bijections $\dom(\bbi) \rightarrow \tau(\dom(\bbi))$. By (2) it holds $I[\hat\tau](\bbi) = I[\tau](\bbi) = I[\sigma](\bbi) = I[\hat\sigma](\bbi)$. Applying $\hat\tau^{-1}$ on the left gives $\bbi = I[\hat\tau^{-1}\circ \hat\sigma](\bbi)$. That is, the bijection $\pi:=\hat\tau^{-1}\circ \hat\sigma$ is element of $\stab(\bbi)$. For $i\in \dom(\bbi)$ it is $\sigma(i) = \hat\sigma(i) \in \tau(\dom(\bbi))$ and hence $\tau\circ \pi(i) = \tau\circ\hat\tau^{-1}(\hat\sigma(i)) = \sigma(i)$.\\
			(5) Reflexivity: it is $I[\id_a](\bbi) = \bbi$ hence $\bbi\sim\bbi$. Symmetry: let $i'\in I_b$ with $\bbi\sim\bbi'$ witnessed by $\tau:b\rightarrow a$ satisfying $I[\tau](\bbi)=\bbi'$. By (2) it is $I[\tau](\bbi) = I[\hat\tau](\bbi) = \bbi'$. Applying $I[\hat\tau^{-1}]$ gives $\bbi = I[\hat\tau^{-1}](\bbi')$ and hence $\bbi'\sim\bbi$. Transitivity follows by composing the witnessing injections.
		\end{proof}

\printbibliography

\end{document}